\documentclass[11pt,a4paper,leqno]{article}
\usepackage{a4wide}
\usepackage{latexsym}
\usepackage{amsmath}
\usepackage{amssymb}
\usepackage{stackrel} 
\usepackage{color}
\usepackage{graphicx}
\newtheorem{defin}{Definition}
\newtheorem{lemma}{Lemma}
\newtheorem{prop}{Proposition}
\newtheorem{theo}{Theorem}

\newenvironment{proof}{\medskip\par\noindent{\bf Proof}}{\hfill $\Box$
\medskip\par}

\newcommand{\C}{\mathbb{C}}
\newcommand{\R}{\mathbb{R}}

\newcommand{\bgamma}{\boldsymbol{\gamma}}
\newcommand{\bnu}{\boldsymbol{\nu}}
\newcommand{\bk}{\boldsymbol{k}}
\newcommand{\bd}{\boldsymbol{d}}
\newcommand{\bT}{\boldsymbol{T}}
\newcommand{\bt}{\boldsymbol{t}}
\newcommand{\btau}{\boldsymbol{\tau}}

\begin{document}
\title{On parametric Gevrey asymptotics for some nonlinear initial value problems in two complex time variables}
\author{{\bf A. Lastra\footnote{The author is partially supported by the project MTM2016-77642-C2-1-P of Ministerio de Econom\'ia y Competitividad, Spain}, S. Malek\footnote{The author is partially supported by the project MTM2016-77642-C2-1-P of Ministerio de Econom\'ia y Competitividad, Spain.}}\\
University of Alcal\'{a}, Departamento de F\'{i}sica y Matem\'{a}ticas,\\
Ap. de Correos 20, E-28871 Alcal\'{a} de Henares (Madrid), Spain,\\
University of Lille 1, Laboratoire Paul Painlev\'e,\\
59655 Villeneuve d'Ascq cedex, France,\\
{\tt alberto.lastra@uah.es}\\
{\tt Stephane.Malek@math.univ-lille1.fr }}
\date{}
\maketitle
\thispagestyle{empty}
{ \small \begin{center}
{\bf Abstract}
\end{center}

The asymptotic behavior of a family of singularly perturbed PDEs in two time variables in the complex domain is studied. The appearance of a multilevel Gevrey asymptotics phenomenon in the perturbation parameter is observed. We construct a family of analytic sectorial solutions in $\epsilon$ which share a common asymptotic expansion at the origin, in different Gevrey levels. Such orders are produced by the action of the two independent time variables.

\medskip

\noindent Key words: asymptotic expansion, Borel-Laplace transform, Fourier transform, initial value problem,  formal power series,
nonlinear integro-differential equation, nonlinear partial differential equation, singular perturbation. 2010 MSC: 35C10, 35C20.}
\bigskip \bigskip

\section{Introduction}

This work is devoted to the study of a family of nonlinear initial value Cauchy problems of the form
\begin{multline}\label{probpral} 
Q_1(\partial_{z})Q_2(\partial_z)\partial_{t_1}\partial_{t_2}u(\bt,z,\epsilon) = (P_{1}(\partial_{z},\epsilon)u(\bt,z,\epsilon))(P_{2}(\partial_{z},\epsilon)u(\bt,z,\epsilon))\\
+ \sum_{0\le l_1\le D_1,0\le l_2\le D_2} \epsilon^{\Delta_{l_1,l_2}}t_1^{d_{l_1}}\partial_{t_1}^{\delta_{l_1}}t_2^{\tilde{d}_{l_2}}\partial_{t_2}^{\tilde{\delta}_{l_2}}R_{l_1,l_2}(\partial_{z})u(\bt,z,\epsilon)\\
+ c_{0}(\bt,z,\epsilon)R_{0}(\partial_{z})u(\bt,z,\epsilon) + f(\bt,z,\epsilon),
\end{multline}
with initial null data $u(0,t_2,z,\epsilon)\equiv u(t_1,0,z,\epsilon)\equiv 0$. Here, $D_1,D_2\ge2$ are integer numbers, and for every $0\le l_1\le D_1$ and $0\le l_2\le D_2$ we take non negative integers $d_{l_1},\delta_{l_1},\tilde{d}_{l_2},\tilde{\delta}_{l_2},\Delta_{l_1,l_2}$. The elements $Q_1,Q_2,R_0$ and $R_{l_1,l_2}$ for $0\le l_1\le D_1$ and $0\le l_2\le D_2$ turn out to be polynomials with complex coefficients, and $P_1,P_2$ are polynomials in their first variable, with coefficients being holomorphic and bounded functions in a neighborhood of the origin, say $D(0,\epsilon_0)$, for some $\epsilon_0>0$: $P_1,P_2\in (\mathcal{O}(\overline{D}(0,\epsilon_0))[X]$. 

The coefficient $c_0(t,z,\epsilon)$ and the forcing term $f(t,z,\epsilon)$ are holomorphic and bounded functions on $D(0,r)^2\times H_\beta\times D(0,\epsilon_0)$, where $r>0$, and $H_\beta$ stands for the horizontal strip
$$H_{\beta}:=\{z\in \C:|\hbox{Im}(z)|<\beta\},$$
for some $\beta>0$. 

The precise assumptions on the elements involved in the definition of the problem, and the construction of the coefficients and the forcing term are described in detail in Section 5. We mention two direct extensions that can be made to this study, and which do not offer any additional difficulty. On one hand, the existence of a quadratic nonlinearity can be extended to any higher order derivatives. On the other hand, the monomials in $t_1$ and $t_2$ appearing in the linear part of the right-hand side of the main equation can be substituted by any polynomial $p(t_1,t_2)\in\C[t_1,t_2]$. We have restricted our study to the family of equations in the form (\ref{probpral}) for the sake of brevity, aiming for a more comprehensive reading.\smallskip

The present work is a continuation of that in~\cite{lama}. In that work, the existence of a $k-$summable formal power series in the perturbation parameter $\epsilon$ is established, connecting the analytic solution of the problem with the formal one via Gevrey asymptotics. More precisely, the analytic solution is constructed in terms of a fixed point argument applied on a contractive map defined on appropriate Banach spaces. It admits the formal solution as its Gevrey asymptotic of certain order with respect to the perturbation parameter $\epsilon$, with coefficients belonging to certain Banach space of functions. The Gevrey order emerges from the coefficients and the forcing term appearing in the Cauchy problem. 

The importance of the present work with respect to that previous one is mainly due to the appearance of a multilevel Gevrey asymptotics phenomenon in the perturbation parameter, when dealing with a multivariable approach in time. On the way, different and more assorted situations appear. In addition to this, novel Banach spaces in zero, one and several time variables appearing in the reasoning are necessary in order to describe exponential growth with respect to certain monomial, or both time variables at the same time.

A recent overview on summability and multisummability techniques under different points of view is displayed in~\cite{loday}.\smallskip

In recent years, an increasing interest on complex singularly perturbed PDEs has been observed in the area. Parametric Borel summability has been described in semilinear systems of PDEs of Fuchsian type by H. Yamazawa and M. Yoshino in~\cite{yayo}
$$\eta\sum_{j=1}^n \lambda_j x_j\frac{\partial}{\partial x_j}u(x)=f(x,u),$$
where $x=(x_1,\ldots,x_n)\in\C^{n}$ and $f(x,u)=(f_1(x,u),\ldots,f_N(x,u))$, for $n,N\ge 1$, and $\lambda_j\in\C$. $\nu$ is a small complex perturbation parameter and $f$ stands for a holomorphic vector function in a neighborhood of the origin in $\C^n\times\C^{N}$. Also, in partial differential equations of irregular singular type by M. Yoshino in~\cite{yo}:
$$\eta\sum_{j=1}^n \lambda_j x_j^{s_j}\frac{\partial}{\partial x_j}u(x)=g(x,u,\eta),$$
where $s_j\ge 2$ for $1\le j\le m<n$, and $g(x,u,\eta)$ is a holomorphic vector function in some neighborhood of the origin in $\C^n\times \C^{N}\times \C$.

Recently, S.A. Carrillo and J. Mozo-Fern\'andez have studied properties on monomial summability and the extension of Borel-Laplace methods to this theory in~\cite{camo1,camo2}. In the last section of the second work, a further development on multisummability with respect to several monomials is proposed. Novel Gevrey asymptotic expansions and summability with respect to an analytic germ are described in~\cite{mosch18} and applied to different families of ODEs and PDEs such as
$$\left(x_2\frac{\partial P}{\partial x_2}+\alpha P^{k+1}+PA\right)x_1\frac{\partial f}{\partial x_1}-\left(x_1\frac{\partial P}{\partial x_1}+\beta P^{k+1}+PB\right)x_2\frac{\partial f}{\partial x_2}=h,$$
where $P$ is an homogeneous polynomial, $k\in\mathbb{N}^{\star}$ and $h,A,B$ are convergent power series, and $\alpha,\beta$ satisfy certain conditions.  

The present research joins monomial summability techniques to obtain multisummability of formal solutions of certain family of nonlinear PDEs.

The procedure in our study is as follows. The main problem under consideration, (\ref{probpral}) is specified in terms of an auxiliary problem through the change of variable $(T_1,T_2)=(\epsilon t_1,\epsilon t_2)$ and the properties of inverse Fourier transform (see Proposition~\ref{prop8}). This change of variable has been successfully applied previously by the authors in~\cite{ma1,lamasa1,lama} and rests on the work by M. Canalis-Durand, J. Mozo-Fern\'andez and R. Sch\" afke~\cite{camosch}. This auxiliary problem is given by (\ref{SCP}). We guarantee the existence of a formal power series 
$$\hat{U}(T_1,T_2,m,\epsilon)=\sum_{n_1,n_2\ge1}U_{n_1,n_2}(m,\epsilon)T_1^{n_1}T_2^{n_2},$$ which formally solves (\ref{SCP}) (see Proposition~\ref{prop10}). Its coefficients depend holomorphically on $\epsilon\in D(0,\epsilon_0)$, for some $\epsilon_0>0$, and belong to a Banach space of continuous functions with exponential decay on $\R$. The singular operators involved in the problem allow to expand them into certain irregular operators at $(T_1,T_2)=(0,0)$, following the approach in the work~\cite{taya}. A second auxiliary problem (\ref{k_Borel_equation}) is obtained by means of two consecutive Borel transforms of order $k_1$ with respect to $T_1$ and order $k_2$ with respect to $T_2$ in order to arrive at a convolution-like problem. 

At this point, we make use of a fixed point argument to an appropriate convolution operator (see Proposition~\ref{prop11}) in a Banach space of functions under certain growth and decay properties on their variables (see Definition 1). It is worth mentioning that several Banach spaces are involved in this result, due to the splitting of the sums in the linear operator which concerns functions which only depend on $\tau_1$ or $\tau_2$, or which depend on both or neither of the variables in time. This entails additional technical considerations on the way that different Banach spaces act on convolution operators (see Section 2). Let $\bk=(k_1,k_2)$. The fixed point theorem guarantees the existence of a solution of the second auxiliary problem (\ref{k_Borel_equation}), $\omega_{\bk}(\tau_1,\tau_2,m,\epsilon)$, continuous on $\overline{D}(0,\rho)^2\times \R\times\overline{D}(0,\epsilon_0)$ and holomorphic with respect to $(\tau_1,\tau_2,\epsilon)$ on $D(0,\rho)^2\times D(0,\epsilon)$, which can be extended to functions $\omega_{\bk}^{\mathfrak{d}_{p_1},\tilde{\mathfrak{d}}_{p_2}}(\tau_1,\tau_2,m,\epsilon)$ defined on a set $S_{\mathfrak{d}_{p_1}}\times S_{\mathfrak{d}_{p_2}}\times \R\times \mathcal{E}_{p_1,p_2}$, for every $0\le p_1\le \varsigma_1-1$ and $0\le p_2\le \varsigma_2-1$. Here, $S_{\mathfrak{d}_{p_1}}$ (resp. $S_{\tilde{\mathfrak{d}}_{p_2}}$) stands for an infinite sector with vertex at the origin and bisecting direction $\mathfrak{d}_{p_1}\in\R$ (resp. $\tilde{\mathfrak{d}}_{p_2}\in\R$), and $(\mathcal{E}_{p_1,p_2})_{\begin{subarray}{l} 0 \leq p_1 \leq \varsigma_1 - 1\\0 \leq p_2 \leq \varsigma_2 - 1\end{subarray}}$ is a good covering of $\C^{\star}$ of finite sectors with vertex at the origin (see Definition~\ref{defgood2}). In addition to this, the functions $\omega_{\bk}^{\mathfrak{d}_{p_1},\tilde{\mathfrak{d}}_{p_2}}$ satisfy that
$$\omega_{\bk}^{\mathfrak{d}_{p_1},\tilde{\mathfrak{d}}_{p_2}}(\tau_1,\tau_2,m,\epsilon)|\le \varpi^{\mathfrak{d}_{p_1},\tilde{\mathfrak{d}}_{p_2}}(1+|m|)^{-\mu}
\frac{|\frac{\tau_1}{\epsilon}|}{1 + |\frac{\tau_1}{\epsilon}|^{2k_1}}\frac{|\frac{\tau_2}{\epsilon}|}{1 + |\frac{\tau_2}{\epsilon}|^{2k_2}}\exp( -\beta|m| + \nu_1|\frac{\tau_1}{\epsilon}|^{k_1}-\nu_2|\frac{\tau_2}{\epsilon}|^{k_2} ),$$
for some $\varpi_{\mathfrak{d}_{p_1},\tilde{\mathfrak{d}}_{p_2}}>0$, and all $(\tau_1,\tau_2,m,\epsilon)\in S_{\mathfrak{d}_{p_1}}\times S_{\mathfrak{d}_{p_2}}\times \R\times \mathcal{E}_{p_1,p_2}$. Returning to our main problem, this entails that the function
$$u_{p_1,p_2}(\bt,z,\epsilon) = \frac{k_1k_2}{(2\pi)^{1/2}}\int_{-\infty}^{+\infty}
\int_{L_{\gamma_{p_1}}}\int_{L_{\gamma_{p_2}}}
\omega_{\bk}^{\mathfrak{d}_{p_1},\tilde{\mathfrak{d}}_{p_2}}(u_1,u_2,m,\epsilon) e^{-(\frac{u_1}{\epsilon t_1})^{k_1}-(\frac{u_2}{\epsilon t_2})^{k_2}} e^{izm} \frac{du_2}{u_2}\frac{du_1}{u_1} dm$$
defines a bounded and holomorphic function on $(\mathcal{T}_1\cap D(0,h''))\times (\mathcal{T}_2\cap D(0,h''))\times H_{\beta'}\times \mathcal{E}_{p_1,p_2}$, for some bounded sectors $\mathcal{T}_1,\mathcal{T}_2$ with vertex at 0, some $h''>0$, and where $H_{\beta'}$ is a horizontal strip. The properties of Laplace and inverse Fourier transform guarantee that $u_{p_1,p_2}$ is an actual solution of the main problem under study (\ref{probpral}).

The previous statement shapes the first part of Theorem~\ref{teo1}. The second part of that result proves that the difference of two consecutive solutions $u_{p_1,p_2}$ and $u_{p'_1,p'_2}$ (in the sense that $\mathcal{E}_{p_1,p_2}$ and $\mathcal{E}_{p'_1,p'_2}$ are consecutive sectors in the good covering, with nonempty intersection) can be classified into two categories: those pairs $((p_1,p_2),(p'_1,p'_2))\in\mathcal{U}_{k_1}$ such that 
$$|u_{p_1,p_2}(\bt,z,\epsilon)-u_{p'_1,p'_2}(\bt,z,\epsilon)|\le K_pe^{-\frac{M_p}{|\epsilon|^{k_1}}},\quad \epsilon\in \mathcal{E}_{p_1,p_2}\cap\mathcal{E}_{p'_1,p'_2},$$
uniformly for every $t_j\in\mathcal{T}_j\cap D(0,h'')$, for $j=1,2$ and all $z\in H_{\beta'}$; 

and those pairs $((p_1,p_2),(p'_1,p'_2))\in\mathcal{U}_{k_2}$ such that 
$$|u_{p_1,p_2}(\bt,z,\epsilon)-u_{p'_1,p'_2}(\bt,z,\epsilon)|\le K_pe^{-\frac{M_p}{|\epsilon|^{k_2}}},\quad \epsilon\in \mathcal{E}_{p_1,p_2}\cap\mathcal{E}_{p'_1,p'_2},$$
uniformly for every $t_j\in\mathcal{T}_j\cap D(0,h'')$, for $j=1,2$ and all $z\in H_{\beta'}$.

The second main result, Theorem~\ref{teo2}, makes use of a multilevel version of Ramis-Sibuya Theorem (see Theorem (RS)) and the exponential decay at zero with respect to the perturbation parameter observed in Theorem~\ref{teo1}, in order to state the existence of a formal power series $\hat{u}(t_1,t_2,z,\epsilon)$, written as a formal power series in $\epsilon$, with coefficients in the Banach space $\mathbb{F}$ of holomorphic and bounded functions on $(\mathcal{T}_1\cap D(0,h''))\times (\mathcal{T}_2\cap D(0,h''))\times H_{\beta'}$ with the supremum norm. This formal power series can be split as a sum of two formal power series in $\mathbb{F}[[\epsilon]]$, and each of the holomorphic solutions is decomposed accordingly in such a way that different Gevrey asymptotic behavior can be observed in each term of the sum. This phenomenon is the key point to multisummability, as described in~\cite{ba} and also Section 7.5 in~\cite{loday}.

The structure of the paper is as follows.

\noindent Section 2 analyzes the structure of Banach spaces involved in the construction of the solution of the main problem, and their behavior with respect to different convolution operators. A brief overview on Laplace and Fourier transforms and related properties is described in Section 3. Section 4 is devoted to the construction of two auxiliary problems and the existence of a fixed point of certain operator which is the source for the construction of analytic solutions to the main problem (\ref{probpral}), obtained in Theorem~\ref{teo1} of Section 5. Finally, in Theorem~\ref{teo2} of Section 6, we proof the existence of a formal solution, which is connected to the analytic solutions via a multilevel Gevrey asymptotic representation.

\section{Banach spaces functions with exponential growth and decay}

We denote by $D(0,r)$ an open disc centered at $0$ with radius $r>0$ in $\mathbb{C}$, and by $\bar{D}(0,r)$ its closure. Let $S_{d_j}$ be open unbounded sectors with bisecting directions $d_j \in \mathbb{R}$ for $j=1,2$, and $\mathcal{E}$ be an open sector with finite radius $r_{\mathcal{E}}$, all centered at $0$ in $\mathbb{C}$. For the sake of brevity, we denote $\btau=(\tau_1,\tau_2)$.

The definition of the following norm heavily rests on that considered in~\cite{lama}. Here, the exponential growth is held with respect to the two time variables which are involved.

\begin{defin} Let $\nu_1,\nu_2,\beta,\mu>0$ and $\rho>0$ be positive real numbers. Let $k_1,k_2 \geq 1$ be integer numbers and let $\epsilon \in \mathcal{E}$. We put $\bnu=(\nu_1,\nu_2)$, $\bk=(k_1,k_2)$, $\bd=(d_1,d_2)$, and denote
$F_{(\bnu,\beta,\mu,\bk,\epsilon)}^{\bd}$ the vector space of continuous functions $(\btau,m) \mapsto h(\btau,m)$ on the set
$(\bar{D}(0,\rho) \cup S_{d_1})\times (\bar{D}(0,\rho) \cup S_{d_2}) \times \mathbb{R}$, which are holomorphic with respect to $(\tau_1,\tau_2)$ on $(D(0,\rho) \cup S_{d_1})\times (D(0,\rho) \cup S_{d_2}) $ and such that
\begin{multline}
||h(\btau,m)||_{(\bnu,\beta,\mu,\bk,\epsilon)}\\
=
\sup_{\stackbin{\btau \in (\bar{D}(0,\rho) \cup S_{d_1})\times(\bar{D}(0,\rho) \cup S_{d_2})}{m \in \mathbb{R}}} (1+|m|)^{\mu}
\frac{1 + |\frac{\tau_1}{\epsilon}|^{2k_1}}{|\frac{\tau_1}{\epsilon}|}\frac{1 + |\frac{\tau_2}{\epsilon}|^{2k_2}}{|\frac{\tau_2}{\epsilon}|}\exp( \beta|m| - \nu_1|\frac{\tau_1}{\epsilon}|^{k_1}-\nu_2|\frac{\tau_2}{\epsilon}|^{k_2} ) |h(\tau,m)|
\end{multline}
is finite. One can check that the normed space
$(F_{(\bnu,\beta,\mu,\bk,\epsilon)}^{\bd},||.||_{(\bnu,\beta,\mu,\bk,\epsilon)})$ is a Banach space.
\end{defin}

Throughout the whole section, we assume $\epsilon \in \mathcal{E}$, $\mu,\beta,>0$ are fixed numbers. We also fix $\bnu=(\nu_1,\nu_2)$ for some positive numbers $\nu_1,\nu_2$, and $\bk=(k_1,k_2)$ is a couple of positive integer numbers. Additionally, we take $d_1,d_2\in\R$, and write $\bd$ for $(d_1,d_2)$. The next results are stated without proofs, which are analogous to those in Section 2 of~\cite{lama}. The integrals appearing in these results can be split accordingly, in order to apply the proof therein.

\begin{lemma}\label{lema1} Let $(\btau,m) \mapsto a(\btau,m)$ be a bounded continuous function on
$(\bar{D}(0,\rho) \cup S_{d_1})\times (\bar{D}(0,\rho) \cup S_{d_2}) \times \mathbb{R}$, holomorphic with respect to $\btau$ on $(D(0,\rho) \cup S_{d_1})\times (D(0,\rho) \cup S_{d_2})$. Then,
\begin{equation}
|| a(\btau,m) h(\btau,m) ||_{(\bnu,\beta,\mu,\bk,\epsilon)} \leq
\left( \sup_{\btau \in (\bar{D}(0,\rho) \cup S_{d_1})\times (\bar{D}(0,\rho) \cup S_{d_2}),m \in \mathbb{R}} |a(\btau,m)| \right)
||h(\btau,m)||_{(\bnu,\beta,\mu,\bk,\epsilon)}
\end{equation}
for all $h(\btau,m) \in F_{(\bnu,\beta,\mu,\bk,\epsilon)}^{\bd}$.
\end{lemma}

\begin{prop}\label{prop1} Let $\gamma_{21},\gamma_{22}>0$ be real numbers. Assume that $k_1,k_2 \geq 1$ are such that $1/k_j \leq \gamma_{2j} \leq 1$, for $j=1,2$. Then, a constant $C_{1}>0$ (depending on $\bnu,\bk,\gamma_{21},\gamma_{22}$) exists with
\begin{multline}
||\int_0^{\tau_1^{k_1}}\int_{0}^{\tau_2^{k_2}} (\tau_1^{k_1}-s_1)^{\gamma_{21}} (\tau_2^{k_2}-s_2)^{\gamma_{22}} f(s_1^{1/k_1},s_2^{1/k_2},m) \frac{ds_2}{s_2}\frac{ds_1}{s_1}
||_{(\bnu,\beta,\mu,\bk,\epsilon)} \\
\leq C_{1}|\epsilon|^{k_1 \gamma_{21}+k_2\gamma_{22}} ||f(\tau,m)||_{(\bnu,\beta,\mu,\bk,\epsilon)}
\end{multline}
for all $f(\btau,m) \in F_{(\bnu,\beta,\mu,\bk,\epsilon)}^{\bd}$.
\end{prop}

Proposition 2 in~\cite{lama} is adapted to the Banach space considered in this work.

\begin{prop}\label{prop2} Let $\gamma_{11},\gamma_{12} \geq 0$ and $\chi_{21},\chi_{22}>-1$ be real numbers. Let $\xi_{21},\xi_{22} \geq 0$ be integer numbers. We write $\bgamma_{1}=(\gamma_{11},\gamma_{12})$. We consider $a_{\bgamma,\bk}\in\mathcal{O}((D(0,\rho)\cup S_{d_1})\times (D(0,\rho)\cup S_{d_2})$, continuous on $(\bar{D}(0,\rho) \cup S_{d_1})\times (\bar{D}(0,\rho) \cup S_{d_2})$, such that
$$ |a_{\bgamma_{1},\bk}(\btau)| \leq \frac{1}{(1+|\tau_1|^{k_1})^{\gamma_{11}}(1+|\tau_2|^{k_2})^{\gamma_{12}}},\quad \btau \in (\bar{D}(0,\rho) \cup S_{d_1})\times (\bar{D}(0,\rho) \cup S_{d_2}).$$\medskip

Assume that for $j=1$ and $j=2$, one of the following holds
\begin{itemize}
\item $\chi_{2j} \geq 0$ and $\xi_{2j}+\chi_{2j}-\gamma_{1j}\le0$, or
\item $\chi_{2j}=\frac{\tilde{\chi}_j}{k_j}-1$, for some $\tilde{\chi}_j\ge 1$ and $\xi_{2j}+\frac{1}{k_j}-\gamma_{1j}\le 0$.
\end{itemize}
Then, there exists a constant $C_{2}>0$ (depending, eventually, on
$\bnu,\xi_{21},\xi_{22},\chi_{21},\chi_{22},\bgamma_{1},\tilde{\chi}_1,\tilde{\chi}_2,\bk$) such that
\begin{multline}
|| a_{\bgamma_{1},\bk}(\btau) \int_{0}^{\tau_1^{k_1}}\int_{0}^{\tau_2^{k_2}} (\tau_1^{k_1}-s_1)^{\chi_{21}}(\tau_2^{k_2}-s_2)^{\chi_{22}}s_1^{\xi_{21}}s_2^{\xi_{22}}f(s_1^{1/k_1},s_2^{1/k_2},m) ds_2ds_1 ||_{(\bnu,\beta,\mu,\bk,\epsilon)} \\
\leq C_{2}|\epsilon|^{k_1(1+\xi_{21}+\chi_{21}-\gamma_{11})+k_2(1+\xi_{22}+\chi_{22}-\gamma_{12})}
||f(\tau,m)||_{(\bnu,\beta,\mu,\bk,\epsilon)}
\end{multline}
for all $f(\btau,m) \in F_{(\bnu,\beta,\mu,\bk,\epsilon)}^{\bd}$.
\end{prop}

The previous result can also be particularized to each of the variables in time in the following manner. We write the result which corresponds to the first time variable, but one can reproduce the same arguments symmetrically on the second variable in time, $\tau_2$.

\begin{prop}\label{prop2lama} Let $\gamma_{1} \geq 0$ and $\chi_{2}>-1$ be real numbers, and $\xi_2\ge0$ be an integer number. We consider $a_{\gamma_1,k_1}\in\mathcal{O}(D(0,\rho)\cup S_{d_1})$, continuous on $\bar{D}(0,\rho) \cup S_{d_1}$, such that
$$ |a_{\gamma_{1},k_1}(\tau_1)| \leq \frac{1}{(1+|\tau_1|^{k_1})^{\gamma_{1}}},\quad \tau_1 \in \bar{D}(0,\rho) \cup S_{d_1}.$$\medskip
Assume that 
\begin{itemize}
\item $\chi_{2} \geq 0$ and $\xi_{2}+\chi_{2}-\gamma_{1}\le0$, or
\item $\chi_{2}=\frac{\tilde{\chi}}{k_1}-1$, for some $\tilde{\chi}\ge 1$ and $\xi_{2}+\frac{1}{k_1}-\gamma_{1}\le 0$.
\end{itemize}
Then, there exists a constant $C_{2}>0$ (depending, eventually, on
$\bnu,\xi_{2},\chi_{2},\gamma_{1},\tilde{\chi},k_1$) such that
\begin{multline}
|| a_{\gamma_{1},k_1}(\tau_1) \int_{0}^{\tau_1^{k_1}} (\tau_1^{k_1}-s_1)^{\chi_{2}}s_1^{\xi_{2}}f(s_1^{1/k_1},\tau_2,m) ds_1 ||_{(\bnu,\beta,\mu,\bk,\epsilon)} \\
\leq C_{2}|\epsilon|^{k_1(1+\xi_{2}+\chi_{2}-\gamma_{1})}
||f(\btau,m)||_{(\bnu,\beta,\mu,\bk,\epsilon)}
\end{multline}
for all $f(\btau,m) \in F_{(\bnu,\beta,\mu,\bk,\epsilon)}^{\bd}$.
\end{prop}

\begin{prop}\label{prop137} Let $Q_{1}(X),Q_{2}(X),R(X) \in \mathbb{C}[X]$ such that
\begin{equation}
\mathrm{deg}(R) \geq \mathrm{deg}(Q_{1}) \ \ , \ \ \mathrm{deg}(R) \geq \mathrm{deg}(Q_{2}) \ \ , \ \ R(im) \neq 0,\qquad m\in\R.
\label{R>Q1_R>Q2_R_nonzero}
\end{equation}
Assume that $\mu > \max( \mathrm{deg}(Q_{1})+1, \mathrm{deg}(Q_{2})+1 )$. Let
$m \mapsto b(m)$ be a continuous function on $\mathbb{R}$ such that
$$
|b(m)| \leq \frac{1}{|R(im)|},\quad m \in \mathbb{R}.
$$
Then, there exists a constant $C_{3}>0$ (depending on $Q_{1},Q_{2},R,\mu,\bk,\bnu$) such that
\begin{multline}
|| b(m) \int_{0}^{\tau_1^{k_1}}\int_{0}^{\tau_2^{k_2}} (\tau_1^{k_1}-s_1)^{\frac{1}{k_1}}(\tau_2^{k_2}-s_2)^{\frac{1}{k_2}}\\
\times \left( \int_{0}^{s_1}\int_{0}^{s_2} \int_{-\infty}^{+\infty} Q_{1}(i(m-m_{1}))
f((s_1-x_1)^{1/{k_1}},(s_2-x_2)^{1/{k_2}},m-m_{1})\right. \\
\left.\times Q_{2}(im_{1}) g(x_1^{1/{k_1}},x_2^{1/{k_2}},m_{1}) \frac{1}{(s_1-x_1)x_1}\frac{1}{(s_2-x_2)x_2} dm_{1} dx_2 dx_1  \right) ds_2 ds_1 ||_{(\bnu,\beta,\mu,\bk,\epsilon)} \\
\leq C_{3}|\epsilon|^2 ||f(\btau,m)||_{(\bnu,\beta,\mu,\bk,\epsilon)} ||g(\btau,m)||_{(\bnu,\beta,\mu,\bk,\epsilon)}
\end{multline}
for all $f(\btau,m), g(\btau,m) \in F_{(\bnu,\beta,\mu,\bk,\epsilon)}^{\bd}$.
\end{prop}

\begin{defin} Let $\epsilon \in \mathcal{E}$. We denote
$F_{(\nu_1,\beta,\mu,k_1,\epsilon)}^{d_1}$ the vector space of continuous functions $(\tau_1,m) \mapsto h(\tau_1,m)$ defined on the set
$(\bar{D}(0,\rho) \cup S_{d_1}) \times \mathbb{R}$, which are holomorphic with respect to $\tau_1$ on $D(0,\rho) \cup S_{d_1}$ and such that
\begin{multline}
||h(\tau_1,m)||_{(\nu_1,\beta,\mu,k_1,\epsilon)}\\
=
\sup_{\tau_1 \in (\bar{D}(0,\rho) \cup S_{d_1}),m \in \mathbb{R}} (1+|m|)^{\mu}
\frac{1 + |\frac{\tau_1}{\epsilon}|^{2k_1}}{|\frac{\tau_1}{\epsilon}|}\exp( \beta|m| - \nu_1|\frac{\tau_1}{\epsilon}|^{k_1}) |h(\tau_1,m)|
\end{multline}
is finite. The normed space
$(F_{(\nu_1,\beta,\mu,k_1,\epsilon)}^{d_1},||.||_{(\nu_1,\beta,\mu,k_1,\epsilon)})$ is a Banach space.
\end{defin}



The enunciate of Proposition~\ref{prop137} can be adapted in the following form to the Banach spaces involved, as follows.

\begin{prop}\label{prop5} Let $Q_{1}(X),Q_{2}(X),R(X) \in \mathbb{C}[X]$ such that
\begin{equation}
\mathrm{deg}(R) \geq \mathrm{deg}(Q_{1}) \ \ , \ \ \mathrm{deg}(R) \geq \mathrm{deg}(Q_{2}) \ \ , \ \ R(im) \neq 0,\qquad m\in\R.
\label{R>Q1_R>Q2_R_nonzero2}
\end{equation}
Assume that $\mu > \max( \mathrm{deg}(Q_{1})+1, \mathrm{deg}(Q_{2})+1 )$. Let
$m \mapsto b(m)$ be a continuous function on $\mathbb{R}$ such that
$$
|b(m)| \leq \frac{1}{|R(im)|},\quad m \in \mathbb{R}.
$$
Then, there exists a constant $C_{3.2}>0$ (depending on $Q_{1},Q_{2},R,\mu,\bk,\bnu$) such that
\begin{multline}
|| b(m) \int_{0}^{\tau_1^{k_1}}\int_{0}^{\tau_2^{k_2}} (\tau_1^{k_1}-s_1)^{\frac{1}{k_1}}(\tau_2^{k_2}-s_2)^{\frac{1}{k_2}}\\
\times \left( \int_{0}^{s_1}\int_{-\infty}^{+\infty} Q_{1}(i(m-m_{1}))
f((s_1-x_1)^{1/{k_1}},m-m_{1})\right. \\
\left.\times Q_{2}(im_{1}) g(x_1^{1/{k_1}},s_2^{1/{k_2}},m_{1}) \frac{1}{(s_1-x_1)x_1}  dm_{1} dx_1  \right)\frac{ds_2}{s_2} ds_1 ||_{(\bnu,\beta,\mu,\bk,\epsilon)} \\
\leq C_{3.2}|\epsilon|^2 ||f(\tau_1,m)||_{(\nu_1,\beta,\mu,k_1,\epsilon)} ||g(\btau,m)||_{(\bnu,\beta,\mu,\bk,\epsilon)}
\end{multline}
for all $f(\tau_1,m)\in F^{d_1}_{(\nu_1,\beta,\mu,k_1,\epsilon)}  $ and $g(\btau,m) \in F_{(\bnu,\beta,\mu,\bk,\epsilon)}^{\bd}$.
\end{prop}
\begin{proof}
A first stage of the proof is analogous to that of Proposition 3 in~\cite{lama}, concerning the operators involving the first variable in time, $\tau_1$. One leads to  
\begin{multline*}
|| b(m) a_{\gamma_2,k_2}(\tau_2)\int_{0}^{\tau_1^{k_1}}\int_{0}^{\tau_2^{k_2}} (\tau_1^{k_1}-s_1)^{\frac{1}{k_1}}(\tau_2^{k_2}-s_2)^{\frac{1}{k_2}}\left( \int_{0}^{s_1}\int_{-\infty}^{+\infty} Q_{1}(i(m-m_{1}))\right.\\
\left.f((s_1-x_1)^{1/{k_1}},m-m_{1})\times Q_{2}(im_{1}) g(x_1^{1/{k_1}},s_2^{1/{k_2}},m_{1}) \frac{1}{(s_1-x_1)x_1}  dx_1 dm_{1} \right)\frac{ds_2}{s_2} ds_1 ||_{(\bnu,\beta,\mu,\bk,\epsilon)}\\
\le\sup_{\tau_2\in \overline{D}(0,\rho)\cup S_{d_2}}\tilde{C}_{3.2}|\epsilon|\frac{1+\left|\frac{\tau_2}{\epsilon}\right|^{2k_2}}{\left|\frac{\tau_2}{\epsilon}\right|}\exp(-\nu_2\left|\frac{\tau_2}{\epsilon}\right|^{k_2})\\
\int_0^{|\tau_2|^{k_2}}(|\tau_2|^{k_2}-|s_2|)^{1/k_2}\frac{|s_2|^{1/k_2}/|\epsilon|}{1+|s_2|^2/|\epsilon|^{2k_2}}\exp(\nu_2\frac{|s_2|}{|\epsilon|^{k_2}})\frac{d|s_2|}{|s_2|}\left\|f\right\|_{(\nu_1,\beta,\mu,k_1,\epsilon)}\left\|g\right\|_{(\bnu,\beta,\mu,\bk,\epsilon)},
\end{multline*}
for some $\tilde{C}_{3.2}>0$. The change of variable $|s_2|=h|\epsilon|^{k_2}$ in the integral above allow us guarantee that the previous expression is upper bounded by
\begin{multline*}
\tilde{C}_{3.2}|\epsilon|^2\sup_{x\ge 0}\frac{1+x^2}{x^{1/k_2}}\exp(-\nu_2x)\int_0^x(x-h)^{1/k_2}\frac{h^{1/k_2}}{1+h^2}\exp(\nu_2 h)\frac{dh}{h}\left\|f\right\|_{(\nu_1,\beta,\mu,k_1,\epsilon)}\left\|g\right\|_{(\bnu,\beta,\mu,\bk,\epsilon)}\\
=\tilde{C}_{3.2}|\epsilon|^2\sup_{x\ge0}A(x)\left\|f\right\|_{(\nu_1,\beta,\mu,k_1,\epsilon)}\left\|g\right\|_{(\bnu,\beta,\mu,\bk,\epsilon)}.
\end{multline*}
It is straight to check that $A(x)$ is bounded for all $x\ge0$, and the result is attained.
\end{proof}

The previous result is also valid by interchanging the role of the variables $\tau_1$ and $\tau_2$. The following definition deals with a Banach space considered in~\cite{lama}. We provide inner operaions linking this and the previous Banach spaces.

\begin{defin} Let $\beta, \mu \in \mathbb{R}$. We denote by
$E_{(\beta,\mu)}$ the vector space of continuous functions $h : \mathbb{R} \rightarrow \mathbb{C}$ such that
$$ ||h(m)||_{(\beta,\mu)} = \sup_{m \in \mathbb{R}} (1+|m|)^{\mu} \exp( \beta |m|) |h(m)| $$
is finite. The space $E_{(\beta,\mu)}$ equipped with the norm $||.||_{(\beta,\mu)}$ is a Banach space.
\end{defin}

\begin{prop}\label{prop6} Let $Q(X),R(X) \in \mathbb{C}[X]$ such that
\begin{equation}
\mathrm{deg}(R) \geq \mathrm{deg}(Q) \ \ , \ \ R(im) \neq 0 \label{cond_R_Q}
\end{equation}
for all $m \in \mathbb{R}$. Assume that $\mu > \mathrm{deg}(Q) + 1$. Let $m \mapsto b(m)$ be a continuous function such that
$$
|b(m)| \leq \frac{1}{|R(im)|},\quad m \in \mathbb{R}.
$$
Then, there exists a constant $C_{4}>0$ (depending on $Q,R,\mu,\bk,\nu$) such that
\begin{multline}
|| b(m) \int_{0}^{\tau_1^{k_1}} \int_{0}^{\tau_2^{k_2}} (\tau_1^{k_1}-s_1)^{\frac{1}{k_1}} (\tau_2^{k_2}-s_2)^{\frac{1}{k_2}}\int_{-\infty}^{+\infty} f(m-m_{1})Q(im_{1})g(s_1^{1/k_1},s_2^{1/k_2},m_{1})dm_{1}
\frac{ds_2}{s_2}\frac{ds_1}{s_1}||_{(\bnu,\beta,\mu,\bk,\epsilon)}\\
\leq C_{4}|\epsilon|^2 ||f(m)||_{(\beta,\mu)} ||g(\btau,m)||_{(\bnu,\beta,\mu,\bk,\epsilon)}
\label{norm_conv_f_g<norm_f_beta_mu_times_norm_g}
\end{multline}
for all $f(m) \in E_{(\beta,\mu)}$, all $g(\btau,m) \in F_{(\bnu,\beta,\mu,\bk,\epsilon)}^{\bd}$.
\end{prop}

\begin{prop}\label{prop7} Let $Q_{1}(X),Q_{2}(X),R(X) \in \mathbb{C}[X]$ such that
\begin{equation}
\mathrm{deg}(R) \geq \mathrm{deg}(Q_{1}) \ \ , \ \ \mathrm{deg}(R) \geq \mathrm{deg}(Q_{2}) \ \ , \ \ R(im) \neq 0,\quad m\in\mathbb{R}.
\label{cond_R_Q1_Q2}
\end{equation}
Assume that $\mu > \max( \mathrm{deg}(Q_{1})+1, \mathrm{deg}(Q_{2})+1 )$. Then, there exists a
constant $C_{5}>0$ (depending on $Q_{1},Q_{2},R,\mu$) such that
\begin{multline}
|| \frac{1}{R(im)} \int_{-\infty}^{+\infty} Q_{1}(i(m-m_{1})) f(m-m_{1}) Q_{2}(im_{1})g(m_{1}) dm_{1} ||_{(\beta,\mu)}\\
\leq C_{5} ||f(m)||_{(\beta,\mu)}||g(m)||_{(\beta,\mu)}
\end{multline}
for all $f(m),g(m) \in E_{(\beta,\mu)}$. Therefore, $(E_{(\beta,\mu)},||.||_{(\beta,\mu)})$ becomes a Banach algebra for the product
$\star$ defined by
$$ f \star g (m) = \frac{1}{R(im)} \int_{-\infty}^{+\infty} Q_{1}(i(m-m_{1})) f(m-m_{1}) Q_{2}(im_{1})g(m_{1}) dm_{1}.$$
As a particular case, when $f,g \in E_{(\beta,\mu)}$ with $\beta>0$ and $\mu>1$, the classical convolution product
$$ f \ast g (m) = \int_{-\infty}^{+\infty} f(m-m_{1})g(m_{1}) dm_{1} $$
belongs to $E_{(\beta,\mu)}$.
\end{prop}

\section{Laplace transform, asymptotic expansions and Fourier transform}

We recall the definition of $k-$Borel summable formal power series with coefficients in a fixed Banach space $( \mathbb{E}, ||.||_{\mathbb{E}} )$. This tool has been adapted from the classical version in \cite{ba}, Section 3.2.

\begin{defin} Let $k \geq 1$ be an integer. Let $m_{k}(n)$ be the sequence defined by
$$ m_{k}(n) = \Gamma(\frac{n}{k}) =  \int_{0}^{+\infty} t^{\frac{n}{k}-1} e^{-t} dt,\quad n\ge1. $$
A formal power series $\hat{X}(T) = \sum_{n=1}^{\infty}  a_{n}T^{n} \in T\mathbb{E}[[T]]$ is $m_{k}-$summable with respect to $t$ in the direction $d \in [0,2\pi)$ if \medskip

{\bf i)} there exists $\rho \in \mathbb{R}_{+}$ such that the following formal series, called a formal $m_{k}-$Borel transform of
$\hat{X}$ 
$$ \mathcal{B}_{m_k}(\hat{X})(\tau) = \sum_{n=1}^{\infty} \frac{ a_{n} }{ \Gamma(\frac{n}{k}) } \tau^{n}
\in \tau\mathbb{E}[[\tau]],$$
is absolutely convergent for $|\tau| < \rho$. \medskip

{\bf ii)} there exists $\delta > 0$ such that the series $\mathcal{B}_{m_k}(\hat{X})(\tau)$ can be analytically continued with
respect to $\tau$ in a sector
$S_{d,\delta} = \{ \tau \in \mathbb{C}^{\ast} : |d - \mathrm{arg}(\tau) | < \delta \} $. Moreover, there exist $C >0$ and $K >0$
such that
$$ ||\mathcal{B}_{m_k}(\hat{X})(\tau)||_{\mathbb{E}}
\leq C e^{ K|\tau|^{k} },\quad \tau \in S_{d, \delta}.$$
\end{defin}

Under the previous hypotheses, the vector valued Laplace transform of $\mathcal{B}_{m_k}(\hat{X})(\tau)$ in the direction $d$ is defined by
$$ \mathcal{L}^{d}_{m_k}(\mathcal{B}(\hat{X}))(T) = k \int_{L_{\gamma}}
\mathcal{B}_{m_k}(\hat{X})(u) e^{ - ( u/T )^{k} } \frac{d u}{u},$$
along a half-line $L_{\gamma} = \mathbb{R}_{+}e^{i\gamma} \subset S_{d,\delta} \cup \{ 0 \}$, where $\gamma$ depends on
$T$ and is chosen in such a way that $\cos(k(\gamma - \mathrm{arg}(T))) \geq \delta_{1} > 0$, for some fixed $\delta_{1}$.
The function $\mathcal{L}^{d}_{m_k}(\mathcal{B}_{m_k}(\hat{X}))(T)$ is well defined, holomorphic and bounded in any sector
$$ S_{d,\theta,R^{1/k}} = \{ T \in \mathbb{C}^{\ast} : |T| < R^{1/k} \ \ , \ \ |d - \mathrm{arg}(T) | < \theta/2 \},$$
where $\frac{\pi}{k} < \theta < \frac{\pi}{k} + 2\delta$ and
$0 < R < \delta_{1}/K$. This function is called the $m_{k}-$sum of the formal series $\hat{X}(T)$ in the direction $d$.\medskip

\noindent Some elementary properties regarding $m_{k}-$sums of formal power series are the following:\\

\noindent 1)  $\mathcal{L}^{d}_{m_k}(\mathcal{B}_{m_k}(\hat{X}))(T)$ admits $\hat{X}(T)$ as its
Gevrey asymptotic expansion of order $1/k$ with respect to $t$ on $S_{d,\theta,R^{1/k}}$. More precisely, for every $\frac{\pi}{k} < \theta_{1} < \theta$, there exist $C,M > 0$
such that
\begin{equation}
 ||\mathcal{L}^{d}_{m_k}(\mathcal{B}_{m_k}(\hat{X}))(T) - \sum_{p=1}^{n-1} a_p T^{p}||_{\mathbb{E}} \leq
CM^{n}\Gamma(1+\frac{n}{k})|T|^{n}, \quad n\ge2,\quad T \in S_{d,\theta_{1},R^{1/k}}\label{Laplace_k_Gevrey_ae}
\end{equation}
Unicity of such function on sectors $S_{d,\theta_{1},R^{1/k}}$ with opening $\theta_{1} > \frac{\pi}{k}$ is guaranteed by Watson's lemma (see Proposition 11 p. 75, \cite{ba}).\medskip

\noindent 2) Let us assume that $( \mathbb{E}, ||.||_{\mathbb{E}} )$ also has the structure of a Banach algebra for a product $\star$.
Let $\hat{X}_{1}(T),\hat{X}_{2}(T) \in T\mathbb{E}[[T]]$ be $m_{k}-$summable formal power series in direction
$d$. Let $q_{1} \geq q_{2} \geq 1$ be integers. We assume that 
$\hat{X}_{1}(T)+\hat{X}_{2}(T)$, $\hat{X}_{1}(T) \star \hat{X}_{2}(T)$ and
$T^{q_1}\partial_{T}^{q_2}\hat{X}_{1}(T)$, which are elements of $T\mathbb{E}[[T]]$, are $m_{k}-$summable in direction $d$.
Then, the following equalities
\begin{multline}
\mathcal{L}^{d}_{m_k}(\mathcal{B}_{m_k}(\hat{X}_{1}))(T) +
\mathcal{L}^{d}_{m_k}(\mathcal{B}_{m_k}(\hat{X}_{2}))(T) =
\mathcal{L}^{d}_{m_k}(\mathcal{B}_{m_k}(\hat{X}_{1} + \hat{X}_{2}))(T),\\
\mathcal{L}^{d}_{m_k}(\mathcal{B}_{m_k}(\hat{X}_{1}))(T) \star
\mathcal{L}^{d}_{m_k}(\mathcal{B}_{m_k}(\hat{X}_{2}))(T) =
\mathcal{L}^{d}_{m_k}(\mathcal{B}_{m_k}(\hat{X}_{1} \star \hat{X}_{2}))(T)\\
T^{q_1}\partial_{T}^{q_2}\mathcal{L}^{d}_{m_k}(\mathcal{B}_{m_k}(\hat{X}_{1}))(T) =
\mathcal{L}^{d}_{m_k}(\mathcal{B}_{m_k}(T^{q_1}\partial_{T}^{q_2}\hat{X}_{1}))(T) \label{sum_prod_deriv_m_k_sum}
\end{multline}
hold for all $T \in S_{d,\theta,R^{1/k}}$.\medskip

The next result recalls some properties on the $m_{k}-$Borel transform, successfully used in~\cite{lama,lama1}
\begin{prop}\label{prop8} Let $(\mathbb{E},||.||_{\mathbb{E}})$ be a Banach algebra for some product $\star$. Let $\hat{f}(t) = \sum_{ n \geq 1} f_{n}t^{n}\in\mathbb{E}[[t]]$, $\hat{g}(t) = \sum_{n \geq 1} g_{n}t^{n}\in\mathbb{E}[[t]]$. Let $k,m \geq 1$ be integers. Then, the following formal identities hold.
\begin{equation}
\mathcal{B}_{m_k}(t^{k+1}\partial_{t}\hat{f}(t))(\tau) = k \tau^{k} \mathcal{B}_{m_k}(\hat{f}(t))(\tau) \label{Borel_diff}
\end{equation}
\begin{equation}
\mathcal{B}_{m_k}(t^{m}\hat{f}(t))(\tau) = \frac{\tau^{k}}{\Gamma(\frac{m}{k})}
\int_{0}^{\tau^{k}} (\tau^{k} - s)^{\frac{m}{k}-1} \mathcal{B}_{m_k}(\hat{f}(t))(s^{1/k}) \frac{ds}{s} \label{Borel_mult_monom}
\end{equation}
and
\begin{equation}
\mathcal{B}_{m_k}( \hat{f}(t) \star \hat{g}(t) )(\tau) = \tau^{k}\int_{0}^{\tau^{k}}
\mathcal{B}_{m_k}(\hat{f}(t))((\tau^{k}-s)^{1/k}) \star \mathcal{B}_{m_k}(\hat{g}(t))(s^{1/k}) \frac{1}{(\tau^{k}-s)s} ds
\label{Borel_product}
\end{equation}
\end{prop}

Some regular properties of inverse Fourier transform are also involved in our construction.
\begin{prop}\label{prop359}
Let $f \in E_{(\beta,\mu)}$ with $\beta > 0$, $\mu > 1$. The inverse Fourier transform of $f$, defined by
$$ \mathcal{F}^{-1}(f)(x) = \frac{1}{ (2\pi)^{1/2} } \int_{-\infty}^{+\infty} f(m) \exp( ixm ) dm,\quad x\in\mathbb{R},$$
extends to an analytic function on the strip
\begin{equation}
H_{\beta} = \{ z \in \mathbb{C} / |\mathrm{Im}(z)| < \beta \}. \label{strip_H_beta}
\end{equation}
Let $\phi(m) = im f(m) \in E_{(\beta,\mu - 1)}$. Then, it holds
\begin{equation}
\partial_{z} \mathcal{F}^{-1}(f)(z) = \mathcal{F}^{-1}(\phi)(z),\quad z \in H_{\beta}. \label{dz_fourier}
\end{equation}

Let $g \in E_{(\beta,\mu)}$ and put $\psi(m) = \frac{1}{(2\pi)^{1/2}} f \ast g(m)$, the convolution product of $f$ and $g$, for all $m \in \mathbb{R}$.
From Proposition~\ref{prop7}, one gets that $\psi \in E_{(\beta,\mu)}$. Moreover, we have
\begin{equation}
\mathcal{F}^{-1}(f)(z)\mathcal{F}^{-1}(g)(z) = \mathcal{F}^{-1}(\psi)(z),\quad z \in H_{\beta} \label{prod_fourier}
\end{equation}
\end{prop}

\section{Formal and analytic solutions of convolution initial value problems with complex parameters}

Let $k_1, k_2 \geq 1$ and $D_1,D_2 \geq 2$ be integers. For $j\in\{1,2\}$ and $1 \leq l_j \leq D_j$, let
$d_{l_1},\delta_{l_1},\Delta_{l_1,l_2}, \tilde{d}_{l_2},\tilde{\delta}_{l_2} $ be non negative integers.
We assume that 
\begin{equation}
1 = \delta_{1}=\tilde{\delta}_1 \ \ , \ \ \delta_{l_1} < \delta_{l_1+1} \ \ , \ \ \tilde{\delta}_{l_2} < \tilde{\delta}_{l_2+1}
\end{equation}
for all $1 \leq l_1 \leq D_1-1$ and $1\le l_2 \le D_2-1$. We also make the assumptions that
\begin{equation}
d_{D_1} = (\delta_{D_1}-1)(k_1+1) \ \ , \ \ d_{l_1} > (\delta_{l_1}-1)(k_1+1)
\label{assum_dl_delta_l_Delta_l1}
\end{equation}
for all $1 \leq l_1 \leq D_1-1$, and
\begin{equation}
\tilde{d}_{D_2} = (\tilde{\delta}_{D_2}-1)(k_2+1) \ \ , \ \ \tilde{d}_{l_2} > (\tilde{\delta}_{l_2}-1)(k_2+1)
\label{assum_dl_delta_l_Delta_l2}
\end{equation}
for all $1 \leq l_2 \leq D_2-1$. We take
\begin{equation}\label{e9000}
\Delta_{D_1,D_2}=d_{D_1}+\tilde{d}_{D_2}-\delta_{D_1}-\tilde{\delta}_{D_2}+2 \ \ , \ \ \Delta_{D_1,0}=d_{D_1}-\delta_{D_1}+1 \ \ , \ \ \Delta_{0,D_2}=\tilde{d}_{D_2}-\tilde{\delta}_{D_2}+1\\
\end{equation}
Let $Q_1(X),Q_{2}(X),R_0(X)\in \mathbb{C}[X]$, and for $1\le l_1\le D_1$ and $1\le l_2\le D_2$, we take $R_{l_1,l_2}(X)\in\mathbb{C}[X]$ such that
\begin{equation}\label{assum_deg_Q_R00}
R_{D_1,l_2}\equiv R_{l_1,D_2}\equiv0,\quad 1\le l_1\le D_1,\quad 1\le l_2\le D_2
\end{equation}
and such that $R_{D_1,D_2}$ can be factorized in the form $R_{D_1,D_2}(X)=R_{D_1,0}(X)R_{0,D_2}(X)$. We write $R_{D_1}:=R_{D_1,0}$ and $R_{D_2}:=R_{0,D_2}$ for simplicity. Let $P_1,P_2$ be polynomials with coefficients belonging to $\mathcal{O}(\overline{D}(0,\epsilon_0))[X]$, for some $\epsilon_0>0$. We assume that
\begin{equation}\label{raicesgrandes0}
\hbox{deg}(Q_j)\ge \hbox{deg}(R_{D_j}),\quad j\in\{1,2\},
\end{equation}
and
\begin{multline}
\mathrm{deg}(Q_j) \geq \mathrm{deg}(R_{D_j}) \ \ , \ \  \mathrm{deg}(R_{D_1,D_2}) \geq \mathrm{deg}(R_{l_1,l_2})\\
\mathrm{deg}(R_{D_1,D_2}) \geq \mathrm{deg}(P_{1}) \ \ , \ \ \mathrm{deg}(R_{D_1,D_2}) \geq \mathrm{deg}(P_{2}) \ \ , \ \ Q_j(im) \neq 0 \ \ , \ \ R_{D_1,D_2}(im) \neq 0 \label{assum_deg_Q_R0}
\end{multline}
for all $m \in \mathbb{R}$, all $j\in\{1,2\}$ and $0 \leq l_j \leq D_j-1$.

For every non negative integer $n$, we take
$m \mapsto C_{0,n}(m,\epsilon)$, and $m \mapsto F_{n+1}(m,\epsilon)$, belonging to the Banach space $E_{(\beta,\mu)}$ for some $\beta > 0$ and
$\mu > \max\{ \mathrm{deg}(P_{1})+1, \mathrm{deg}(P_{2})+1\}$, depending holomorphically on $\epsilon \in D(0,\epsilon_{0})$, for some positive $\epsilon_0$. We assume the existence of $K_0,T_0>0$ such that 
\begin{equation}
||C_{n_1,n_2}(m,\epsilon)|| _{(\beta,\mu)} \leq K_{0} (\frac{1}{T_{0}})^{n_1+n_2} \ \ , \ \
||F_{n_1,n_2}(m,\epsilon)||_{(\beta,\mu)} \leq K_{0} (\frac{1}{T_{0}})^{n_1+n_2} \label{norm_beta_mu_F_n}
\end{equation}
for all $n_1,n_2 \geq 1$ and $\epsilon \in D(0,\epsilon_{0})$. We write $\bT=(T_1,T_2)$ and put 
$$ C_{0}(\bT,m,\epsilon) = \sum_{n_1,n_2 \geq 0} C_{n_1,n_2}(m,\epsilon) T_1^{n_1}T_2^{n_2} \ \ , \ \
F(\bT,m,\epsilon) = \sum_{n_1,n_2 \geq 1} F_{n_1,n_2}(m,\epsilon) T_1^{n_1}T_2^{n_2} $$
which are convergent series on $D(0,T_{0}/2)\times D(0,T_{0}/2)$ with values in $E_{(\beta,\mu)}$. We consider the following singular initial value problem
\begin{multline}
\left(Q_1(im)\partial_{T_1}-T_1^{(\delta_{D_1}-1)(k_1-1)}\partial_{T_1}^{\delta_{D_1}}R_{D_1}(im)\right)\left(Q_2(im)\partial_{T_2}-T_2^{(\tilde{\delta}_{D_2}-1)(k_2-1)}\partial_{T_2}^{\tilde{\delta}_{D_2}}R_{D_2}(im)\right)U(\bT,m,\epsilon)\\
=\epsilon^{-2}\frac{1}{(2\pi)^{1/2}}\int_{-\infty}^{+\infty}P_{1}(i(m-m_{1}),\epsilon)U(\bT,m-m_{1},\epsilon)P_{2}(im_{1},\epsilon)U(\bT,m_{1},\epsilon) dm_{1}\\
+ \sum_{1\le l_1\le D_1-1,1\le l_2\le D_2-1} \epsilon^{\Delta_{l_1,l_2}-d_{l_1}-\tilde{d}_{l_2}+\delta_{l_1}+\tilde{\delta}_{l_2}- 2} T_1^{d_{l_1}}T_2^{\tilde{d}_{l_2}} \partial_{T_1}^{\delta_{l_1}}\partial_{T_2}^{\tilde{\delta}_{l_2}}R_{\ell_1,\ell_2}(im)U(\bT,m,\epsilon)\\
+ \epsilon^{-2}\frac{1}{(2\pi)^{1/2}}\int_{-\infty}^{+\infty}C_{0}(\bT,m-m_{1},\epsilon)R_{0}(im_{1})U(\bT,m_{1},\epsilon) dm_{1}\\
+\epsilon^{-2}F(\bT,m,\epsilon)
\label{SCP}
\end{multline}
for given initial data $U(T_1,0,m,\epsilon)\equiv U(0,T_2,m,\epsilon) \equiv 0$.\medskip

\begin{prop}\label{prop10} There exists a unique formal series
$$ \hat{U}(\bT,m,\epsilon) = \sum_{n_1,n_2 \geq 1} U_{n_1,n_2}(m,\epsilon) T_1^{n_1}T_2^{n_2} $$
solution of (\ref{SCP}) with initial data $U(T_1,0,m,\epsilon)\equiv U(0,T_2,m,\epsilon) \equiv 0$, where the coefficients
$m \mapsto U_{n_1,n_2}(m,\epsilon)$ belong to $E_{(\beta,\mu)}$ for $\beta>0$ and
$\mu>\max( \mathrm{deg}(P_{1}) + 1, \mathrm{deg}(P_{2})+1)$ given above and depend
holomorphically on $\epsilon$ in $D(0,\epsilon_{0}) \setminus \{ 0 \}$. 
\end{prop}
\begin{proof} Proposition~\ref{prop8} and the conditions in the statement above yield
$U_{n_1,n_2}(m,\epsilon)$ are determined by the following recursion formula and belong to $E_{(\beta,\mu)}$ for all
$\epsilon \in D(0,\epsilon_{0}) \setminus \{ 0 \}$, 
\begin{multline}
(n_1+1)(n_2+1)U_{n_1+1,n_2+1}(m,\epsilon)\\
= \frac{\epsilon^{-2}}{Q_1(im)Q_2(im)}
\sum_{   \stackrel{n_{11}+n_{12}=n_1}{n_{11},n_{12} \geq 1}        }\sum_{\stackrel{n_{21}+n_{22}=n_2}{n_{21},n_{22} \geq 1}}
\frac{1}{(2\pi)^{1/2}} \\
\times \int_{-\infty}^{+\infty} P_{1}(i(m-m_{1}))U_{n_{11},n_{21}}(m-m_{1},\epsilon)
P_{2}(im_{1})U_{n_{12}n_{22}}(m_{1},\epsilon) dm_{1}\\
+\frac{R_{D_1}(im)}{Q_1(im)}(n_2+1)\prod_{j=0}^{\delta_{D_1}-1}(n_1+\delta_{D_1}-(\delta_{D_1}-1)(k_1-1)-j)U_{n_1+\delta_{D_1}-(\delta_{D_1}-1)(k_1-1),n_2+1}\\
+\frac{R_{D_2}(im)}{Q_2(im)}(n_1+1)\prod_{j=0}^{\tilde{\delta}_{D_2}-1}(n_2+\tilde{\delta}_{D_2}-(\tilde{\delta}_{D_2}-1)(k_2-1)-j)U_{n_1+1,n_2+\tilde{\delta}_{D_2}-(\tilde{\delta}_{D_2}-1)(k_2-1)}\\
-\frac{R_{D_1}(im)}{Q_1(im)}\frac{R_{D_2}(im)}{Q_2(im)}\prod_{j_1=0}^{\delta_{D_1}-1}\prod_{j_2=0}^{\tilde{\delta}_{D_2}-1}(n_1+\delta_{D_1}-(\delta_{D_1}-1)(k_1-1)-j_1)(n_2+\tilde{\delta}_{D_2}-(\tilde{\delta}_{D_2}-1)(k_2-1)-j_2)\\
\times U_{n_1+\delta_{D_1}-(\delta_{D_1}-1)(k_1-1),n_2+\tilde{\delta}_{D_2}-(\tilde{\delta}_{D_2}-1)(k_2-1)}\\
+\sum_{1\le l_1\le D_1-1,1\le l_2\le D_2-1} \epsilon^{\Delta_{l_1,l_2}-d_{l_1}-\tilde{d}_{l_2}+\delta_{l_1}+\tilde{\delta}_{l_2}- 2} \frac{R_{\ell_1,\ell_2}(im)}{Q_1(im)Q_2(im)}\prod_{j_1=0}^{\delta_{\ell_1}-1}(n_1+\delta_{l_1}-d_{l_1}-j_1)\\
\times \prod_{j_2=0}^{\tilde{\delta}_{\ell_2}-1}(n_2+\tilde{\delta}_{l_2}-\tilde{d}_{l_2}-j_2)U_{n_1+\delta_{l_1}-d_{l_1},n_2+\tilde{\delta}_{l_2}-\tilde{d}_{l_2}}\\
+ \frac{\epsilon^{-2}}{Q_1(im)Q_2(im)}
\sum_{\stackrel{n_{11}+n_{12}=n_1}{n_{11},n_{12} \geq 1}} \sum_{\stackrel{n_{21}+n_{22}=n_2}{n_{21},n_{22} \geq 1}}\frac{1}{(2\pi)^{1/2}} \int_{-\infty}^{+\infty}
C_{n_{11},n_{21}}(m-m_{1},\epsilon)R_{0}(im_{1})U_{n_{12},n_{22}}(m_{1},\epsilon) dm_{1}\\
+ \frac{\epsilon^{-2}}{Q_1(im)Q_2(im)}
\sum_{\stackrel{n_{11}+n_{12}=n_1}{n_{11},n_{12} \geq 1}} \frac{1}{(2\pi)^{1/2}} \int_{-\infty}^{+\infty}
C_{n_{11},0}(m-m_{1},\epsilon)R_{0}(im_{1})U_{n_{12},n_2}(m_{1},\epsilon) dm_{1}\\
+ \frac{\epsilon^{-2}}{Q_1(im)Q_2(im)}
\sum_{\stackrel{n_{21}+n_{22}=n_2}{n_{21},n_{22} \geq 1}} \frac{1}{(2\pi)^{1/2}} \int_{-\infty}^{+\infty}
C_{0,n_{21}}(m-m_{1},\epsilon)R_{0}(im_{1})U_{n_{1},n_{22}}(m_{1},\epsilon) dm_{1}\\
+ \frac{\epsilon^{-2}}{Q_1(im)Q_2(im)}
\frac{1}{(2\pi)^{1/2}} \int_{-\infty}^{+\infty}
C_{0,0}(m-m_{1},\epsilon)R_{0}(im_{1})U_{n_{1},n_{2}}(m_{1},\epsilon) dm_{1}\\
+ \frac{\epsilon^{-2}}{Q_1(im)Q_2(im)}F_{n_1,n_2}(m,\epsilon)
\end{multline}
for all $n_1 \geq \max_{1 \leq l \leq D_1}d_{l}$ and $n_2 \geq \max_{1 \leq l \leq D_2}\tilde{d}_{l}$. 
\end{proof}

The following relations (see~\cite{taya}, p. 40) hold
\begin{multline}
T_1^{\delta_{D_1}(k_1+1)} \partial_{T_1}^{\delta_{D_1}} = (T_1^{k_1+1}\partial_{T_1})^{\delta_{D_1}} +
\sum_{1 \leq p_1 \leq \delta_{D_1}-1} A_{\delta_{D_1},p_1} T_1^{k_1(\delta_{D_1}-p_1)} (T_1^{k_1+1}\partial_{T_1})^{p_1}\\
=(T_1^{k_1+1}\partial_{T_1})^{\delta_{D_1}}+A_{\delta_{D_1}}(T_1,\partial_{T_1})
 \label{expand_op_diff}
\end{multline}
\begin{multline}
T_2^{\tilde{\delta}_{D_2}(k_2+1)} \partial_{T_2}^{\tilde{\delta}_{D_2}} = (T_2^{k_2+1}\partial_{T_2})^{\tilde{\delta}_{D_2}} +
\sum_{1 \leq p_2 \leq \tilde{\delta}_{D_2}-1} \tilde{A}_{\tilde{\delta}_{D_2},p_2} T_2^{k_2(\tilde{\delta}_{D_2}-p_2)} (T_2^{k_2+1}\partial_{T_2})^{p_2}\\
=(T_2^{k_2+1}\partial_{T_2})^{\tilde{\delta}_{D_2}}+\tilde{A}_{\tilde{\delta}_{D_2}}(T_2,\partial_{T_2}) \label{expand_op_diff2}
\end{multline}
for some real numbers $A_{\delta_{D_1},p_1}$, $p_1=1,\ldots,\delta_{D_1}-1$ and $\tilde{A}_{\tilde{\delta}_{D_2},p_2}$, $p_2=1,\ldots,\tilde{\delta}_{D_2}-1$. We write $A_{D_1}$ (resp. $\tilde{A}_{D_2}$) for $A_{\delta_{D_1}}$ (resp. $\tilde{A}_{\tilde{\delta}_{D_2}}$) for the sake of simplicity. Let $d_{l_1,k_1},\tilde{d}_{l_1,k_2} \geq 0$ satisfying
\begin{equation}
d_{l_1} + k_1 + 1 = \delta_{l_1}(k_1+1) + d_{l_1,k_1}\qquad \tilde{d}_{l_2} + k_2 + 1 = \tilde{\delta}_{l_2}(k_2+1) + \tilde{d}_{l_2,k_2}
\end{equation}
for all $1 \leq l_1 \leq D_1-1$ and $1 \leq l_2 \leq D_2-1$. Multiplying the equation (\ref{SCP}) by $T_1^{k_1+1}T_2^{k_2+1}$ and taking into account (\ref{expand_op_diff},\ref{expand_op_diff2}), we rewrite (\ref{SCP}) in the form

\begin{multline}
\left(Q_1(im)T_1^{k_1+1}\partial_{T_1}-\left((T_1^{k_1+1}\partial_{T_1})^{\delta_{D_1}}+A_{D_1}(T_1,\partial_{T_1})\right)R_{D_1}(im)\right)\\
\times\left(Q_2(im)T_2^{k_2+1}\partial_{T_2}-\left((T_2^{k_2+1}\partial_{T_2})^{\tilde{\delta}_{D_2}}+\tilde{A}_{D_2}(T_2,\partial_{T_2})\right)R_{D_2}(im)\right)U(\bT,m,\epsilon)  \\
= \epsilon^{-2}T_1^{k_1+1}T_2^{k_2+1}
\frac{1}{(2\pi)^{1/2}}\int_{-\infty}^{+\infty} P_{1}(i(m-m_{1}),\epsilon)U(\bT,m-m_{1},\epsilon)P_{2}(im_{1},\epsilon)U(T,m_{1},\epsilon) dm_{1} \\
+ \sum_{1 \leq l_1 \leq D_1-1,1 \leq l_2 \leq D_2-1} \epsilon^{\Delta_{l_1,l_2}-d_{l_1}-d_{l_2}+\delta_{l_1}+\tilde{\delta}_{l_2}-2} T_1^{\delta_{l_1}(k_1+1)+d_{l_1,k_1}}\partial_{T_1}^{\delta_{l_1}}\\
\times T_2^{\tilde{\delta}_{l_2}(k_2+1)+\tilde{d}_{l_2,k_2}}\partial_{T_2}^{\tilde{\delta}_{l_2}}R_{l_1,l_2}(im)U(\bT,m,\epsilon) \\
+ \epsilon^{-2}T_1^{k_1+1}T_2^{k_2+1}
\frac{1}{(2\pi)^{1/2}}\int_{-\infty}^{+\infty} C_{0}(\bT,m-m_{1},\epsilon) R_{0}(im_{1})U(\bT,m_{1},\epsilon) dm_{1}\\
+ \epsilon^{-2}T_1^{k_1+1}T_2^{k_2+1}F(\bT,m,\epsilon)
\label{SCP_irregular}
\end{multline}

We write $\btau=(\tau_1,\tau_2)$ and denote by $\omega_{\bk}(\btau,m,\epsilon)$ for the formal $m_{k_1}-$Borel transform with respect to $T_1$ and the $m_{k_2}-$Borel transform with respect to $T_2$ of $\hat{U}(T_1,T_2,m,\epsilon)$. Let $\varphi_{\bk}(\tau_1,\tau_2,m,\epsilon)$ denote the formal $m_{k_1}-$Borel transform with respect to $T_1$ and the $m_{k_2}-$Borel transform with respect to $T_2$ of
$C_{0}(\bT,m,\epsilon)$; and $\psi_{\bk}(\btau,m,\epsilon)$ the formal $m_{k_1}-$Borel transform with respect to $T_1$ and the $m_{k_2}-$Borel transform with respect to $T_2$ of $F(\bT,m,\epsilon)$,
\begin{multline*}
 \omega_{\bk}(\btau,m,\epsilon) = \sum_{n_1,n_2 \geq 1} U_{n_1,n_2}(m,\epsilon) \frac{\tau_1^{n_1}}{\Gamma(\frac{n_1}{k_1})}\frac{\tau_2^{n_2}}{\Gamma(\frac{n_2}{k_2})}, \\
\varphi_{\bk}(\btau,m,\epsilon) = \sum_{n_1,n_2 \geq 1} C_{n_1,n_2}(m,\epsilon) \frac{\tau_1^{n_1}}{\Gamma(\frac{n_1}{k_1})}\frac{\tau_2^{n_2}}{\Gamma(\frac{n_2}{k_2})},\\
\varphi_{\bk}^1(\tau_1,m,\epsilon) = \sum_{n_1\geq 1} C_{n_1,0}(m,\epsilon) \frac{\tau_1^{n_1}}{\Gamma(\frac{n_1}{k_1})},\\
\varphi_{\bk}^2(\tau_2,m,\epsilon) = \sum_{n_2\geq 1} C_{0,n_2}(m,\epsilon) \frac{\tau_2^{n_2}}{\Gamma(\frac{n_2}{k_2})},\\
\psi_{\bk}(\btau,m,\epsilon) = \sum_{n_1,n_2 \geq 1} F_{n_1,n_2}(m,\epsilon) \frac{\tau_1^{n_1}}{\Gamma(\frac{n_1}{k_1})}\frac{\tau_2^{n_2}}{\Gamma(\frac{n_2}{k_2})}
\end{multline*}
Using (\ref{norm_beta_mu_F_n}) we arrive at
$\varphi_{\bk}(\btau,\epsilon) \in F_{(\bnu,\beta,\mu,\bk,\epsilon)}^{\bd}$ and
$\psi_{\bk}(\btau,m,\epsilon) \in F_{(\bnu,\beta,\mu,\bk,\epsilon)}^{\bd}$, for
all $\epsilon \in D(0,\epsilon_{0}) \setminus \{ 0 \}$, any unbounded sectors $S_{d_1}$ and $S_{d_2}$ centered at 0 and bisecting directions $d_1 \in \mathbb{R}$ and $d_2\in\mathbb{R}$, respectively, for some $\bnu=(\nu_1,\nu_2)\in(0,+\infty)^2$. Indeed, we have 
\begin{multline}
||\varphi_{\bk}(\btau,m,\epsilon)||_{(\bnu,\beta,\mu,\bk,\epsilon)} \leq \sum_{n_1,n_2 \geq 1}
||C_{n_1,n_2}(m,\epsilon)||_{(\beta,\mu)}\\
\times (\sup_{\btau \in(\bar{D}(0,\rho) \cup S_{d_1})\times(\bar{D}(0,\rho) \cup S_{d_2})}
\frac{1 + |\frac{\tau_1}{\epsilon}|^{2k_1}}{|\frac{\tau_1}{\epsilon}|}\frac{1 + |\frac{\tau_2}{\epsilon}|^{2k_2}}{|\frac{\tau_2}{\epsilon}|} \exp(-\nu_1 |\frac{\tau_1}{\epsilon}|^{k_1}-\nu_2 |\frac{\tau_2}{\epsilon}|^{k_2})
\frac{|\tau_1|^{n_1}|\tau_2|^{n_2}}{\Gamma(\frac{n_1}{k_1})\Gamma(\frac{n_2}{k_2})}),\\
||\psi_{\bk}(\btau,m,\epsilon)||_{(\bnu,\beta,\mu,\bk,\epsilon)} \leq \sum_{n_1,n_2 \geq 1}
||F_{n_1,n_2}(m,\epsilon)||_{(\beta,\mu)}\\
\times (\sup_{\btau \in (\bar{D}(0,\rho) \cup S_{d_1})\times (\bar{D}(0,\rho) \cup S_{d_2})}
\frac{1 + |\frac{\tau_1}{\epsilon}|^{2k_1}}{|\frac{\tau_1}{\epsilon}|}\frac{1 + |\frac{\tau_2}{\epsilon}|^{2k_2}}{|\frac{\tau_2}{\epsilon}|} \exp(-\nu_1 |\frac{\tau_1}{\epsilon}|^{k_1}-\nu_2 |\frac{\tau_2}{\epsilon}|^{k_2})
\frac{|\tau_1|^{n_1}|\tau_2|^{n_2}}{\Gamma(\frac{n_1}{k_1})\Gamma(\frac{n_2}{k_2})}) \label{maj_norm_psi_k_1}
\end{multline}
Using classical estimates and Stirling formula we guarantee the existence of $A_{1},A_{2}>0$ depending on $\bnu,\bk$ such that, if $\epsilon_0A_2<T_0$, then
for all $\epsilon \in D(0,\epsilon_{0}) \setminus \{ 0 \}$. One has
\begin{equation}
||\varphi_{\bk}(\btau,m,\epsilon)||_{(\bnu,\beta,\mu,\bk,\epsilon)} \leq \frac{A_{1}K_{0}}{\left(\frac{T_0}{\epsilon_0 A_2}-1\right)^2},\quad
||\psi_{\bk}(\btau,m,\epsilon)||_{(\bnu,\beta,\mu,\bk,\epsilon)} \leq  \frac{A_{1}K_{0}}{\left(\frac{T_0}{\epsilon_0 A_2}-1\right)^2}
\label{norm_F_varphi_k_psi_k_epsilon_0}
\end{equation}
for all $\epsilon \in D(0,\epsilon_{0}) \setminus \{ 0 \}$.

One can also check that in the case that $\epsilon_{0}$ fulfills $\epsilon_{0}A_{2} < T_{0}$, then
\begin{equation}
||\varphi_{\bk}^1(\tau_1,m,\epsilon)||_{(\nu_1,\beta,\mu,k_1,\epsilon)} \leq \frac{A_{1}K_{0}}{\frac{T_0}{\epsilon_0 A_2}-1},\quad
||\varphi_{\bk}^2(\tau_2,m,\epsilon)||_{(\nu_2,\beta,\mu,k_2,\epsilon)} \leq \frac{A_{1}K_{0}}{\frac{T_0}{\epsilon_0 A_2}-1}
\label{norm_F_varphi_k_psi_k_epsilon_012}
\end{equation}
for all $\epsilon \in D(0,\epsilon_{0}) \setminus \{ 0 \}$.

From the properties of the formal $m_{k_1}-$Borel and $m_{k_2}-$Borel transforms stated in Proposition~\ref{prop6} we get the following equation satisfied by $\omega_{\bk}(\btau,m,\epsilon)$. In the following writing, $\mathcal{A}_{D_1}(\tau_1,\partial_{T_1})$ (resp. $\tilde{\mathcal{A}}_{D_2}(\tau_2,\partial_{T_2})$) stands for the $m_{k_1}-$Borel transform of the operator $A_{D_1}(T_1,\partial_{T_1})$ with respect to $T_1$ (resp. the $m_{k_2}-$Borel transform of the operator $\tilde{A}_{D_2}(T_2,\partial_{T_2})$ with respect to $T_2$), i.e.

\begin{equation}\label{e558}
\mathcal{A}_{\delta_{D_1}}\omega_{\bk}(\btau,m,\epsilon)=\sum_{1\le p_1\le \delta_{D_1}-1}\frac{A_{\delta_{D_1},p_1} \tau_1^{k_1}  }{\Gamma(\delta_{D_1}-p_1)}\int_{0}^{\tau_1^{k_1}}(\tau_1^{k_1}-s_1)^{\delta_{D_1}-p_1-1}k_1s_1^{p_1}\omega_{\bk}(s_1^{1/k_1},\tau_2,m,\epsilon)\frac{ds_1}{s_1},
\end{equation}
$$\tilde{\mathcal{A}}_{\tilde{\delta}_{D_2}}\omega_{\bk}(\btau,m,\epsilon)=\sum_{1\le p_2\le \tilde{\delta}_{D_2}-1}\frac{\tilde{A}_{\tilde{\delta}_{D_2},p_2}\tau_2^{k_2}           }{\Gamma(\tilde{\delta}_{D_2}-p_2)}\int_{0}^{\tau_2^{k_2}}(\tau_2^{k_2}-s_2)^{\tilde{\delta}_{D_2}-p_2-1}k_2s_2^{p_2}\omega_{\bk}(\tau_1,s_2^{1/k_2},m,\epsilon)\frac{ds_2}{s_2}.$$
We arrive at

\begin{multline}
(Q_1(im) k_1 \tau_1^{k_1}-(k_1\tau_1^{k_1})^{\delta_{D_1}}R_{D_1}(im))(Q_2(im) k_2 \tau_2^{k_2}-(k_2\tau_2^{k_2})^{\tilde{\delta}_{D_2}}R_{D_2}(im))\omega_{\bk}(\btau,m,\epsilon)\\
=(Q_1(im) k_1 \tau_1^{k_1}-(k_1\tau_1^{k_1})^{\delta_{D_1}}R_{D_1}(im))\tilde{\mathcal{A}}_{D_2}R_{D_2}(im)\omega_{\bk}(\btau,m,\epsilon)\\
+(Q_2(im) k_2 \tau_2^{k_2}-(k_2\tau_2^{k_2})^{\tilde{\delta}_{D_2}}R_{D_2}(im))\mathcal{A}_{D_1}R_{D_1}(im)\omega_{\bk}(\btau,m,\epsilon)\\
-\mathcal{A}_{D_1}\tilde{\mathcal{A}}_{D_2}R_{D_1}(im)R_{D_2}(im)\omega_{\bk}(\btau,m,\epsilon)\\
+\epsilon^{-2}
\frac{\tau_1^{k_1}\tau_2^{k_2}}{\Gamma(1 + \frac{1}{k_1})\Gamma(1 + \frac{1}{k_2})} \int_{0}^{\tau_1^{k_1}}\int_{0}^{\tau_2^{k_2}}
(\tau_1^{k_1}-s_1)^{1/k_1}(\tau_2^{k_2}-s_2)^{1/k_2}\\
\times \left( \frac{1}{(2\pi)^{1/2}} s_1s_2\int_{0}^{s_1}\int_{0}^{s_2} \int_{-\infty}^{+\infty} \right.
P_{1}(i(m-m_{1}),\epsilon)\omega_{\bk}((s_1-x_1)^{1/k_1},(s_2-x_2)^{1/k_2},m-m_{1},\epsilon)\\
\left. \times  P_{2}(im_{1},\epsilon)
\omega_{\bk}(x_1^{1/k_1},x_2^{1/k_2},m_{1},\epsilon) \frac{1}{(s_1-x_1)x_1(s_2-x_2)x_2} dm_{1}dx_2dx_1 \right) \frac{ds_2}{s_2}\frac{ds_1}{s_1}\\
+ \sum_{1 \leq l_1 \leq D_1-1,1 \leq l_2 \leq D_2-1}R_{l_1,l_2}(im) \epsilon^{\Delta_{l_1,l_2}-d_{l_1}-\tilde{d}_{l_2}+\delta_{l_1}+\tilde{\delta}_{l_2}-2}\frac{\tau_1^{k_1}\tau_2^{k_2}}{\Gamma\left(\frac{d_{l_1,k_1}}{k_1}\right)\Gamma\left(\frac{\tilde{d}_{l_2,k_2}}{k_2}\right)}\\
\int_0^{\tau_1^{k_1}}\int_0^{\tau_2^{k_2}}(\tau_1^{k_1}-s_1)^{d_{l_1,k_1}/k_1-1}(\tau_2^{k_2}-s_2)^{\tilde{d}_{l_2,k_2}/k_2-1}k_1^{\delta_{l_1}}k_2^{\tilde{\delta}_{l_2}}s_1^{\delta_{l_1}}s_2^{\tilde{\delta}_{l_2}}\omega_{\bk}(s_1^{1/k_1},s_2^{1/k_2},m,\epsilon)\frac{ds_2}{s_2}\frac{ds_1}{s_1}\\
+\epsilon^{-2}\frac{\tau_1^{k_1}\tau_2^{k_2}}{\Gamma\left(1+\frac{1}{k_1}\right)\Gamma\left(1+\frac{1}{k_2}\right)}\int_0^{\tau_1^{k_1}}\int_0^{\tau_2^{k_2}}(\tau_1^{k_1}-s_1)^{1/k_1}(\tau_2^{k_2}-s_2)^{1/k_2}\frac{1}{(2\pi)^{1/2}}\\
\times s_1s_2\int_0^{s_1}\int_0^{s_2}\int_{-\infty}^{\infty}\varphi_{\bk}((s_1-x_1)^{1/k_1},(s_2-x_2)^{1/k_2},m-m_1,\epsilon)R_0(im_1)\omega_{\bk}(x_1^{1/k_1},x_2^{1/k_2},m_1,\epsilon)\\
\frac{1}{(s_1-x_1)x_1(s_2-x_2)x_2}dm_1dx_2dx_1\frac{ds_2}{s_2}\frac{ds_1}{s_1}\\
+\epsilon^{-2}\frac{\tau_1^{k_1}\tau_2^{k_2}}{\Gamma\left(1+\frac{1}{k_1}\right)\Gamma\left(1+\frac{1}{k_2}\right)}\int_0^{\tau_1^{k_1}}\int_0^{\tau_2^{k_2}}(\tau_1^{k_1}-s_1)^{1/k_1}(\tau_2^{k_2}-s_2)^{1/k_2}\frac{1}{(2\pi)^{1/2}}\\
\times s_1\int_0^{s_1}\int_{-\infty}^{\infty}\varphi^1_{\bk}((s_1-x_1)^{1/k_1},m-m_1,\epsilon)R_0(im_1)\omega_{\bk}(x_1^{1/k_1},s_2^{1/k_2},m_1,\epsilon)\frac{1}{(s_1-x_1)x_1}dm_1dx_1\frac{ds_2}{s_2}\frac{ds_1}{s_1}\\
+\epsilon^{-2}\frac{\tau_1^{k_1}\tau_2^{k_2}}{\Gamma\left(1+\frac{1}{k_1}\right)\Gamma\left(1+\frac{1}{k_2}\right)}\int_0^{\tau_1^{k_1}}\int_0^{\tau_2^{k_2}}(\tau_1^{k_1}-s_1)^{1/k_1}(\tau_2^{k_2}-s_2)^{1/k_2}\frac{1}{(2\pi)^{1/2}}\\
\times s_2\int_0^{s_2}\int_{-\infty}^{\infty}\varphi^2_{\bk}((s_2-x_2)^{1/k_2},m-m_1,\epsilon)R_0(im_1)\omega_{\bk}(s_1^{1/k_1},x_2^{1/k_2},m_1,\epsilon)\frac{1}{(s_2-x_2)x_2}dm_1dx_2\frac{ds_2}{s_2}\frac{ds_1}{s_1}\\
+\epsilon^{-2}\frac{\tau_1^{k_1}\tau_2^{k_2}}{\Gamma\left(1+\frac{1}{k_1}\right)\Gamma\left(1+\frac{1}{k_2}\right)}\int_0^{\tau_1^{k_1}}\int_0^{\tau_2^{k_2}}(\tau_1^{k_1}-s_1)^{1/k_1}(\tau_2^{k_2}-s_2)^{1/k_2}\frac{1}{(2\pi)^{1/2}}\\
\times \int_{-\infty}^{\infty}C_{0,0}(m-m_1,\epsilon)R_0(im_1)\omega_{\bk}(s_1^{1/k_1},s_2^{1/k_2},m_1,\epsilon)dm_1\frac{ds_2}{s_2}\frac{ds_1}{s_1}\\
+\epsilon^{-2}\frac{\tau_1^{k_1}\tau_2^{k_2}}{\Gamma\left(1+\frac{1}{k_1}\right)\Gamma\left(1+\frac{1}{k_2}\right)}\int_0^{\tau_1^{k_1}}\int_0^{\tau_2^{k_2}}(\tau_1^{k_1}-s_1)^{1/k_1}(\tau_2^{k_2}-s_2)^{1/k_2}\psi_{\bk}(s_1^{1/k_1},s_2^{1/k_2},m,\epsilon)\frac{ds_2}{s_2}\frac{ds_1}{s_1}.\label{k_Borel_equation}
\end{multline}

For the sake of simplicity, we write the previous equation in the form
\begin{multline}\label{e588}
(Q_1(im) k_1 \tau_1^{k_1}-(k_1\tau_1^{k_1})^{\delta_{D_1}}R_{D_1}(im))(Q_2(im) k_2 \tau_2^{k_2}-(k_2\tau_2^{k_2})^{\delta_{D_2}}R_{D_2}(im))\omega_{\bk}(\btau,m,\epsilon)\\
=(Q_1(im) k_1 \tau_1^{k_1}-(k_1\tau_1^{k_1})^{\delta_{D_1}}R_{D_1}(im))\tilde{\mathcal{A}}_{D_2}R_{D_2}(im)\omega_{\bk}(\btau,m,\epsilon)\\
+(Q_2(im) k_2 \tau_2^{k_2}-(k_2\tau_2^{k_2})^{\delta_{D_2}}R_{D_2}(im))\mathcal{A}_{D_1}R_{D_1}(im)\omega_{\bk}(\btau,m,\epsilon)\\
-\mathcal{A}_{D_1}\tilde{\mathcal{A}}_{D_2}R_{D_1}(im)R_{D_2}(im)\omega_{\bk}(\btau,m,\epsilon)+\Theta(\btau,m,\epsilon)\omega_{\bk}(\btau,m,\epsilon).
\end{multline}

We make the additional assumption that for $j=1,2$, there exist unbounded sectors
$$ S_{Q_j,R_{D_j}} = \{ z \in \mathbb{C} / |z| \geq r_{Q_j,R_{D_j}} \ \ , \ \ |\mathrm{arg}(z) - d_{Q_j,R_{D_j}}| \leq \eta_{Q_j,R_{D_j}} \} $$
with directions $d_{Q_j,R_{D_j}} \in \mathbb{R}$, aperture $\eta_{Q_j,R_{D_j}}>0$ for some radius $r_{Q_j,R_{D_j}}>0$ such that
\begin{equation}
\frac{Q_j(im)}{R_{D_j}(im)} \in S_{Q_j,R_{D_j}} \label{quotient_Q_RD_in_S}
\end{equation} 
for all $m \in \mathbb{R}$. We consider the polynomial $P_{m,j}(\tau_j) = Q_j(im)k_j - R_{D_j}(im)k_j^{\delta_{D_j}}\tau_j^{(\delta_{D_j}-1)k_j}$ and assume that $\{q_{l,1}\}_{0\le l\le(\delta_{D_1}-1)k_1-1}$ and $\{q_{l,2}\}_{0\le l\le(\tilde{\delta}_{D_2}-1)k_2-1}$ are the complex roots of each polynomial, for $m\in\R$. Following an analogous manner as in the construction of~\cite{lama}, one can choose unbounded sectors $S_{d_1}$ and $S_{d_2}$, with vertex at 0 and $\rho>0$ such that
\begin{equation}\label{e494}
|P_{m,1}(\tau_1)|\ge C_{P}(r_{Q_1,R_{D_1}})^{\frac{1}{(\delta_{D_1}-1)k_1}}|R_{D_1}(im)|(1+|\tau_1|^{k_1})^{(\delta_{D_1}-1)-\frac{1}{k_1}},
\end{equation}
for all $\tau_1\in S_{d_1}\cup \overline{D}(0,\rho)$, and $m\in\R$; and
\begin{equation}\label{e494b}
|P_{m,2}(\tau_2)|\ge C_{P}(r_{Q_2,R_{D_2}})^{\frac{1}{(\tilde{\delta}_{D_2}-1)k_2}}|R_{D_2}(im)|(1+|\tau_2|^{k_2})^{(\tilde{\delta}_{D_2}-1)-\frac{1}{k_2}},
\end{equation}
for all $\tau_2\in S_{d_2}\cup \overline{D}(0,\rho)$, and $m\in\R$. From now on, we write
$$ C_{k_1}=C_{P}(r_{Q_1,R_{D_1}})^{\frac{1}{(\delta_{D_1}-1)k_1}},\quad C_{k_2}:=C_{P}(r_{Q_2,R_{D_2}})^{\frac{1}{(\tilde{\delta}_{D_2}-1)k_2}}$$
for a more compact writing.

Let $\bd=(d_1,d_2)$. The next result guarantees the existence of an element in $F^{\bd}_{(\bnu,\beta,\mu,\bk,\epsilon)}$ which turns out to be a fixed point for certain operator to be described, solution of (\ref{k_Borel_equation}). Here $\beta,\mu$ are fixed at the beginning of this section.

\begin{prop}\label{prop11} Under the assumption that
\begin{equation}
\delta_{D_1} \geq \delta_{l_1} + \frac{2}{k_1},\quad \tilde{\delta}_{D_2} \geq \tilde{\delta}_{l_2} + \frac{2}{k_2}, \quad \Delta_{l_1,l_2} + k_1(1 - \delta_{D_1}) +k_2(1 - \tilde{\delta}_{D_2})+ 2 \geq 0,
\label{constraints_k_Borel_equation}
\end{equation}
for all $1 \leq l_1 \leq D_1-1$, $1 \leq l_2 \leq D_2-1$, there exist constants $\varpi,\zeta_1,\zeta_2>0$ (depending on $Q_{1},Q_{2},\bk,C_{P},\mu,\bnu,\epsilon_{0},R_{l_1,l_2},\Delta_{l_1,l_2},\delta_{l_1},\tilde{\delta}_{l_2},d_{l_1},\tilde{d}_{l_2}$ for
$0 \leq l_1 \leq D_1$ and $0\le l_2\le D_2$) such that if
\begin{multline}
||\varphi_{\bk}(\btau,m,\epsilon)||_{(\bnu,\beta,\mu,\bk,\epsilon)} \leq \zeta_{1} \ \ , \ \
||\psi_{\bk}(\btau,m,\epsilon)||_{(\bnu,\beta,\mu,\bk,\epsilon)} \leq \zeta_{2}\\
||\varphi_{\bk}^1(\tau_1,m,\epsilon)||_{(\nu_1,\beta,\mu,k_1,\epsilon)} \leq \zeta^1_{1} \ \ , \ \  ||\varphi_{\bk}^2(\tau_2,m,\epsilon)||_{(\nu_2,\beta,\mu,k_2,\epsilon)} \leq \zeta^2_{1} \ \ , \ \  ||C_{0,0}(m,\epsilon)||_{(\beta,\mu)} \leq \zeta^0_{1}
    \label{norm_F_varphi_k_psi_k_small}
\end{multline}

for all $\epsilon \in D(0,\epsilon_{0}) \setminus \{ 0 \}$, the equation (\ref{k_Borel_equation}) has a unique solution
$\omega_{\bk}^{\bd}(\btau,m,\epsilon)$ in the space $F_{(\bnu,\beta,\mu,\bk,\epsilon)}^{\bd}$ where $\beta,\mu>0$ are defined in
Proposition~\ref{prop6} which verifies $||\omega_{\bk}^{\bd}(\btau,m,\epsilon)||_{(\bnu,\beta,\mu,\bk,\epsilon)} \leq \varpi$, for all
$\epsilon \in D(0,\epsilon_{0}) \setminus \{ 0 \}$.
\end{prop}
\begin{proof} 

Let $\epsilon\in D(0,\epsilon_0)\setminus\{0\}$. We consider the operator $\mathcal{H}_\epsilon$, defined by 

\begin{equation}
\mathcal{H}_\epsilon(\omega(\btau,m)):=\sum_{j=1}^{8}\mathcal{H}^j_\epsilon(\omega(\btau,m))\label{k_Borel_equation2}
\end{equation}
where
\begin{multline*}
\mathcal{H}^1_\epsilon(\omega(\btau,m)):=\frac{R_{D_2}(im)}{P_{m,2}(\tau_2)}\frac{\tilde{\mathcal{A}}_{D_2}}{\tau_2^{k_2}}\omega(\btau,m)+\frac{R_{D_1}(im)}{P_{m,1}(\tau_1)}\frac{\mathcal{A}_{D_1}}{\tau_1^{k_1}}\omega(\btau,m)\\
-\frac{R_{D_1}(im)}{P_{m,1}(\tau_1)}\frac{R_{D_2}(im)}{P_{m,2}(\tau_2)}\frac{\mathcal{A}_{D_1}}{\tau_1^{k_1}}\frac{\tilde{\mathcal{A}}_{D_2}}{\tau_2^{k_2}}\omega(\btau,m)
\end{multline*}
\begin{multline*}
\mathcal{H}^2_\epsilon(\omega(\btau,m)):=\frac{\epsilon^{-2}}{P_{m,1}(\tau_1)P_{m,2}(\tau_2)\Gamma(1 + \frac{1}{k_1})\Gamma(1 + \frac{1}{k_2})} \int_{0}^{\tau_1^{k_1}}\int_{0}^{\tau_2^{k_2}}
(\tau_1^{k_1}-s_1)^{1/k_1}(\tau_2^{k_2}-s_2)^{1/k_2}\\
\times \left( \frac{1}{(2\pi)^{1/2}} s_1s_2\int_{0}^{s_1}\int_{0}^{s_2} \int_{-\infty}^{+\infty} \right.
P_{1}(i(m-m_{1}),\epsilon)\omega((s_1-x_1)^{1/k_1},(s_2-x_2)^{1/k_2},m-m_{1})\\
\left. \times  P_{2}(im_{1},\epsilon)
\omega(x_1^{1/k_1},x_2^{1/k_2},m_{1}) \frac{1}{(s_1-x_1)x_1(s_2-x_2)x_2} dm_1dx_2dx_1 \right) \frac{ds_2}{s_2}\frac{ds_1}{s_1}
\end{multline*}
\begin{multline*}
\mathcal{H}^3_\epsilon(\omega(\btau,m)):=
\sum_{1 \leq l_1 \leq D_1-1,1 \leq l_2 \leq D_2-1}\frac{R_{l_1,l_2}(im)}{P_{m,1}(\tau_1)P_{m,2}(\tau_2)} \epsilon^{\Delta_{l_1,l_2}-d_{l_1}-\tilde{d}_{l_2}+\delta_{l_1}+\tilde{\delta}_{l_2}-2}\frac{1}{\Gamma\left(\frac{d_{l_1,k_1}}{k_1}\right)\Gamma\left(\frac{\tilde{d}_{l_2,k_2}}{k_2}\right)}\\
\int_0^{\tau_1^{k_1}}\int_0^{\tau_2^{k_2}}(\tau_1^{k_1}-s_1)^{d_{l_1,k_1}/k_1-1}(\tau_2^{k_2}-s_2)^{\tilde{d}_{l_2,k_2}/k_2-1}k_1^{\delta_{l_1}}k_2^{\tilde{\delta}_{l_2}}s_1^{\delta_{l_1}}s_2^{\tilde{\delta}_{l_2}}\omega(s_1^{1/k_1},s_2^{1/k_2},m)\frac{ds_2}{s_2}\frac{ds_1}{s_1}
\end{multline*}
\begin{multline*}
\mathcal{H}^4_\epsilon(\omega(\btau,m)):=
\frac{\epsilon^{-2}}{P_{m,1}(\tau_1)P_{m,2}(\tau_2)\Gamma\left(1+\frac{1}{k_1}\right)\Gamma\left(1+\frac{1}{k_2}\right)}\int_0^{\tau_1^{k_1}}\int_0^{\tau_2^{k_2}}(\tau_1^{k_1}-s_1)^{1/k_1}(\tau_2^{k_2}-s_2)^{1/k_2}\\
\times \frac{s_1s_2}{(2\pi)^{1/2}}\int_0^{s_1}\int_0^{s_2}\int_{-\infty}^{\infty}\varphi_{\bk}((s_1-x_1)^{1/k_1},(s_2-x_2)^{1/k_2},m-m_1,\epsilon)R_0(im_1)\omega_{\bk}(x_1^{1/k_1},x_2^{1/k_2},m_1)\\
\frac{1}{(s_1-x_1)x_1(s_2-x_2)x_2}dm_1dx_2dx_1\frac{ds_2}{s_2}\frac{ds_1}{s_1}
\end{multline*}
\begin{multline*}
\mathcal{H}^5_\epsilon(\omega(\btau,m)):=\frac{\epsilon^{-2}}{P_{m,1}(\tau_1)P_{m,2}(\tau_2)\Gamma\left(1+\frac{1}{k_1}\right)\Gamma\left(1+\frac{1}{k_2}\right)}\int_0^{\tau_1^{k_1}}\int_0^{\tau_2^{k_2}}(\tau_1^{k_1}-s_1)^{1/k_1}(\tau_2^{k_2}-s_2)^{1/k_2}\\
\times \frac{s_1}{(2\pi)^{1/2}}\int_0^{s_1}\int_{-\infty}^{\infty}\varphi^1_{\bk}((s_1-x_1)^{1/k_1},m-m_1,\epsilon)R_0(im_1)\omega_{\bk}(x_1^{1/k_1},s_2^{1/k_2},m_1)\frac{1}{(s_1-x_1)x_1}dm_1dx_1\frac{ds_2}{s_2}\frac{ds_1}{s_1}
\end{multline*}
\begin{multline*}
\mathcal{H}^6_\epsilon(\omega(\btau,m)):= \frac{\epsilon^{-2}}{P_{m,1}(\tau_1)P_{m,2}(\tau_2)\Gamma\left(1+\frac{1}{k_1}\right)\Gamma\left(1+\frac{1}{k_2}\right)}\int_0^{\tau_1^{k_1}}\int_0^{\tau_2^{k_2}}(\tau_1^{k_1}-s_1)^{1/k_1}(\tau_2^{k_2}-s_2)^{1/k_2}\\
\times \frac{s_2}{(2\pi)^{1/2}}\int_0^{s_2}\int_{-\infty}^{\infty}\varphi^2_{\bk}((s_2-x_2)^{1/k_2},m-m_1,\epsilon)R_0(im_1)\omega_{\bk}(s_1^{1/k_1},x_2^{1/k_2},m_1)\frac{1}{(s_2-x_2)x_2}dm_1dx_2\frac{ds_2}{s_2}\frac{ds_1}{s_1}
\end{multline*}
\begin{multline*}
\mathcal{H}^7_\epsilon(\omega(\btau,m)):=\frac{\epsilon^{-2}}{P_{m,1}(\tau_1)P_{m,2}(\tau_2)\Gamma\left(1+\frac{1}{k_1}\right)\Gamma\left(1+\frac{1}{k_2}\right)}\int_0^{\tau_1^{k_1}}\int_0^{\tau_2^{k_2}}(\tau_1^{k_1}-s_1)^{1/k_1}(\tau_2^{k_2}-s_2)^{1/k_2}\\
\times \frac{1}{(2\pi)^{1/2}}\int_{-\infty}^{\infty}C_{0,0}(m-m_1,\epsilon)R_0(im_1)\omega_{\bk}(s_1^{1/k_1},s_2^{1/k_2},m_1)dm_1\frac{ds_2}{s_2}\frac{ds_1}{s_1}
\end{multline*}
\begin{multline*}
\mathcal{H}^8_\epsilon(\omega(\btau,m)):=
\frac{\epsilon^{-2}}{P_{m,1}(\tau_1)P_{m,2}(\tau_2)\Gamma\left(1+\frac{1}{k_1}\right)\Gamma\left(1+\frac{1}{k_2}\right)}\\
\times\int_0^{\tau_1^{k_1}}\int_0^{\tau_2^{k_2}}(\tau_1^{k_1}-s_1)^{1/k_1}(\tau_2^{k_2}-s_2)^{1/k_2}\psi_{\bk}(s_1^{1/k_1},s_2^{1/k_2},m,\epsilon)\frac{ds_2}{s_2}\frac{ds_1}{s_1}.
\end{multline*}

Let $\varpi>0$ and assume that $\omega(\btau,m)\in F^{\bd}_{\nu,\beta,\mu,\bk,\epsilon}$. Assume that $\left\|\omega(\btau,m)\right\|_{(\bnu,\beta,\mu,\bk,\epsilon)}\le \varpi$ for all $\epsilon\in D(0,\epsilon_0)\setminus\{0\}$. We first obtain the existence of $\varpi>0$ such that the operator $\mathcal{H}_\epsilon$ sends $\overline{B}(0,\varpi)\subseteq F^{\bd}_{(\bnu,\beta,\mu,\bk,\epsilon)}$ into itself. Here, $\overline{B}(0,\varpi)$ stands for the closed ball of radius $\varpi$, centerd at 0, in the Banach space $F^{\bd}_{(\bnu,\beta,\mu,\bk,\epsilon)}$.


Using Lemma~\ref{lema1} and Proposition~\ref{prop2lama}, with (\ref{e494}) and (\ref{e494b}) we get
\begin{equation}\label{e559}
\left\|\frac{R_{D_2}(im)}{P_{m,2}(\tau_2)}\frac{\tilde{\mathcal{A}}_{D_2}}{\tau_2^{k_2}}\omega(\btau,m)\right\|_{(\bnu,\beta,\mu,\bk,\epsilon)}\le\sum_{1\le p_2\le\tilde{\delta}_{D_2}-1}\frac{\tilde{A}_{\tilde{\delta}_{D_2},p_2}}{\Gamma(\tilde{\delta}_{D_2}-p_2)}\frac{C_2}{C_{k_2}}|\epsilon|\left\|\omega(\btau,m)\right\|_{(\bnu,\beta,\mu,\bk,\epsilon)}
\end{equation}

\begin{equation}\label{e560}
\left\|\frac{R_{D_1}(im)}{P_{m,1}(\tau_1)}\frac{\mathcal{A}_{D_1}}{\tau_1^{k_1}}\omega(\btau,m)\right\|_{(\bnu,\beta,\mu,\bk,\epsilon)}\le\sum_{1\le p_1\le\delta_{D_1}-1}\frac{A_{\delta_{D_1},p_1}}{\Gamma(\delta_{D_1}-p_1)}\frac{C_2}{C_{k_1}}|\epsilon|\left\|\omega(\btau,m)\right\|_{(\bnu,\beta,\mu,\bk,\epsilon)}
\end{equation}

\begin{multline}\label{e561}
\left\|\frac{R_{D_1}(im)}{P_{m,1}(\tau_1)}\frac{\mathcal{A}_{D_1}}{\tau_1^{k_1}}\frac{R_{D_2}(im)}{P_{m,2}(\tau_2)}\frac{\tilde{\mathcal{A}}_{D_2}}{\tau_2^{k_2}}\omega(\btau,m)\right\|_{(\bnu,\beta,\mu,\bk,\epsilon)}\\
\le\sum_{1\le p_1\le\delta_{D_1}-1}\sum_{1\le p_2\le\tilde{\delta}_{D_2}-1}\frac{A_{\delta_{D_1},p_1}}{\Gamma(\delta_{D_1}-p_1)}\frac{\tilde{A}_{\tilde{\delta}_{D_2},p_2}}{\Gamma(\tilde{\delta}_{D_2}-p_2)}\frac{(C_2)^2}{C_{k_1}C_{k_2}}|\epsilon|^2\left\|\omega(\btau,m)\right\|_{(\bnu,\beta,\mu,\bk,\epsilon)}
\end{multline}

Proposition~\ref{prop137} and Lemma~\ref{lema1} yield

\begin{multline}\label{e632}
\left\|
\frac{\epsilon^{-2}}{P_{m,1}(\tau_1)P_{m,2}(\tau_2)} \int_{0}^{\tau_1^{k_1}}\int_{0}^{\tau_2^{k_2}}
(\tau_1^{k_1}-s_1)^{1/k_1}(\tau_2^{k_2}-s_2)^{1/k_2}\right.\\
\times \left( \int_{0}^{s_1}\int_{0}^{s_2} \int_{-\infty}^{+\infty} \right.
P_{1}(i(m-m_{1}),\epsilon)\omega((s_1-x_1)^{1/k_1},(s_2-x_2)^{1/k_2},m-m_{1})\\
\left.\left. \times  P_{2}(im_{1},\epsilon)
\omega(x_1^{1/k_1},x_2^{1/k_2},m_{1}) \frac{1}{(s_1-x_1)x_1(s_2-x_2)x_2} dm_1dx_2dx_1 \right) ds_2 ds_1\right\|_{(\bnu,\beta,\mu,\bk,\epsilon)}\\
\le \frac{C_3}{C_{k_1}C_{k_2}}\left\|\omega(\btau,m)\right\|_{(\bnu,\beta,\mu,\bk,\epsilon)}^2
\end{multline}

From Proposition~\ref{prop2} and Lemma~\ref{lema1}, we get 

\begin{multline}\label{e645}
\left\|\frac{R_{l_1,l_2}(im)}{P_{m,1}(\tau_1)P_{m,2}(\tau_2)} \epsilon^{\Delta_{l_1,l_2}-d_{l_1}-\tilde{d}_{l_2}+\delta_{l_1}+\tilde{\delta}_{l_2}-2}\right.\\
\left.\int_0^{\tau_1^{k_1}}\int_0^{\tau_2^{k_2}}(\tau_1^{k_1}-s_1)^{d_{l_1,k_1}/k_1-1}(\tau_2^{k_2}-s_2)^{\tilde{d}_{l_2,k_2}/k_2-1}s_1^{\delta_{l_1}}s_2^{\tilde{\delta}_{l_2}}\omega(s_1^{1/k_1},s_2^{1/k_2},m)\frac{ds_2}{s_2}\frac{ds_1}{s_1}\right\|_{(\bnu,\beta,\mu,\bk,\epsilon)}\\
\le \frac{C_2}{C_{k_1}C_{k_2}}|\epsilon|^{\Delta_{l_1,l_2}+\delta_{l_1}(1+k_1)-k_1\delta_{D_1}+\tilde{\delta}_{l_2}(1+k_2)-k_2\tilde{\delta}_{D_2}}\left\|\omega(\btau,m)\right\|_{(\bnu,\beta,\mu,\bk,\epsilon)}
\end{multline}

Also, Proposition~\ref{prop137} and Lemma~\ref{lema1} yield

\begin{multline}\label{e654}
\left\|\epsilon^{-2}\frac{1}{P_{m,1}(\tau_1)P_{m,2}(\tau_2)}\int_0^{\tau_1^{k_1}}\int_0^{\tau_2^{k_2}}(\tau_1^{k_1}-s_1)^{1/k_1}(\tau_2^{k_2}-s_2)^{1/k_2}\right.\\
\times \int_0^{s_1}\int_0^{s_2}\int_{-\infty}^{\infty}\varphi_{\bk}((s_1-x_1)^{1/k_1},(s_2-x_2)^{1/k_2},m-m_1,\epsilon)R_0(im_1)\\
\times \omega(x_1^{1/k_1},x_2^{1/k_2},m_1)
\left.\frac{1}{(s_1-x_1)x_1(s_2-x_2)x_2}dm_1dx_2dx_1ds_2ds_1\right\|_{(\bnu,\beta,\mu,\bk,\epsilon)}\\
\le \frac{C_3}{C_{k_1}C_{k_2}} \zeta_1 \left\|\omega(\btau,m)\right\|_{(\bnu,\beta,\mu,\bk,\epsilon)}
\end{multline}

In view of Proposition~\ref{prop5} and Lemma~\ref{lema1} we get

\begin{multline}\label{e663}
\left\|\epsilon^{-2}\frac{1}{P_{m,1}(\tau_1)P_{m,2}(\tau_2)}\int_0^{\tau_1^{k_1}}\int_0^{\tau_2^{k_2}}(\tau_1^{k_1}-s_1)^{1/k_1}(\tau_2^{k_2}-s_2)^{1/k_2}\right.\\
\times \int_0^{s_1}\int_{-\infty}^{\infty}\varphi^1_{\bk}((s_1-x_1)^{1/k_1},m-m_1,\epsilon)R_0(im_1)\omega(x_1^{1/k_1},s_2^{1/k_2},m_1)\frac{1}{(s_1-x_1)x_1}dm_1dx_1\frac{ds_2}{s_2}ds_1\\
\le \frac{C_{3.2}}{C_{k_1}C_{k_2}}\left\|\omega(\btau,m)\right\|_{(\bnu,\beta,\mu,\bk,\epsilon)}\left\|\varphi_{\bk}^{1}(\tau_1,m,\epsilon)\right\|_{(\nu_1,\beta,\mu,k_1,\epsilon)}\\
\le \frac{C_{3.2}}{C_{k_1}C_{k_2}}\left\|\omega(\btau,m)\right\|_{(\bnu,\beta,\mu,\bk,\epsilon)}\zeta_1^{1}
\end{multline}

\begin{multline}\label{e664}
\left\|\epsilon^{-2}\frac{1}{P_{m,1}(\tau_1)P_{m,2}(\tau_2)}\int_0^{\tau_1^{k_1}}\int_0^{\tau_2^{k_2}}(\tau_1^{k_1}-s_1)^{1/k_1}(\tau_2^{k_2}-s_2)^{1/k_2}\right.\\
\times \int_0^{s_1}\int_{-\infty}^{\infty}\varphi^2_{\bk}((s_2-x_2)^{1/k_2},m-m_1,\epsilon)R_0(im_1)\omega(s_1^{1/k_1},x_2^{1/k_2},m_1)\frac{1}{(s_2-x_2)x_2}dm_1dx_2ds_2\frac{ds_1}{s_1}\\
\le \frac{C_{3.2}}{C_{k_1}C_{k_2}}\left\|\omega(\btau,m)\right\|_{(\bnu,\beta,\mu,\bk,\epsilon)}\left\|\varphi_{\bk}^{2}(\tau_2,m,\epsilon)\right\|_{(\nu_2,\beta,\mu,k_2,\epsilon)}\\
\le \frac{C_{3.2}}{C_{k_1}C_{k_2}}\left\|\omega(\btau,m)\right\|_{(\bnu,\beta,\mu,\bk,\epsilon)}\zeta_1^{2}
\end{multline}

Proposition~\ref{prop6} and Lemma~\ref{lema1} yield
\begin{multline}\label{e683}
\left\|\epsilon^{-2}\frac{1}{P_{m,1}(\tau_1)P_{m,2}(\tau_2)}\int_0^{\tau_1^{k_1}}\int_0^{\tau_2^{k_2}}(\tau_1^{k_1}-s_1)^{1/k_1}(\tau_2^{k_2}-s_2)^{1/k_2}\right.\\
\left.\times \int_{-\infty}^{\infty}C_{0,0}(m-m_1,\epsilon)R_0(im_1)\omega(s_1^{1/k_1},s_2^{1/k_2},m_1)dm_1\frac{ds_2}{s_2}\frac{ds_1}{s_1}\right\|_{(\bnu,\beta,\mu,\bk,\epsilon)}\\
\le \frac{C_4}{C_{k_1}C_{k_2}}\zeta^0_1 \left\|\omega(\btau,m)\right\|_{(\bnu,\beta,\mu,\bk,\epsilon)}
\end{multline}

Finally, from Proposition~\ref{prop1} and Lemma~\ref{lema1} we get 

\begin{multline}\label{e684}
\left\|\epsilon^{-2}\frac{1}{P_{m,1}(\tau_1)P_{m,2}(\tau_2)}\int_0^{\tau_1^{k_1}}\int_0^{\tau_2^{k_2}}(\tau_1^{k_1}-s_1)^{1/k_1}(\tau_2^{k_2}-s_2)^{1/k_2}\right.\\
\left.\times\psi_{\bk}(s_1^{1/k_1},s_2^{1/k_2},m,\epsilon)\frac{ds_2}{s_2}\frac{ds_1}{s_1}\right\|_{(\bnu,\beta,\mu,\bk,\epsilon)}\\
\le \frac{C_1}{C_{k_1}C_{k_2}}\left\|\psi_{\bk}(s_1^{1/k_1},s_2^{1/k_2},m,\epsilon)\right\|_{(\bnu,\beta,\mu,\bk,\epsilon)}\\
\le \frac{C_1}{C_{k_1}C_{k_2}}\zeta_2.
\end{multline}

In view of (\ref{norm_F_varphi_k_psi_k_small}) and by choosing $\epsilon_0,\varpi,\zeta_1,\zeta_2,\zeta_1^1,\zeta_1^0,\zeta_1^2>0$ such that

\begin{multline}\label{e728}
\sum_{1\le p_2\le\tilde{\delta}_{D_2}-1}\frac{C_{k_1}\tilde{A}_{\tilde{\delta}_{D_2},p_2}}{\Gamma(\tilde{\delta}_{D_2}-p_2)}C_2|\epsilon_0|\varpi+\sum_{1\le p_1\le\delta_{D_1}-1}\frac{C_{k_2}A_{\delta_{D_1},p_1}}{\Gamma(\delta_{D_1}-p_1)}C_2|\epsilon_0|\varpi\\
+\sum_{1\le p_1\le\delta_{D_1}-1}\sum_{1\le p_2\le\tilde{\delta}_{D_2}-1}\frac{A_{\delta_{D_1},p_1}}{\Gamma(\delta_{D_1}-p_1)}\frac{\tilde{A}_{\tilde{\delta}_{D_2},p_2}}{\Gamma(\tilde{\delta}_{D_2}-p_2)}(C_2)^2|\epsilon_0|^2\varpi+\frac{C_3\varpi^2}{\Gamma(1+\frac{1}{k_1})\Gamma(1+\frac{1}{k_2})(2\pi)^{1/2}}\\
+\sum_{1\le l_1\le D_1-1,1\le l_2\le D_2-1}C_2|\epsilon_0|^{\Delta_{l_1,l_2}+\delta_{l_1}(1+k_1)-k_1\delta_{D_1}+\tilde{\delta}_{l_2}(1+k_2)-k_2\tilde{\delta}_{D_2}}\frac{k_1^{\delta_{l_1}}k_2^{\tilde{\delta}_{l_2}}}{\Gamma(\frac{d_{l_1,k_1}}{k_1})\Gamma(\frac{\tilde{d}_{l_2,k_2}}{k_2})}\varpi\\
+\frac{C_3 \zeta_1\varpi}{\Gamma(1+\frac{1}{k_1})\Gamma(1+\frac{1}{k_2})(2\pi)^{1/2}}+\frac{C_{3.2} \zeta_1^1\varpi}{\Gamma(1+\frac{1}{k_1})\Gamma(1+\frac{1}{k_2})(2\pi)^{1/2}}+\frac{C_{3.2} \zeta_1^2\varpi}{\Gamma(1+\frac{1}{k_1})\Gamma(1+\frac{1}{k_2})(2\pi)^{1/2}}\\
+\frac{C_4 \zeta_1^0\varpi}{\Gamma(1+\frac{1}{k_1})\Gamma(1+\frac{1}{k_2})(2\pi)^{1/2}}+\frac{C_1 \zeta_2}{\Gamma(1+\frac{1}{k_1})\Gamma(1+\frac{1}{k_2})\hbox{min}_{m\in\R}|R_{D_1}(im)R_{D_2}(im)|}\le \varpi C_{k_1}C_{k_2}.
\end{multline}

In view of (\ref{e559}), (\ref{e560}), (\ref{e561}), (\ref{e632}), (\ref{e645}), (\ref{e654}), (\ref{e663}), (\ref{e664}), (\ref{e683}), (\ref{e684}), and (\ref{e728}), one gets that the operator $\mathcal{H}_\epsilon$ is such that $\mathcal{H}_\epsilon(\overline{B}(0,\varpi))\subseteq \overline{B}(0,\varpi)$. The next stage of the proof is to show that, indeed, $\mathcal{H}_\epsilon$ is a contractive map in that ball. Let $\omega_1,\omega_2\in F_{(\bnu,\beta,\mu,\bk,\epsilon)}^{\bd}$ with $\left\|\omega_j\right\|_{(\bnu,\beta,\mu,\bk,\epsilon)}\le\varpi$. Then, it holds that
\begin{equation}\label{e741}
\left\|\mathcal{H}_\epsilon(\omega_1)-\mathcal{H}_\epsilon(\omega_2)\right\|_{(\bnu,\beta,\mu,\bk,\epsilon)}\le\frac{1}{2}\left\|\omega_1-\omega_2\right\|_{(\bnu,\beta,\mu,\bk,\epsilon)},
\end{equation} for all $\epsilon\in D(0,\epsilon_0)\setminus\{0\}$.

Analogous estimates as in (\ref{e559}), (\ref{e560}), (\ref{e561}), (\ref{e645}), (\ref{e654}), (\ref{e663}), (\ref{e664}) and (\ref{e683}) yield

\begin{multline}\label{e559b}
\left\|\frac{R_{D_2}(im)}{P_{m,2}(\tau_2)}\frac{\tilde{\mathcal{A}}_{D_2}}{\tau_2^{k_2}}(\omega_1(\btau,m)-\omega_2(\btau,m))\right\|_{(\bnu,\beta,\mu,\bk,\epsilon)}\\
\le\sum_{1\le p_2\le\tilde{\delta}_{D_2}-1}\frac{\tilde{A}_{\tilde{\delta}_{D_2},p_2}}{\Gamma(\tilde{\delta}_{D_2}-p_2)}\frac{C_2}{C_{k_2}}|\epsilon|\left\|\omega_1(\btau,m)-\omega_2(\btau,m)\right\|_{(\bnu,\beta,\mu,\bk,\epsilon)}
\end{multline}

\begin{multline}\label{e560b}
\left\|\frac{R_{D_1}(im)}{P_{m,1}(\tau_1)}\frac{\mathcal{A}_{D_1}}{\tau_1^{k_1}}(\omega_1(\btau,m)-\omega_2(\btau,m))\right\|_{(\bnu,\beta,\mu,\bk,\epsilon)}\\
\le\sum_{1\le p_1\le\delta_{D_1}-1}\frac{A_{\delta_{D_1},p_1}}{\Gamma(\delta_{D_1}-p_1)}\frac{C_2}{C_{k_1}}|\epsilon|\left\|\omega_1(\btau,m)-\omega_2(\btau,m)\right\|_{\bnu,\beta,\mu,\bk,\epsilon}
\end{multline}

\begin{multline}\label{e561b}
\left\|\frac{R_{D_1}(im)}{P_{m,1}(\tau_1)}\frac{\mathcal{A}_{D_1}}{\tau_1^{k_1}}\frac{R_{D_2}(im)}{P_{m,2}(\tau_2)}\frac{\tilde{\mathcal{A}}_{D_2}}{\tau_2^{k_2}}(\omega_1(\btau,m)-\omega_2(\btau,m))\right\|_{(\bnu,\beta,\mu,\bk,\epsilon)}\\
\le\sum_{1\le p_1\le\delta_{D_1}-1}\sum_{1\le p_2\le\tilde{\delta}_{D_2}-1}\frac{A_{\delta_{D_1},p_1}}{\Gamma(\delta_{D_1}-p_1)}\frac{\tilde{A}_{\tilde{\delta}_{D_2},p_2}}{\Gamma(\tilde{\delta}_{D_2}-p_2)}\frac{(C_2)^2}{C_{k_1}C_{k_2}}|\epsilon|^2\left\|\omega_1(\btau,m)-\omega_2(\btau,m)\right\|_{(\bnu,\beta,\mu,\bk,\epsilon)}
\end{multline}

\begin{multline}\label{e645b}
\left\|\frac{R_{l_1,l_2}(im)}{P_{m,1}(\tau_1)P_{m,2}(\tau_2)} \epsilon^{\Delta_{l_1,l_2}-d_{l_1}-\tilde{d}_{l_2}+\delta_{l_1}+\tilde{\delta}_{l_2}-2}\int_0^{\tau_1^{k_1}}\int_0^{\tau_2^{k_2}}(\tau_1^{k_1}-s_1)^{d_{l_1,k_1}/k_1-1}(\tau_2^{k_2}-s_2)^{\tilde{d}_{l_2,k_2}/k_2-1}s_1^{\delta_{l_1}}s_2^{\tilde{\delta}_{l_2}}\right.\\
\left.\times(\omega_1(s_1^{1/k_1},s_2^{1/k_2},m)-\omega_2(s_1^{1/k_1},s_2^{1/k_2},m))\frac{ds_2}{s_2}\frac{ds_1}{s_1}\right\|_{(\bnu,\beta,\mu,\bk,\epsilon)}\\
\le \frac{C_2}{C_{k_1}C_{k_2}}|\epsilon|^{\Delta_{l_1,l_2}+\delta_{l_1}(1+k_1)-k_1\delta_{D_1}+\tilde{\delta}_{l_2}(1+k_2)-k_2\tilde{\delta}_{D_2}}\left\|\omega_1(\btau,m)-\omega_2(\btau,m)\right\|_{(\bnu,\beta,\mu,\bk,\epsilon)}
\end{multline}

\begin{multline}\label{e654b}
\left\|\epsilon^{-2}\frac{1}{P_{m,1}(\tau_1)P_{m,2}(\tau_2)}\int_0^{\tau_1^{k_1}}\int_0^{\tau_2^{k_2}}(\tau_1^{k_1}-s_1)^{1/k_1}(\tau_2^{k_2}-s_2)^{1/k_2}\right.\\
\times \int_0^{s_1}\int_0^{s_2}\int_{-\infty}^{\infty}\varphi_{\bk}((s_1-x_1)^{1/k_1},(s_2-x_2)^{1/k_2},m-m_1,\epsilon)R_0(im_1)\\
\times 
\left.\frac{\omega_1(x_1^{1/k_1},x_2^{1/k_2},m_1)-\omega_2(x_1^{1/k_1},x_2^{1/k_2},m_1)}{(s_1-x_1)x_1(s_2-x_2)x_2}dm_1dx_2dx_1ds_2ds_1\right\|_{(\bnu,\beta,\mu,\bk,\epsilon)}\\
\le \frac{C_3}{C_{k_1}C_{k_2}} \zeta_1 \left\|\omega_1(\btau,m)-\omega_2(\btau,m)\right\|_{(\bnu,\beta,\mu,\bk,\epsilon)}
\end{multline}

\begin{multline}\label{e663b}
\left\|\epsilon^{-2}\frac{1}{P_{m,1}(\tau_1)P_{m,2}(\tau_2)}\int_0^{\tau_1^{k_1}}\int_0^{\tau_2^{k_2}}(\tau_1^{k_1}-s_1)^{1/k_1}(\tau_2^{k_2}-s_2)^{1/k_2}\right.\\
\times \int_0^{s_1}\int_{-\infty}^{\infty}\varphi^1_{\bk}((s_1-x_1)^{1/k_1},m-m_1,\epsilon)R_0(im_1)(\omega_{1}(x_1^{1/k_1},s_2^{1/k_2},m_1)-\omega_{2}(x_1^{1/k_1},s_2^{1/k_2},m_1))\\
\frac{1}{(s_1-x_1)x_1}dm_1dx_1\frac{ds_2}{s_2}ds_1\le \frac{C_{3.2}}{C_{k_1}C_{k_2}}\left\|\omega_1(\btau,m)-\omega_2(\btau,m)\right\|_{(\bnu,\beta,\mu,\bk,\epsilon)}\zeta_1^{1}
\end{multline}

\begin{multline}\label{e664b}
\left\|\epsilon^{-2}\frac{1}{P_{m,1}(\tau_1)P_{m,2}(\tau_2)}\int_0^{\tau_1^{k_1}}\int_0^{\tau_2^{k_2}}(\tau_1^{k_1}-s_1)^{1/k_1}(\tau_2^{k_2}-s_2)^{1/k_2}\right.\\
\times \int_0^{s_1}\int_{-\infty}^{\infty}\varphi^2_{\bk}((s_2-x_2)^{1/k_2},m-m_1,\epsilon)R_0(im_1)(\omega_{1}(s_1^{1/k_1},x_2^{1/k_2},m_1)-\omega_{2}(s_1^{1/k_1},x_2^{1/k_2},m_1))\\
\frac{1}{(s_2-x_2)x_2}dm_1dx_2ds_2\frac{ds_1}{s_1}\le \frac{C_{3.2}}{C_{k_1}C_{k_2}}\left\|\omega_1(\btau,m)-\omega_2(\btau,m)\right\|_{(\bnu,\beta,\mu,\bk,\epsilon)}\zeta_1^{2}
\end{multline}

\begin{multline}\label{e683b}
\left\|\epsilon^{-2}\frac{1}{P_{m,1}(\tau_1)P_{m,2}(\tau_2)}\int_0^{\tau_1^{k_1}}\int_0^{\tau_2^{k_2}}(\tau_1^{k_1}-s_1)^{1/k_1}(\tau_2^{k_2}-s_2)^{1/k_2}\right.\\
\left.\times \int_{-\infty}^{\infty}C_{0,0}(m-m_1,\epsilon)R_0(im_1)(\omega_{1}(s_1^{1/k_1},s_2^{1/k_2},m_1)-\omega_{2}(s_1^{1/k_1},s_2^{1/k_2},m_1))dm_1\frac{ds_2}{s_2}\frac{ds_1}{s_1}\right\|_{(\bnu,\beta,\mu,\bk,\epsilon)}\\
\le \frac{C_4}{C_{k_1}C_{k_2}}\zeta_1^0 \left\|\omega_1(\btau,m)-\omega_2(\btau,m)\right\|_{(\bnu,\beta,\mu,\bk,\epsilon)}.
\end{multline}

Finally, put 
$$W_1:=w_{1}((s_1-x_1)^{1/k_1},(s_2-x_2)^{1/k_2},m-m_{1}) - w_{2}((s_1-x_1)^{1/k_1},(s_2-x_2)^{1/k_2},m-m_{1}),$$
and $W_2:=w_{1}(x_1^{1/k_1},x_2^{1/k_2},m_{1}) - w_{2}(x_1^{1/k_1},x_2^{1/k_2},m_{1}).$
Then, taking into account that
\begin{multline}
P_{1}(i(m-m_{1}),\epsilon)w_{1}((s_1-x_1)^{1/k_1},(s_2-x_2)^{1/k_2},m-m_{1})P_{2}(im_{1},\epsilon)w_{1}(x_1^{1/k_1},x_2^{1/k_2},m_{1})\\
-P_{1}(i(m-m_{1}),\epsilon)w_{2}((s_1-x_1)^{1/k_1},(s_2-x_2)^{1/k_2},m-m_{1})P_{2}(im_{1},\epsilon)w_{2}(x_1^{1/k_1},x_2^{1/k_2},m_{1})\\
= P_{1}(i(m-m_{1}),\epsilon)W_1 P_2(im_1,\epsilon)w_{1}(x_1^{1/k_1},x_2^{1/k_2},m_{1})\\
+ P_{1}(i(m-m_{1}),\epsilon)w_{2}((s_1-x_1)^{1/k_1},(s_2-x_2)^{1/k_2},m-m_{1})P_2(im_1,\epsilon)W_2,
\end{multline} 

and by using Lemma~\ref{lema1} and analogous estimates as in (\ref{e632}) and (\ref{e494}), we get 

\begin{multline}\label{e632b}
\left\|
\frac{\epsilon^{-2}}{P_{m,1}(\tau_1)P_{m,2}(\tau_2)} \int_{0}^{\tau_1^{k_1}}\int_{0}^{\tau_2^{k_2}}
(\tau_1^{k_1}-s_1)^{1/k_1}(\tau_2^{k_2}-s_2)^{1/k_2}\right.\\
\times \left( \int_{0}^{s_1}\int_{0}^{s_2} \int_{-\infty}^{+\infty} \right.
P_{1}(i(m-m_{1}),\epsilon)W_1\\
\left.\left. \times  P_{2}(im_{1},\epsilon)
W_2 \frac{1}{(s_1-x_1)x_1(s_2-x_2)x_2} dm_1dx_2dx_1 \right) ds_2 ds_1\right\|_{(\bnu,\beta,\mu,\bk,\epsilon)}\\
\le \frac{2C_3\varpi}{C_{k_1}C_{k_2}}\left\|\omega_1(\btau,m)-\omega_2(\btau,m)\right\|_{(\bnu,\beta,\mu,\bk,\epsilon)}^2
\end{multline}

Let $\varpi,\epsilon_0,\zeta_1>0$ such that

\begin{multline}\label{e728b}
\sum_{1\le p_2\le\tilde{\delta}_{D_2}-1}\frac{C_{k_1}\tilde{A}_{\tilde{\delta}_{D_2},p_2}}{\Gamma(\tilde{\delta}_{D_2}-p_2)}C_2|\epsilon_0|+\sum_{1\le p_1\le\delta_{D_1}-1}\frac{C_{k_2}A_{\delta_{D_1},p_1}}{\Gamma(\delta_{D_1}-p_1)}C_2|\epsilon_0|\\
+\sum_{1\le p_1\le\delta_{D_1}-1}\sum_{1\le p_2\le\tilde{\delta}_{D_2}-1}\frac{A_{\delta_{D_1},p_1}}{\Gamma(\delta_{D_1}-p_1)}\frac{\tilde{A}_{\tilde{\delta}_{D_2},p_2}}{\Gamma(\tilde{\delta}_{D_2}-p_2)}(C_2)^2|\epsilon_0|^2\\
+\frac{2C_3\varpi}{\Gamma(1+\frac{1}{k_1})\Gamma(1+\frac{1}{k_2})(2\pi)^{1/2}}\\
+\sum_{1\le l_1\le D_1-1,1\le l_2\le D_2-1}C_2|\epsilon_0|^{\Delta_{l_1,l_2}+\delta_{l_1}(1+k_1)-k_1\delta_{D_1}+\tilde{\delta}_{l_2}(1+k_2)-k_2\tilde{\delta}_{D_2}}\frac{k_1^{\delta_{l_1}}k_2^{\tilde{\delta}_{l_2}}}{\Gamma(\frac{d_{l_1,k_1}}{k_1})\Gamma(\frac{\tilde{d}_{l_2,k_2}}{k_2})}\\
+\frac{C_3}{\Gamma(1+\frac{1}{k_1})\Gamma(1+\frac{1}{k_2})(2\pi)^{1/2}}\zeta_1+\frac{C_{3.2}}{\Gamma(1+\frac{1}{k_1})\Gamma(1+\frac{1}{k_2})(2\pi)^{1/2}}\zeta_1^1+\frac{C_{3.2}}{\Gamma(1+\frac{1}{k_1})\Gamma(1+\frac{1}{k_2})(2\pi)^{1/2}}\zeta_1^2\\
+\frac{C_4}{\Gamma(1+\frac{1}{k_1})\Gamma(1+\frac{1}{k_2})(2\pi)^{1/2}}\zeta_1^0\le \frac{1}{2}C_{k_1}C_{k_2}.
\end{multline} 

Then, (\ref{e728b}) combined with (\ref{e559b}), (\ref{e560b}), (\ref{e561b}), (\ref{e645b}), (\ref{e654b}), (\ref{e663b}), (\ref{e664b}), (\ref{e683b}) and (\ref{e632b}) yields (\ref{e741}).

We consider the ball $\bar{B}(0,\varpi) \subset F_{(\bnu,\beta,\mu,\bk,\epsilon)}^{\bd}$ constructed above. It turns out to be a complete metric space for the norm $||.||_{(\bnu,\beta,\mu,\bk,\epsilon)}$. As $\mathcal{H}_{\epsilon}$ is a
contractive map from $\bar{B}(0,\varpi)$ into itself, the classical contractive mapping theorem, guarantees the existence of a unique fixed point $\omega_{\bk}(\btau,m,\epsilon)\in \bar{B}(0,\varpi)\subseteq F_{(\bnu,\beta,\mu,\bk,\epsilon)}^{\bd}$ for $\mathcal{H}_{\epsilon}$. The function $\omega_{\bk}(\btau,m,\epsilon)$ depends holomorphically on $\epsilon$ in $D(0,\epsilon_{0}) \setminus \{ 0 \}$. By construction, $\omega_{\bk}(\btau,m,\epsilon)$ defines a solution of the equation (\ref{k_Borel_equation}). 
\end{proof}

Regarding the construction of the auxiliary equations, one can obtain the analytic solutions of (\ref{SCP}) by means of Laplace transform.

\begin{prop}\label{prop12} Under the hypotheses of Proposition~\ref{prop11}, choose the sectors $S_{d_1},S_{d_2}$ and $S_{Q_1,R_{D_1}},S_{Q_2,R_{D_2}}$
in such a way that the roots of $P_{m,1}(\tau_1)$ and $P_{m,2}(\tau_2)$ fall appart from $S_{d_1}$ and $S_{d_2}$, respectively, as stated before (\ref{e494}).  

Notice that they apply for any small enough $\epsilon_{0}>0$, provided that (\ref{norm_F_varphi_k_psi_k_epsilon_0}) and (\ref{norm_F_varphi_k_psi_k_epsilon_012}) hold.

Let $S_{d_1,\theta_1,h'|\epsilon|}, S_{d_2,\theta_2,h'|\epsilon|}$ be bounded sectors with aperture $\pi/k_j < \theta_j < \pi/k_j + 2\delta_{S,j}$, for $j=1,2$
(where $2\delta_{S,j}$ is the opening of $S_{d_j}$), with direction $d_j$ and radius $h'|\epsilon|$ for some $h'>0$ independent of $\epsilon$. We choose $0 < \beta' < \beta$.

Then, equation 
(\ref{SCP}) with initial condition $U(T_1,0,m,\epsilon)\equiv U(0,T_2,m,\epsilon) \equiv 0$ has a solution $(\bT,m) \mapsto U(\bT,m,\epsilon)$ defined on
$S_{d_1,\theta_1,h'|\epsilon|} \times S_{d_2,\theta_2,h'|\epsilon|} \times \mathbb{R}$ for some $h'>0$ and all
$\epsilon \in D(0,\epsilon_{0}) \setminus \{ 0 \}$. Let $\epsilon \in D(0,\epsilon_{0}) \setminus \{ 0 \}$, then
for $j=1,2$ and all $T_j \in S_{d_j,\theta_j,h'|\epsilon|}$, the function
$m \mapsto U(\bT,m,\epsilon)$ belongs to the space $E_{(\beta',\mu)}$ and for each $m \in \mathbb{R}$, the function
$\bT \mapsto U(\bT,m,\epsilon)$ is bounded and holomorphic on $S_{d_1,\theta_1,h'|\epsilon|}\times S_{d_2,\theta_2,h'|\epsilon|}$. Moreover, $U(\bT,m,\epsilon)$ can be written as a Laplace transform of order $k_1$ in the direction $d_1$ with respect to $T_1$ and the Laplace transform of order $k_2$ in the direction $d_2$ with respect to $T_2$,
\begin{equation}
U(\bT,m,\epsilon) = k_1k_2 \int_{L_{\gamma_1}}\int_{L_{\gamma_2}} \omega_{\bk}^{\bd}(u_1,u_2,m,\epsilon) e^{-(\frac{u_1}{T_1})^{k_1}-(\frac{u_2}{T_2})^{k_2}} \frac{du_2}{u_2}\frac{du_1}{u_1} \label{int_repres_U}
\end{equation} 
where $L_{\gamma_j}=\mathbb{R}_{+}e^{i\gamma_j} \in S_{d_j} \cup \{ 0 \}$, for $j=1,2$ has bisecting direction which might depend on $T_j$. The function $\omega_{\bk}^{\bd}(\btau,m,\epsilon)$ defines a continuous function on
$(\bar{D}(0,\rho) \cup S_{d_1}) \times (\bar{D}(0,\rho) \cup S_{d_2}) \times \mathbb{R} \times D(0,\epsilon_{0}) \setminus \{ 0 \}$, holomorphic with respect to
$(\btau,\epsilon)$ on $(D(0,\rho) \cup S_{d_1}) \times (D(0,\rho) \cup S_{d_2}) \times (D(0,\epsilon_{0}) \setminus \{ 0 \})$. Moreover, there exists a constant $\varpi_{\bd}$ (independent of $\epsilon$) such that
\begin{equation}
|\omega_{\bk}^{\bd}(\btau,m,\epsilon)| \leq \varpi_{\bd}(1+ |m|)^{-\mu} e^{-\beta|m|}
\frac{ |\frac{\tau_1}{\epsilon}|}{1 + |\frac{\tau_1}{\epsilon}|^{2k_1}}\frac{ |\frac{\tau_2}{\epsilon}|}{1 + |\frac{\tau_2}{\epsilon}|^{2k_2}} \exp( \nu_1 |\frac{\tau_1}{\epsilon}|^{k_1}+\nu_2 |\frac{\tau_2}{\epsilon}|^{k_2}) \label{|omega_k_d|<} 
\end{equation}
for all $(\btau) \in (\overline{D}(0,\rho) \cup S_{d_1})\times (\overline{D}(0,\rho) \cup S_{d_2})$, all $m \in \mathbb{R}$, and $\epsilon \in D(0,\epsilon_{0}) \setminus \{ 0 \}$.
\end{prop}
\begin{proof} Let $\epsilon\in D(0,\epsilon_0)\setminus\{0\}$. We take the function $\omega_{\bk}^{\bd}(\btau,m,\epsilon)$ constructed in Proposition~\ref{prop11} and consider $(\tau_1,m)\mapsto \omega_{\bk}^{\bd}(\btau,m,\epsilon)$, which is a function belonging to $F_{(\nu_1,\beta,\mu,k_1,\epsilon)}^{d_1}$, with values in the Banach space of holomorphic functions in $D(0,\rho)\cup S_{d_2}$ with exponential growth of order $k_2$ in $S_{d_2}$, and continuous in $\overline{D}(0,\rho)\times S_{d_2}$. In view of the results stated in Section 3, one can apply Laplace transform of order $m_{k_1}$  following direction $d_1$ in order to obtain a holomorphic and bounded function defined in $S_{d_1,\theta_1, h'|\epsilon|}$, for some $h'>0$. The function
$$(\tau_2,m)\mapsto \mathcal{L}_{m_{k_1}}^{d_1}(\omega_{\bk}^{\bd}(\btau,m,\epsilon) )(T_1)$$
belongs to $F_{(\nu_2,\beta,\mu,k_2,\epsilon)}^{d_2}$, and takes its values in the Banach space of holomorphic and bounded functions in $S_{d_1,\theta_1, h'|\epsilon|}$. One can apply Laplace transform $\mathcal{L}_{m_{k_2}}^{d_2}$ in order to obtain the function $U(\bT,m,\epsilon)$, satisfying the statements above. Moreover, this function is of the form (\ref{int_repres_U}), and preserves holomorphy with respect to the perturbation parameter in $D(0,\epsilon_0)\setminus\{0\}$.

Observe that $U(\bT,m,\epsilon)$ is a solution of equation (\ref{SCP}) due to the properties satisfied by Laplace transform described in (\ref{sum_prod_deriv_m_k_sum}), the construction of $C_0(\bT,m,\epsilon)$ and $F(\bT,m,\epsilon)$ in~\ref{norm_F_varphi_k_psi_k_epsilon_0} and~\ref{norm_F_varphi_k_psi_k_epsilon_012}, and the fact that $\omega_{\bk}^{\bd}(\btau,m,\epsilon)$ is a solution of (\ref{k_Borel_equation}), as stated in Proposition~\ref{prop11}.
 \end{proof}

\section{Analytic solutions of a nonlinear initial value Cauchy problem with complex parameter}

Let $k_1,k_2 \geq 1$ and $D_1,D_2 \geq 2$ be integer numbers. We fix, without loss of generality, that 
$$k_1<k_2,$$
for the roles of such parameters is symmetric. For $j\in\{1,2\}$ and $1 \leq l_j \leq D_j$, let
$d_{l_1},\tilde{d}_{l_2},\delta_{l_1},\tilde{d}_{l_2},\Delta_{l_1,l_2} \geq 0$ be non negative integers.
We assume that 
\begin{equation}
1 = \delta_{1}=\tilde{\delta}_{1} \ \ , \ \ \delta_{l_1} < \delta_{l_1+1} \ \ , \ \ \tilde{\delta}_{l_2} < \tilde{\delta}_{l_2+1} \label{assum_delta_l}
\end{equation}
for all $1 \leq l_1 \leq D_1-1$ and $1 \leq l_2 \leq D_2-1$. We also assume that
\begin{equation}
d_{D_1} = (\delta_{D_1}-1)(k_1+1) \ \ , \ \ d_{l_1} > (\delta_{l_1}-1)(k_1+1) \label{assum_d_delta_Delta1}
\end{equation}
for all $1 \leq l_1 \leq D_1-1$, and
\begin{equation}
\tilde{d}_{D_2} = (\tilde{\delta}_{D_2}-1)(k_2+1) \ \ , \ \ \tilde{d}_{l_2} > (\tilde{\delta}_{l_2}-1)(k_2+1)  \label{assum_d_delta_Delta2}
\end{equation}
for all $1 \leq l_2 \leq D_2-1$. In addition to this, we take 
\begin{equation}\label{e900}
\Delta_{D_1,D_2}=d_{D_1}+\tilde{d}_{D_2}-\delta_{D_1}-\tilde{\delta}_{D_2}+2 \ \ , \ \ \Delta_{D_1,0}=d_{D_1}-\delta_{D_1}+1 \ \ , \ \ \Delta_{0,D_2}=\tilde{d}_{D_2}-\tilde{\delta}_{D_2}+1\\
\end{equation}
Let $Q_{1}(X),Q_{2}(X),R_0(X)\in\C[X]$, and for $0\le l_1\le D_1$ and $0\le l_2\le D_2$ we take $R_{l_1,l_2}(X)\in\C[X]$ such that
$$R_{D_1,l_2}\equiv R_{l_1,D_2}\equiv0,\quad 1\le l_1\le D_1,\quad 1\le l_2\le D_2$$
and such that $R_{D_1,D_2}$ can be factorized in the form $R_{D_1,D_2}(X)=R_{D_1,0}(X)R_{0,D_2}(X)$. We write $R_{D_1}:=R_{D_1,0}$ and $R_{D_2}:=R_{0,D_2}$ for simplicity. Let $P_1,P_2$ be polynomials with coefficients belonging to $\mathcal{O}(\overline{D}(0,\epsilon_0))[X]$, for some $\epsilon_0>0$. We assume that
\begin{equation}\label{raicesgrandes}
\hbox{deg}(Q_j)\ge \hbox{deg}(R_{D_j}),\quad j\in\{1,2\}.
\end{equation} and
\begin{multline}
\mathrm{deg}(Q_j) \geq \mathrm{deg}(R_{D_j}) \ \ , \ \  \mathrm{deg}(R_{D_1,D_2}) \geq \mathrm{deg}(R_{l_1,l_2}) \ \ , \ \
\mathrm{deg}(R_{D_1,D_2}) \geq \mathrm{deg}(P_{j})\\
Q_j(im) \neq 0 \ \ , \ \ R_{D_1,D_2}(im) \neq 0 \label{assum_deg_Q_R}
\end{multline}
for all $m \in \mathbb{R}$, all $j\in\{1,2\}$ and $0 \leq l_j \leq D_j-1$. We denote $\bt:=(t_1,t_2)$.

We consider the following nonlinear initial value problem
\begin{multline}
Q_1(\partial_{z})Q_2(\partial_z)\partial_{t_1}\partial_{t_2}u(\bt,z,\epsilon) = (P_{1}(\partial_{z},\epsilon)u(\bt,z,\epsilon))(P_{2}(\partial_{z},\epsilon)u(\bt,z,\epsilon))\\
+ \sum_{0\le l_1\le D_1,0\le l_2\le D_2} \epsilon^{\Delta_{l_1,l_2}}t_1^{d_{l_1}}\partial_{t_1}^{\delta_{l_1}}t_2^{\tilde{d}_{l_2}}\partial_{t_2}^{\tilde{\delta}_{l_2}}R_{l_1,l_2}(\partial_{z})u(\bt,z,\epsilon)\\
+ c_{0}(\bt,z,\epsilon)R_{0}(\partial_{z})u(\bt,z,\epsilon) + f(\bt,z,\epsilon) \label{ICP_main0}
\end{multline}
for given initial data $u(t_1,0,z,\epsilon)\equiv u(0,t_2,z,\epsilon) \equiv 0$.

The coefficient $c_{0}(\bt,z,\epsilon)$ and the forcing term $f(\bt,z,\epsilon)$ are constructed as follows. We consider families of functions
$m \mapsto C_{n_1,n_2}(m,\epsilon)$, for $n_1,n_2 \geq 0$ and $m \mapsto F_{n_1,n_2}(m,\epsilon)$, for $n_1,n_2 \geq 1$, that belong to the Banach space
$E_{(\beta,\mu)}$ for some $\beta > 0$, $\mu > \max( \mathrm{deg}(P_{1})+1, \mathrm{deg}(P_{2})+1)$ and which
depend holomorphically on $\epsilon \in D(0,\epsilon_{0})$. We assume there exist constants $K_{0},T_{0}>0$
such that (\ref{norm_beta_mu_F_n}) hold for all $n_1,n_2 \geq 1$, for all $\epsilon \in D(0,\epsilon_{0})$.
We deduce that the functions
$$ \mathbf{C}_{0}(\bT,z,\epsilon) = \sum_{n_1,n_2 \geq 0} \mathcal{F}^{-1}(m \mapsto C_{n_1,n_2}(m,\epsilon))(z) T_1^{n_1}T_2^{n_2}$$
$$\mathbf{F}(\bT,z,\epsilon) = \sum_{n_1,n_2 \geq 1} \mathcal{F}^{-1}(m \mapsto F_{n_1,n_2}(m,\epsilon))(z) T_1^{n_1}T_2^{n_2} $$
represent bounded holomorphic functions on $D(0,T_{0}/2)^2 \times H_{\beta'} \times D(0,\epsilon_{0})$ for any
$0 < \beta' < \beta$ (where
$\mathcal{F}^{-1}$ stands for the inverse Fourier transform, see Proposition 9). We define
the coefficient $c_{0}(\bt,z,\epsilon)$ and the forcing term $f(\bt,z,\epsilon)$ as
\begin{equation}
c_{0}(\bt,z,\epsilon) = \mathbf{C}_{0}(\epsilon t_1,\epsilon t_2,z,\epsilon) \ \ , \ \ f(\bt,z,\epsilon) = \mathbf{F}(\epsilon t_1,\epsilon t_2 , z,\epsilon).
\label{defin_c_0_f}
\end{equation}
The functions $c_{0}$ and $f$ are holomorphic and bounded on $D(0,r)^2 \times H_{\beta'} \times D(0,\epsilon_{0})$ where
$r \epsilon_{0} < T_{0}/2$.

We make the additional assumption that there exist unbounded sectors
$$ S_{Q_j,R_{D_j}} = \{ z \in \mathbb{C} / |z| \geq r_{Q_j,R_{D_j}} \ \ , \ \ |\mathrm{arg}(z) - d_{Q_j,R_{D_j}}| \leq \eta_{Q_j,R_{D_j}} \} $$
with direction $d_{Q_j,R_{D_j}} \in \mathbb{R}$, aperture $\eta_{Q_j,R_{D_j}}>0$ for some radius $r_{Q_j,R_{D_j}}>0$ such that
\begin{equation}
\frac{Q_j(im)}{R_{D_j}(im)} \in S_{Q_j,R_{D_j}} \label{assum_Q_R_D}
\end{equation} 
for all $m \in \mathbb{R}$, and for $j=1,2$.

The assumptions made at the beginning of this section allow us to write equation (\ref{ICP_main0}) in the form
\begin{multline}
\left(Q_1(\partial_{z})\partial_{t_1}-\epsilon^{(\delta_{D_1}-1)(k_1+1)-\delta_{D_1}+1}t_1^{(\delta_{D_1}-1)(k_1+1)}\partial_{t_1}^{\delta_{D_1}}R_{D_1}(\partial_z)\right)\\
\times \left(Q_2(\partial_{z})\partial_{t_2}-\epsilon^{(\tilde{\delta}_{D_2}-1)(k_2+1)-\tilde{\delta}_{D_2}+1}t_2^{(\delta_{D_2}-1)(k_2+1)}\partial_{t_2}^{\tilde{\delta}_{D_2}}R_{D_2}(\partial_z)\right) u(\bt,z,\epsilon)\\
 = (P_{1}(\partial_{z},\epsilon)u(\bt,z,\epsilon))(P_{2}(\partial_{z},\epsilon)u(\bt,z,\epsilon))\\
+ \sum_{1\le l_1\le D_1,1\le l_2\le D_2} \epsilon^{\Delta_{l_1,l_2}}t_1^{d_{l_1}}\partial_{t_1}^{\delta_{l_1}}t_2^{\tilde{d}_{l_2}}\partial_{t_2}^{\tilde{\delta}_{l_2}}R_{l_1,l_2}(\partial_{z})u(\bt,z,\epsilon)\\
+ c_{0}(\bt,z,\epsilon)R_{0}(\partial_{z})u(\bt,z,\epsilon) + f(\bt,z,\epsilon) \label{ICP_main}
\end{multline}

We recall the definition of a good covering in $\C^\star$. 

\begin{defin} Let $\varsigma_1,\varsigma_2 \geq 2$ be integer numbers. Let $\{ \mathcal{E}_{p_1,p_2} \}_{\begin{subarray}{l} 0 \leq p_1 \leq \varsigma_1 - 1\\0 \leq p_2 \leq \varsigma_2 - 1\end{subarray}}$ be a finite family of open sectors with vertex at $0$, radius $\epsilon_{0}$ and opening strictly larger than $\frac{\pi}{k_2}$. We assume that the intersection of three different sectors in the good covering is empty, and
$\cup_{\begin{subarray}{l} 0 \leq p_1 \leq \varsigma_1 - 1\\0 \leq p_2 \leq \varsigma_2 - 1\end{subarray}} \mathcal{E}_{p_1,p_2} = \mathcal{U} \setminus \{ 0 \}$,
for some neighborhood of 0, $\mathcal{U}\in\mathbb{C}$. Such set of sectors is called a good covering in $\mathbb{C}^{\ast}$.
\end{defin}

\begin{defin}\label{defgood2} Let $\varsigma_1,\varsigma_2\ge 2$ and $\{ \mathcal{E}_{p_1,p_2} \}_{\begin{subarray}{l} 0 \leq p_1 \leq \varsigma_1 - 1\\0 \leq p_2 \leq \varsigma_2 - 1\end{subarray}}$ be a good covering in $\mathbb{C}^{\ast}$. Let
$\mathcal{T}_j$ be open bounded sectors centered at 0 with radius $r_{\mathcal{T}_j}$ for $j\in\{1,2\}$, and consider two families of open sectors as follows. The first one is given by
$$ S_{\mathfrak{d}_{p_1},\theta_1,\epsilon_{0}r_{\mathcal{T}_1}} =
\{ T_1 \in \mathbb{C}^{\ast} / |T_1| < \epsilon_{0}r_{\mathcal{T}_1} \ \ , \ \ |\mathfrak{d}_{p_1} - \mathrm{arg}(T_1)| < \theta_1/2 \} $$
with opening $\theta_1 > \pi/k_1$, and some $\mathfrak{d}_{p_1} \in \mathbb{R}$, for all $0 \leq p_1 \leq \varsigma_1-1$. This family is chosen to satisfy that:

1) There exists a constant $M_{1}>0$ such that
\begin{equation}
|\tau_1 - q_{l_1}(m)| \geq M_{1}(1 + |\tau_1|) \label{root_cond_1_in_defin1}
\end{equation}
for all $0 \leq l_1 \leq (\delta_{D_1}-1)k_1-1$, $m \in \mathbb{R}$, and $\tau_1 \in S_{\mathfrak{d}_{p_1}} \cup \bar{D}(0,\rho)$, for all
$0 \leq p_1 \leq \varsigma_1-1$, and every root $q_{l_1}$ of the polynomial $P_{m,1}(\tau_1)$.\\
2) There exists a constant $M_{2}>0$ such that
\begin{equation}
|\tau_1 - q_{l_{1,0}}(m)| \geq M_{2}|q_{l_{1,0}}(m)| \label{root_cond_2_in_defin}
\end{equation}
for some root of $P_{m,1}$, $q_{l_{0}}$, all $m \in \mathbb{R}$, $\tau_1 \in S_{\mathfrak{d}_{p_1}} \cup \bar{D}(0,\rho)$, for
all $0 \leq p_1 \leq \varsigma_1 - 1$.

The second family is chosen in an analogous manner. It is given by
$$ S_{\tilde{\mathfrak{d}}_{p_2},\theta_2,\epsilon_{0}r_{\mathcal{T}_2}} =
\{ T_2 \in \mathbb{C}^{\ast} / |T_2| < \epsilon_{0}r_{\mathcal{T}_2} \ \ , \ \ |\tilde{\mathfrak{d}}_{p_2} - \mathrm{arg}(T_2)| < \theta_2/2 \} $$
with opening $\theta_2 > \pi/k_2$, and some $\tilde{\mathfrak{d}}_{p_2} \in \mathbb{R}$, for all $0 \leq p_2 \leq \varsigma_2-1$. This family is chosen to satisfy analogous conditions with respect to the roots of the polynomial $P_{m,2}(\tau_2)$.

In addition to the previous assumptions, we consider $S_{\mathfrak{d}_{p_1},\theta_1,\epsilon_{0}r_{\mathcal{T}_1}}$ and $S_{\tilde{\mathfrak{d}}_{p_2},\theta_2,\epsilon_{0}r_{\mathcal{T}_2}}$ such that for all $0 \leq p_1 \leq \varsigma_1 - 1$, $0 \leq p_2 \leq \varsigma_2 - 1$, $\bt \in \mathcal{T}_1\times \mathcal{T}_2$, and $\epsilon \in \mathcal{E}_{p_1,p_2}$, one has
$$\epsilon t_1 \in S_{\mathfrak{d}_{p_1},\theta_1,\epsilon_{0}r_{\mathcal{T}_1}}\hbox{ and }\epsilon t_2 \in S_{\tilde{\mathfrak{d}}_{p_2},\theta_2,\epsilon_{0}r_{\mathcal{T}_2}}.$$

\noindent We say that the family
$\{ (S_{\mathfrak{d}_{p_1},\theta_1,\epsilon_{0}r_{\mathcal{T}_1}})_{0 \leq p_1 \leq \varsigma_1-1}, (S_{\tilde{\mathfrak{d}}_{p_2},\theta_2,\epsilon_{0}r_{\mathcal{T}_2}})_{0 \leq p_2 \leq \varsigma_2-1} ,\mathcal{T}_1\times \mathcal{T}_2 \}$
is associated to the good covering $\{ \mathcal{E}_{p_1,p_2} \}_{\begin{subarray}{l}0 \leq p_1 \leq \varsigma_1 - 1\\0 \leq p_2 \leq \varsigma_2 - 1\end{subarray}}$.
\end{defin}

The first main result of the present work is devoted to the construction of a family of actual holomorphic solutions to the equation (\ref{ICP_main}) for null initial data. Each of the elements in the family of solutions is associated to an element of a good covering with respect to the complex parameter $\epsilon$. The strategy leans on the control of the difference of two solutions defined in domains with nonempty intersection with respect to the perturbation parameter $\epsilon$. The construction of each analytic solution in terms of two Laplace transforms in different time variables requires to distinguish different cases, depending on the coincidence of the integration paths or not.

\begin{theo}\label{teo1} We consider the equation (\ref{ICP_main}) and we assume that 
(\ref{assum_delta_l}-\ref{assum_deg_Q_R}) and (\ref{assum_Q_R_D}) hold. We also make the additional assumption that
\begin{equation}
\delta_{D_1} \geq \delta_{l_1} + \frac{2}{k_1},\quad \tilde{\delta}_{D_2} \geq \tilde{\delta}_{l_2} + \frac{2}{k_2}, \quad \Delta_{l_1,l_2} + k_1(1 - \delta_{D_1}) +k_2(1 - \tilde{\delta}_{D_2})+ 2 \geq 0,
\label{constraints_k_Borel_equation_for_u_p}
\end{equation}
for all $1 \leq l_1 \leq D_1-1$ and $1 \leq l_2 \leq D_2-1$. Let the coefficient $c_{0}(t_1,t_2,z,\epsilon)$ and forcing term $f(t_1,t_2,z,\epsilon)$ be constructed as in
(\ref{defin_c_0_f}). Let $\{ \mathcal{E}_{p_1,p_2} \}_{\stackrel{0 \leq p_1 \leq \varsigma_1 - 1}{0\le p_2\le \varsigma_2-1}}$ be a good covering in $\mathbb{C}^{\ast}$ such that a family $\{ (S_{\mathfrak{d}_{p_1},\theta_1,\epsilon_{0}r_{\mathcal{T}_1}})_{0 \leq p_1 \leq \varsigma_1-1},(S_{\tilde{\mathfrak{d}}_{p_2},\theta_2,\epsilon_{0}r_{\mathcal{T}_2}})_{0 \leq p_2 \leq \varsigma_2-1},\mathcal{T}_1\times\mathcal{T}_2 \}$
associated to this good covering can be considered.

Then, there exist $r_{Q,R_{D}}>0$, small enough $\epsilon_{0},\zeta_0>0$  such that if
$$ ||C_{0,0}(m,\epsilon)||_{(\beta,\mu)} < \zeta_{0} $$
for all $\epsilon \in D(0,\epsilon_{0}) \setminus \{ 0 \}$, then for every $0 \leq p_1 \leq \varsigma_1-1$ and $0 \leq p_2 \leq \varsigma_2-1$,
one can construct a solution $u_{p_1,p_2}(\bt,z,\epsilon)$ of (\ref{ICP_main}) with $u_{p_1,p_2}(0,t_2,z,\epsilon)\equiv u_{p_1,p_2}(t_1,0,z,\epsilon) \equiv 0$ which  defines a bounded holomorphic function on the domain $(\mathcal{T}_1 \cap D(0,h'))\times (\mathcal{T}_2 \cap D(0,h')) \times H_{\beta'} \times
\mathcal{E}_{p_1,p_2}$ for any given $0< \beta'< \beta$ and for some $h'>0$.

 Moreover, there exist constants $0 < h'' \leq h'$,
$K_{p},M_{p}>0$ (independent of $\epsilon$), and sets $\mathcal{U}_{k_1}\times\mathcal{U}_{k_2}\subseteq\{0,1,\ldots,\varsigma_1-1\}\times \{0,1,\ldots,\varsigma_2-1\}$ such that for every $(p_1,p_2),(p'_1,p'_2)\in \{0,1,\ldots,\varsigma_1-1\}\times\{0,1,\ldots,\varsigma_2-1\}$, one of the following holds:
\begin{itemize}
\item $\mathcal{E}_{p_1,p_2}\cap\mathcal{E}_{p'_1,p'_2}=\emptyset$.
\item $\mathcal{E}_{p_1,p_2}\cap\mathcal{E}_{p'_1,p'_2}\neq\emptyset$ and 
\begin{equation}
\sup_{\bt \in (\mathcal{T}_1 \cap D(0,h''))\times (\mathcal{T}_2 \cap D(0,h'')), z \in H_{\beta'}}
|u_{p_1,p_2}(\bt,z,\epsilon) - u_{p'_1,p'_2}(\bt,z,\epsilon)| \leq K_{p}e^{-\frac{M_p}{|\epsilon|^{k_1}}}
\label{exp_small_difference_u_p11}
\end{equation}
for all $\epsilon \in \mathcal{E}_{p_1,p_2} \cap \mathcal{E}_{p'_1,p'_2}$. In this situation, we say that $\{(p_1,p_2),(p'_1,p'_2)\}$ belong to $\mathcal{U}_{k_1}$.
\item $\mathcal{E}_{p_1,p_2}\cap\mathcal{E}_{p'_1,p'_2}\neq\emptyset$ and 
\begin{equation}
\sup_{\bt \in (\mathcal{T}_1 \cap D(0,h''))\times (\mathcal{T}_2 \cap D(0,h'')), z \in H_{\beta'}}
|u_{p_1,p_2}(\bt,z,\epsilon) - u_{p'_1,p'_2}(\bt,z,\epsilon)| \leq K_{p}e^{-\frac{M_p}{|\epsilon|^{k_2}}}
\label{exp_small_difference_u_p12}
\end{equation}
for all $\epsilon \in \mathcal{E}_{p_1,p_2} \cap \mathcal{E}_{p'_1,p'_2}$. In this situation, we say that $\{(p_1,p_2),(p'_1,p'_2)\}$ belong to $\mathcal{U}_{k_2}$.
\end{itemize} 
\end{theo}
\begin{proof} Regarding Proposition~\ref{prop12}, one can choose $r_{Q_1,R_{D_1}}>0$, and small enough $\epsilon_{0},\zeta_0>0$ such that
$$ \left\|C_{0,0}(m,\epsilon)\right\|_{(\beta,\mu)} \leq \zeta_{0} $$
for all $\epsilon \in D(0,\epsilon_{0}) \setminus \{ 0 \}$. For each pair $(p_1,p_2)$, we fix the multidirection $(\mathfrak{d}_{p_1},\tilde{\mathfrak{d}}_{p_2})$ with $0 \leq p_j \leq \varsigma_j - 1$ and construct $U^{\mathfrak{d}_{p_1},\tilde{\mathfrak{d}}_{p_2}}(\bT,m,\epsilon)$ such that
$U^{\mathfrak{d}_{p_1},\tilde{\mathfrak{d}}_{p_2}}(0,T_2,m,\epsilon)\equiv U^{\mathfrak{d}_{p_1},\tilde{\mathfrak{d}}_{p_2}}(T_1,0,m,\epsilon) \equiv 0$ and is a solution of
\begin{multline}
\left(Q_1(im)\partial_{T_1}-T_1^{(\delta_{D_1}-1)(k_1-1)}\partial_{T_1}^{\delta_{D_1}}R_{D_1}(im)\right)\left(Q_2(im)\partial_{T_2}-T_2^{(\tilde{\delta}_{D_2}-1)(k_2-1)}\partial_{T_2}^{\tilde{\delta}_{D_2}}R_{D_2}(im)\right)U(\bT,m,\epsilon)\\
=\epsilon^{-2}\frac{1}{(2\pi)^{1/2}}\int_{-\infty}^{+\infty}P_{1}(i(m-m_{1}),\epsilon)U(\bT,m-m_{1},\epsilon)P_{2}(im_{1},\epsilon)U(\bT,m_{1},\epsilon) dm_{1}\\
+ \sum_{1\le l_1\le D_1-1,1\le l_2\le D_2-1} \epsilon^{\Delta_{l_1,l_2}-d_{l_1}-d_{l_2}+\delta_{l_1}+\tilde{d}_{l_2}- 2} T_1^{d_{l_1}}T_2^{\tilde{d}_{l_2}} \partial_{T_1}^{\delta_{l_1}}\partial_{T_2}^{\tilde{\delta}_{l_2}}R_{\ell_1,\ell_2}(im)U(\bT,m,\epsilon)\\
+ \epsilon^{-2}\frac{1}{(2\pi)^{1/2}}\int_{-\infty}^{+\infty}C_{0}(\bT,m-m_{1},\epsilon)R_{0}(im_{1})U(\bT,m_{1},\epsilon) dm_{1}\\
+\epsilon^{-2}F(\bT,m,\epsilon),\label{SCP_2}
\end{multline}
where
$$ C_{0}(\bT,m,\epsilon) = \sum_{n_1,n_2 \geq 1} C_{0,n_1,n_2}(m,\epsilon) T_1^{n_1}T_2^{n_2} \ \ , \ \ F(\bT,m,\epsilon) = \sum_{n_1,n_2 \geq 1} F_{n_1,n_2}(m,\epsilon) T_1^{n_1}T_2^{n_2} $$
are convergent series in $D(0,T_{0}/2)^2$ with values in $E_{(\beta,\mu)}$, for all $\epsilon \in D(0,\epsilon_{0}) \setminus \{ 0 \}$. The
function $(\bT,m) \mapsto U^{\mathfrak{d}_{p_1},\tilde{\mathfrak{d}}_{p_2}}(\bT,m,\epsilon)$ is well defined on
$S_{\mathfrak{d}_{p_1},\theta_1,h'|\epsilon|}\times S_{\tilde{\mathfrak{d}}_{p_2},\theta_2,h'|\epsilon|}\times \mathbb{R}$
where $h'>0$, for all $\epsilon \in D(0,\epsilon_{0}) \setminus \{ 0 \}$. Moreover,
$U^{\mathfrak{d}_{p_1},\tilde{\mathfrak{d}}_{p_2}}(\bT,m,\epsilon)$ can be written as the iterated Laplace transform of order $k_1$ in the direction $\mathfrak{d}_{p_1}$, and the Laplace transform of order $k_2$ in the direction $\tilde{\mathfrak{d}}_{p_2}$ 
\begin{equation}
U^{\mathfrak{d}_{p_1},\tilde{\mathfrak{d}}_{p_2}}(\bT,m,\epsilon) = k_1k_2 \int_{L_{\gamma_{p_1}}}\int_{L_{\gamma_{p_2}}} \omega_{\bk}^{\mathfrak{d}_{p_1},\tilde{\mathfrak{d}}_{p_2}}(u_1,u_2,m,\epsilon)
e^{-(\frac{u_1}{T_1})^{k_1}-(\frac{u_2}{T_2})^{k_2}} \frac{du_2}{u_2}\frac{du_1}{u_1} \label{int_repres_U_d_p}
\end{equation} 
along $L_{\gamma_{p_j}}=\mathbb{R}_{+}e^{i\gamma_{p_j}}$ which might depend on $T_j$. Here, $\omega_{\bk}^{\mathfrak{d}_{p_1},\tilde{\mathfrak{d}}_{p_2}}(\btau,m,\epsilon)$ defines a continuous function on
$(\bar{D}(0,\rho) \cup S_{d_{p_1}})\times (\bar{D}(0,\rho) \cup S_{d_{p_2}}) \times \mathbb{R} \times D(0,\epsilon_{0}) \setminus \{ 0 \}$,  holomorphic with respect to
$(\btau,\epsilon)$ on $(D(0,\rho) \cup S_{\mathfrak{d}_{p_1}})\times (D(0,\rho) \cup S_{\tilde{\mathfrak{d}}_{p_2}})\times  (D(0,\epsilon_{0}) \setminus \{ 0 \})$ for all $m \in \mathbb{R}$. Moreover, there exists
a constant $\varpi_{\mathfrak{d}_{p_1},\tilde{\mathfrak{d}}_{p_2}}$ (independent of $\epsilon$) such that
\begin{equation}
|\omega_{\bk}^{\mathfrak{d}_{p_1},\tilde{\mathfrak{d}}_{p_2}}(\btau,m,\epsilon)| \leq \varpi_{\mathfrak{d}_{p_1},\tilde{\mathfrak{d}}_{p_2}}(1+ |m|)^{-\mu} e^{-\beta|m|}
\frac{ |\frac{\tau_1}{\epsilon}|}{1 + |\frac{\tau_1}{\epsilon}|^{2k_1}}\frac{ |\frac{\tau_2}{\epsilon}|}{1 + |\frac{\tau_2}{\epsilon}|^{2k_2}} \exp( \nu_1 |\frac{\tau_1}{\epsilon}|^{k_1}+\nu_2 |\frac{\tau_2}{\epsilon}|^{k_2}) \label{|omega_k_d_p|<} 
\end{equation}
for all $\btau \in (D(0,\rho) \cup S_{\mathfrak{d}_{p_1}})\times (D(0,\rho) \cup S_{\tilde{\mathfrak{d}}_{p_2}}) $, all $m \in \mathbb{R}$ and
$\epsilon \in D(0,\epsilon_{0}) \setminus \{ 0 \}$.  The function
$$ (\bT,z) \mapsto \mathbf{U}^{\mathfrak{d}_{p_1},\tilde{\mathfrak{d}}_{p_2}}(\bT,z,\epsilon) = \mathcal{F}^{-1}(m \mapsto U^{\mathfrak{d}_{p_1},\tilde{\mathfrak{d}}_{p_2}}(\bT,m,\epsilon))(z) $$
turns out to be holomorphic on $S_{\mathfrak{d}_{p_1},\theta_1,h'|\epsilon|}\times S_{\tilde{\mathfrak{d}}_{p_2},\theta_2,h'|\epsilon|} \times H_{\beta'}$, for all
$\epsilon \in D(0,\epsilon_{0}) \setminus \{ 0 \}$ and $0 < \beta' < \beta$. For all $0 \leq p_j \leq \varsigma_j - 1$, $j\in\{1,2\}$ let
\begin{multline*}
u_{p_1,p_2}(\bt,z,\epsilon) = \mathbf{U}^{\mathfrak{d}_{p_1},\tilde{\mathfrak{d}}_{p_2}}(\epsilon t_1,\epsilon t_2,z,\epsilon)\\
= \frac{k_1k_2}{(2\pi)^{1/2}}\int_{-\infty}^{+\infty}
\int_{L_{\gamma_{p_1}}}\int_{L_{\gamma_{p_2}}}
\omega_{\bk}^{\mathfrak{d}_{p_1},\tilde{\mathfrak{d}}_{p_2}}(u_1,u_2,m,\epsilon) e^{-(\frac{u_1}{\epsilon t_1})^{k_1}-(\frac{u_2}{\epsilon t_2})^{k_2}} e^{izm} \frac{du_2}{u_2}\frac{du_1}{u_1} dm.
\end{multline*}
By construction (see Definition~\ref{defgood2}), the function $u_{p_1,p_2}(\bt,z,\epsilon)$ defines a bounded holomorphic function on
 $(\mathcal{T}_1 \cap D(0,h')) \times (\mathcal{T}_2 \cap D(0,h'))\times H_{\beta'} \times \mathcal{E}_{p_1,p_2}$. Moreover,
$u_{p_1,p_2}(0,t_2,z,\epsilon) \equiv u_{p_1,p_2}(t_1,0,z,\epsilon)\equiv 0$. Moreover, the properties of inverse Fourier transform described in Proposition~\ref{prop359} guarantee that
$u_{p_1,p_2}(\bt,z,\epsilon)$ is a solution of the main problem under study (\ref{ICP_main}) on
$(\mathcal{T}_1 \cap D(0,h'))\times (\mathcal{T}_2 \cap D(0,h')) \times H_{\beta'} \times \mathcal{E}_{p_1,p_2}$.

It is worth mentioning that all the functions $\btau \mapsto \omega_{\bk}^{\mathfrak{d}_{p_1},\tilde{\mathfrak{d}}_{p_2}}(\btau,m,\epsilon)$ provide the analytic continuation of a common function 
$$\btau\mapsto  \omega_{\bk}(\btau,m,\epsilon) = \sum_{n_1 \geq 1,n_2\ge 1} U_{n_1,n_2}(m,\epsilon) \frac{\tau_1^{n_1}}{\Gamma(\frac{n_1}{k_1})} \frac{\tau_2^{n_2}}{\Gamma(\frac{n_2}{k_2})}\in\mathcal{O}(D(0,\rho)^2,E_{(\beta,\mu)}) $$
to $S_{\mathfrak{d}_{p_1}}\times S_{\tilde{\mathfrak{d}}_{p_2}}$. Observe that $U_{n_1,n_2}(m,\epsilon) \in E_{(\beta,\mu)}$ are the coefficients of the formal solution of the equation (\ref{SCP_2}), for
all $\epsilon \in D(0,\epsilon_{0}) \setminus \{ 0 \}$, $\hat{U}(T_1,T_2,m,\epsilon) = \sum_{n_1 \geq 1,n_2\ge 1} U_{n_1,n_2}(m,\epsilon)T_1^{n_1}T_2^{n_2}$.

The proof of the estimates (\ref{exp_small_difference_u_p11}) and (\ref{exp_small_difference_u_p12}) leans on those in the proof of Theorem 1 in~\cite{lama}. In the present situation, different digressions are considered, due to the presence of two time variables. Let $p_j,p'_j \in \{ 0,\ldots,\varsigma_j - 1 \}$ for $j\in\{1,2\}$, and assume that $\mathcal{E}_{p_1,p_2}\cap\mathcal{E}_{p'_1,p'_2}\neq\emptyset$. Then, three different cases should be considered:

\textbf{Case 1:} Assume that the path $L_{\gamma_{p_1}}$ coincides with $L_{\gamma_{p'_1}}$, and $L_{\gamma_{p_2}}$ does not coincide with $L_{\gamma_{p'_2}}$. Then, using that
 $u_2 \mapsto \omega_{\bk}^{\mathfrak{d}_{p_1},\tilde{\mathfrak{d}}_{p_2}}(u_1,u_2,m,\epsilon) \exp( -(\frac{u_2}{\epsilon t_2})^{k_2} )/u_2$ is holomorphic on $D(0,\rho)$ for all
$(m,\epsilon) \in \mathbb{R} \times (D(0,\epsilon_{0}) \setminus \{ 0 \})$, and every $u_1\in L_{\gamma_{p_1}}$, one can deform one of the integration paths to write
$$I=\int_{L_{\gamma_{p_2}}}\omega_{\bk}^{\mathfrak{d}_{p_1},\tilde{\mathfrak{d}}_{p_2}}(u_1,u_2,m,\epsilon)e^{-\left(\frac{u_2}{\epsilon t_2}\right)^{k_2}}\frac{du_2}{u_2}-\int_{L_{\gamma_{p'_2}}}\omega_{\bk}^{\mathfrak{d}_{p_1},\tilde{\mathfrak{d}}_{p_2}}(u_1,u_2,m,\epsilon)e^{-\left(\frac{u_2}{\epsilon t_2}\right)^{k_2}}\frac{du_2}{u_2}$$
in the form
\begin{multline}
\int_{L_{\rho/2,\gamma_{p_2}}}
\omega_{\bk}^{\mathfrak{d}_{p_1},\tilde{\mathfrak{d}}_{p_2}}(u_1,u_2,m,\epsilon) e^{-(\frac{u_2}{\epsilon t_2})^{k_2}} \frac{du_2}{u_2} \\ 
-\int_{L_{\rho/2,\gamma_{p'_2}}}
\omega_{\bk}^{\mathfrak{d}_{p_1},\tilde{\mathfrak{d}}_{p_2}}(u_1,u_2,m,\epsilon) e^{-(\frac{u_2}{\epsilon t_2})^{k_2}} \frac{du_2}{u_2}\\
+ \int_{C_{\rho/2,\gamma_{p'_2},\gamma_{p_2}}}
\omega_{\bk}^{\mathfrak{d}_{p_1},\tilde{\mathfrak{d}}_{p_2}}(u_1,u_2,m,\epsilon) e^{-(\frac{u_2}{\epsilon t_2})^{k_2}} \frac{du_2}{u_2}. \label{difference_u_p_decomposition}
\end{multline}
where $L_{\rho/2,\gamma_{p_2}} = [\rho/2,+\infty)e^{i\gamma_{p_2}}$,
$L_{\rho/2,\gamma_{p'_2}} = [\rho/2,+\infty)e^{i\gamma_{p'_2}}$ and
$C_{\rho/2,\gamma_{p'_2},\gamma_{p_2}}$ is an arc of circle connecting
$(\rho/2)e^{i\gamma_{p'_2}}$ and $(\rho/2)e^{i\gamma_{p_2}}$ with the adequate orientation.\medskip

The estimates for the previous expression can be found in detail in the proof of Theorem 1,~\cite{lama}. 



Namely, we get the existence of constants $C_{p_2,p'_2},M_{p_2,p'_2}>0$ such that
$$|I|\le C_{p_2,p'_2}\varpi_{\mathfrak{d}_{p_1},\tilde{\mathfrak{d}}_{p_2}}(1+ |m|)^{-\mu} e^{-\beta|m|} \frac{ |\frac{u_1}{\epsilon}|}{1 + |\frac{u_1}{\epsilon}|^{2k_1}}
\exp( \nu_1 |\frac{u_1}{\epsilon}|^{k_1})e^{-\frac{M_{p_2,p'_2}}{|\epsilon|^{k_2}}},$$
for $t_2\in\mathcal{T}_2\cap D(0,h')$ and $\epsilon\in\mathcal{E}_{p_1,p_2}\cap\mathcal{E}_{p'_1,p'_2}$ and $u_1\in L_{\gamma_{p_1}}$. We have
\begin{multline}
|u_{p_1,p_2}(\bt,z,\epsilon)-u_{p'_1,p'_2}(\bt,z,\epsilon)|\\
\le \frac{k_1k_2}{(2\pi)^{1/2}}C_{p_2,p'_2}\left(\int_{-\infty}^{\infty}(1+|m|)^{-\mu}e^{-\beta|m|}e^{-m|\hbox{Im}(z)|}dm\right)\\
\times \int_{L_{\gamma_{p_1}}}\frac{ |\frac{u_1}{\epsilon}|}{1 + |\frac{u_1}{\epsilon}|^{2k_1}}
\exp( \nu_1 |\frac{u_1}{\epsilon}|^{k_1})\exp(-\left(\frac{u_1}{\epsilon t_1}\right)^{k_1})\left|\frac{du_1}{u_1}\right|    e^{-\frac{M_{p_2,p'_2}}{|\epsilon|^{k_2}}}.
\end{multline}
The last integral is estimated via the reparametrization $u_1=re^{\gamma_{p_1}\sqrt{-1}}$ and the change of variable $r=|\epsilon|s$ by 
$$\int_0^{\infty}\frac{1}{1+s^2}e^{-\delta_{1}s^{k_1}}ds,$$
for some $\delta_{1}>0$, whenever $t_1\in\mathcal{T}_1\cap D(0,h')$.

From the fact that $z\in H_{\beta'}$, we get that $\{(p_1,p_2),(p'_1,p'_2)\}$ belong to $\mathcal{U}_{k_2}$.\medskip

\begin{figure}
	\centering
	\includegraphics[width=0.4\textwidth]{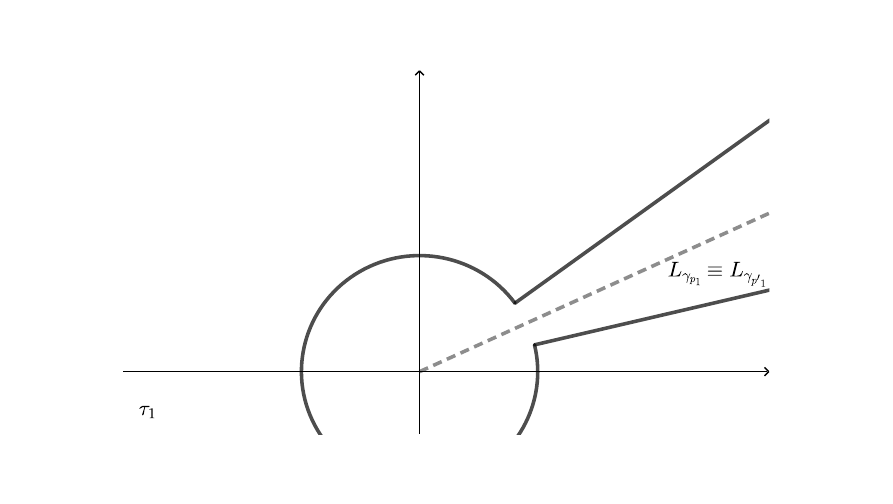}
	\includegraphics[width=0.4\textwidth]{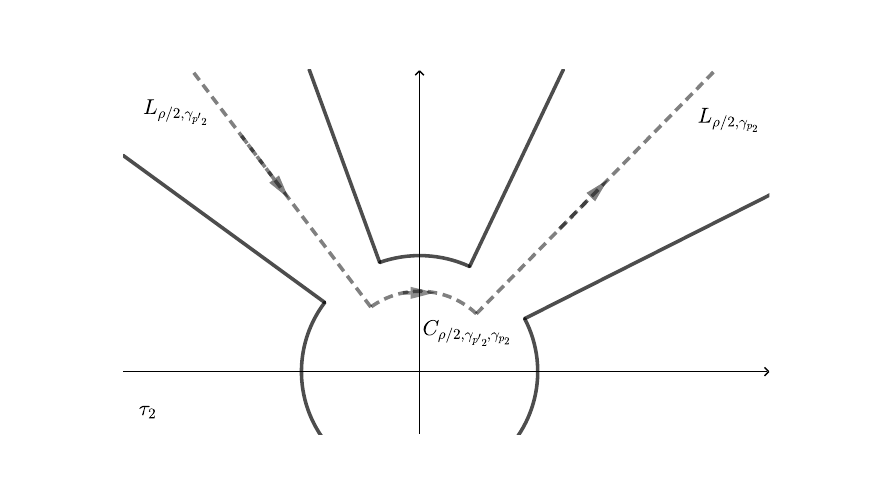}
	\caption{Path deformation in Case 1}
\end{figure}

\textbf{Case 2}: The path $L_{\gamma_{p_2}}$ coincides with $L_{\gamma_{p'_2}}$, and $L_{\gamma_{p_1}}$ does not coincide with $L_{\gamma_{p'_1}}$. It can be handled analogously as Case 1. We get that the set $\{(p_1,p_2),(p'_1,p'_2)\}$ belongs to $\mathcal{U}_{k_1}$. More precisely, we arrive at the expression
\begin{multline*}
|u_{p_1,p_2}(\bt,z,\epsilon)-u_{p'_1,p'_2}(\bt,z,\epsilon)|\\
\le \frac{k_1k_2}{(2\pi)^{1/2}}C_{p_1,p'_1}\left(\int_{-\infty}^{\infty}(1+|m|)^{-\mu}e^{-\beta|m|}e^{-m|\hbox{Im}(z)|}dm\right)\\
\times \int_{L_{\gamma_{p_2}}}\frac{ |\frac{u_2}{\epsilon}|}{1 + |\frac{u_2}{\epsilon}|^{2k_2}}
\exp( \nu_2 |\frac{u_2}{\epsilon}|^{k_2})\exp(-\left(\frac{u_2}{\epsilon t_2}\right)^{k_2})\left|\frac{du_2}{u_2}\right|    e^{-\frac{M_{p_1,p'_1}}{|\epsilon|^{k_1}}}.
\end{multline*}

\textbf{Case 3:} Assume that neither $L_{\gamma_{p_1}}$ coincides with $L_{\gamma_{p'_1}}$, nor $L_{\gamma_{p_2}}$ coincides with $L_{\gamma_{p'_2}}$. 

We deform the integration paths with respect to the first time variable and write

$$u_{p_1,p_2}(\bt,z,\epsilon)-u_{p'_1,p'_2}(\bt,z,\epsilon)=J_1-J_2+J_3,$$
where
$$J_1=\frac{k_1k_2}{(2\pi)^{1/2}} \int_{L_{\gamma_{p_1},1}}\int_{L_{\gamma_{p_2}}}\int_{-\infty}^{\infty} \omega_{\bk}^{\mathfrak{d}_{p_1},\tilde{\mathfrak{d}}_{p_2}}(u_1,u_2,m,\epsilon)
e^{-(\frac{u_1}{\epsilon t_1})^{k_1}-(\frac{u_2}{\epsilon t_2})^{k_2}}e^{izm} dm\frac{du_2}{u_2}\frac{du_1}{u_1}.$$

$$J_2=\frac{k_1k_2}{(2\pi)^{1/2}} \int_{L_{\gamma_{p'_1},1}}\int_{L_{\gamma_{p'_2}}}\int_{-\infty}^{\infty} \omega_{\bk}^{\mathfrak{d}_{p'_1},\tilde{\mathfrak{d}}_{p'_2}}(u_1,u_2,m,\epsilon)
e^{-(\frac{u_1}{\epsilon t_1})^{k_1}-(\frac{u_2}{\epsilon t_2})^{k_2}}e^{izm} dm\frac{du_2}{u_2}\frac{du_1}{u_1}.$$

\begin{multline*}
J_3=\frac{k_1k_2}{(2\pi)^{1/2}} \int_{0}^{\frac{\rho}{2}e^{i\theta}}\left(\int_{-\infty}^{\infty}\left(\int_{L_{\gamma_{p_2}}} \omega_{\bk}^{\mathfrak{d}_{p_1},\tilde{\mathfrak{d}}_{p_2}}(u_1,u_2,m,\epsilon)
e^{-(\frac{u_2}{\epsilon t_2})^{k_2}} \frac{du_2}{u_2}\right.\right.\\
\left.\left.-\int_{L_{\gamma_{p'_2}}} \omega_{\bk}^{\mathfrak{d}_{p'_1},\tilde{\mathfrak{d}}_{p'_2}}(u_1,u_2,m,\epsilon)
e^{-(\frac{u_2}{\epsilon t_2})^{k_2}} \frac{du_2}{u_2}\right)e^{izm}dm\right)e^{-(\frac{u_1}{\epsilon t_1})^{k_1}}\frac{du_1}{u_1}, 
\end{multline*}
where $\frac{\rho}{2}e^{i\theta}$ is such that $\theta$ is an argument between $\gamma_{p_1}$ and $\gamma_{p'_1}$. The path $L_{\gamma_{p_1},1}$ (resp. $L_{\gamma_{p'_1},1}$) consists of the concatenation of the arc of circle connecting $\frac{\rho}{2}e^{i\theta}$ with $\frac{\rho}{2}e^{i\gamma_{p_1}}$ (resp. with $\frac{\rho}{2}e^{i\gamma_{p'_1}}$) and the half line $[\frac{\rho}{2}e^{i\gamma_{p_1}},\infty)$ (resp. $[\frac{\rho}{2}e^{i\gamma_{p'_1}},\infty)$).

\begin{figure}
	\centering
	\includegraphics[width=0.3\textwidth]{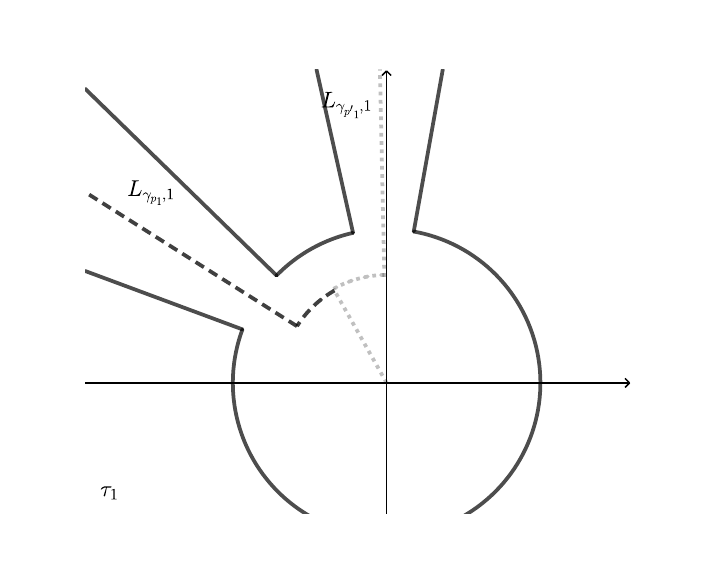} \vline\vline \includegraphics[width=0.3\textwidth]{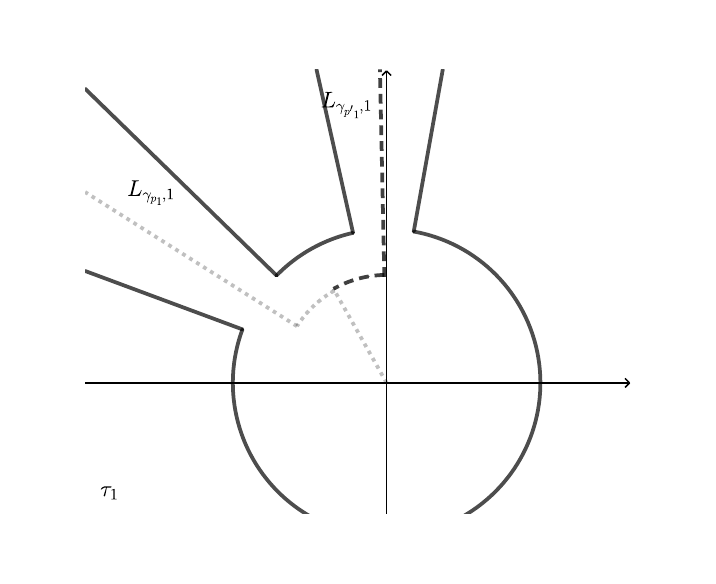} \vline\vline \includegraphics[width=0.3\textwidth]{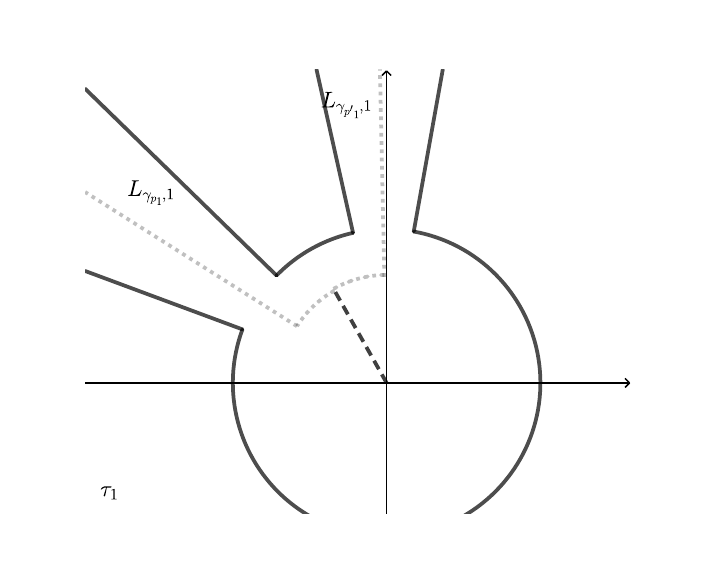}\\
	\includegraphics[width=0.3\textwidth]{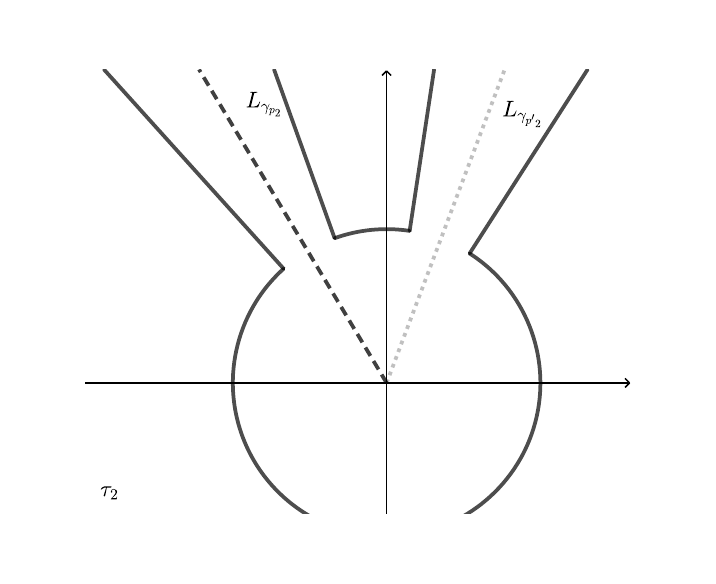} \vline\vline \includegraphics[width=0.3\textwidth]{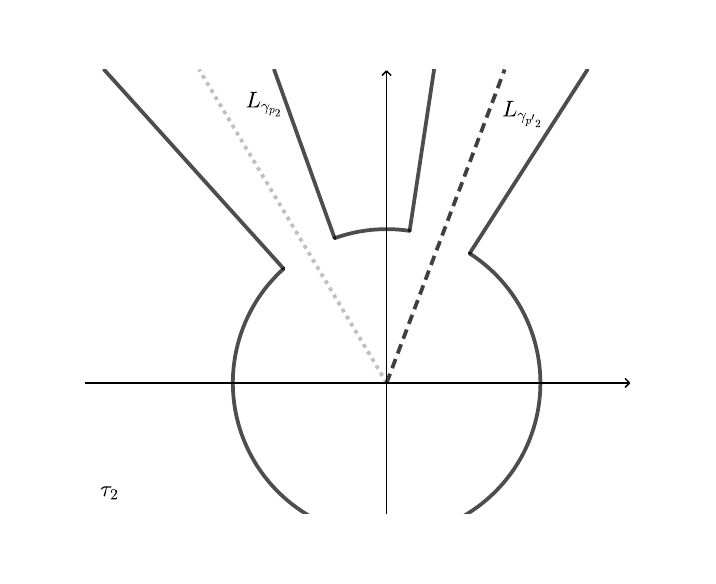} \vline\vline \includegraphics[width=0.3\textwidth]{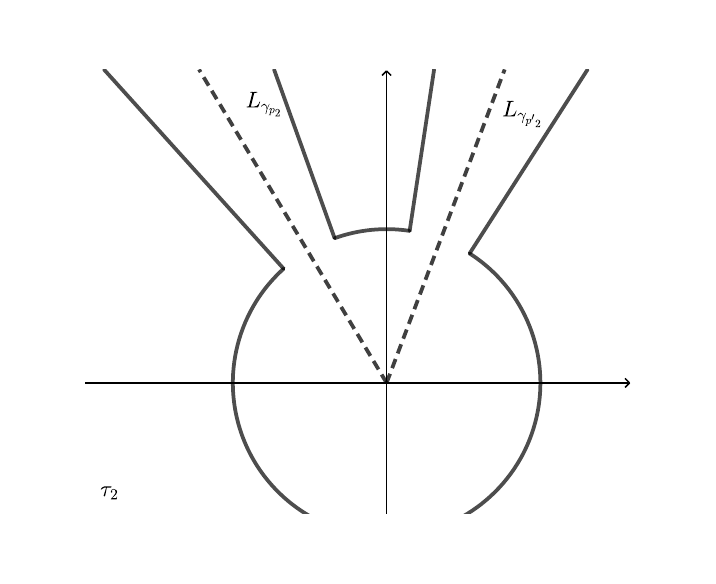}
	\caption{Path deformation in Case 3}
\end{figure}

We first give estimates for $|J_1|$. We have
\begin{multline*}
\left|\int_{L_{\gamma_{p_2}}} \omega_{\bk}^{\mathfrak{d}_{p_1},\tilde{\mathfrak{d}}_{p_2}}(u_1,u_2,m,\epsilon)
e^{-(\frac{u_2}{\epsilon t_2})^{k_2}}\frac{du_2}{u_2}\right|\le \varpi_{\mathfrak{d}_{p_1},\tilde{\mathfrak{d}}_{p_2}}(1+ |m|)^{-\mu} e^{-\beta|m|} \frac{ |\frac{u_1}{\epsilon}|}{1 + |\frac{u_1}{\epsilon}|^{2k_1}}
\exp( \nu_1 |\frac{u_1}{\epsilon}|^{k_1})\\
\times \int_{L_{\gamma_{p_2}}} \left(\frac{ |\frac{u_2}{\epsilon}|}{1 + |\frac{u_2}{\epsilon}|^{2k_2}}\exp(\nu_2\left|\frac{u_2}{\epsilon}\right|^{k_2}\right)|e^{-\left(\frac{u_2}{\epsilon t_2}\right)^{k_2}}|\left|\frac{du_2}{u_2}\right|\\
\le \varpi_{\mathfrak{d}_{p_1},\tilde{\mathfrak{d}}_{p_2}}C_{p_2}(1+ |m|)^{-\mu} e^{-\beta|m|} \frac{ |\frac{u_1}{\epsilon}|}{1 + |\frac{u_1}{\epsilon}|^{2k_1}}
\exp( \nu_1 |\frac{u_1}{\epsilon}|^{k_1}),
\end{multline*}
for some $C_{p_2}>0$, and $t_2\in\mathcal{T}_2\cap D(0,h')$. Using the parametrization $u_2=re^{\gamma_{p_2}\sqrt{-1}}$ and the change of variable $r=|\epsilon|s$. Using analogous estimations as in the Case 1, we arrive at
$$|J_1|\le C_{p,1}e^{-\frac{M_{p,1}}{|\epsilon|^{k_1}}},$$
for some $C_{p,1},M_{p,1}>0$, for all $\epsilon\in\mathcal{E}_{p_1,p_2}\cap \mathcal{E}_{p'_1,p'_2}$, where $t_1\in\mathcal{T}_{1}\cap D(0,h')$ and $t_2\in\mathcal{T}_{2}\cap D(0,h')$, $z\in H_{\beta'}$.

Analogous calculations yield to 
$$|J_2|\le C_{p,2}e^{-\frac{M_{p,2}}{|\epsilon|^{k_1}}},$$
for some $C_{p,2},M_{p,2}>0$, for all $\epsilon\in\mathcal{E}_{p_1,p_2}\cap \mathcal{E}_{p'_1,p'_2}$, where $t_1\in\mathcal{T}_{1}\cap D(0,h')$ and $t_2\in\mathcal{T}_{2}\cap D(0,h')$, $z\in H_{\beta'}$.

In order to give upper bounds for $|J_3|$, we consider
$$\left|\int_{L_{\gamma_{p_2}}} \omega_{\bk}^{\mathfrak{d}_{p_1},\tilde{\mathfrak{d}}_{p_2}}(u_1,u_2,m,\epsilon)
e^{-(\frac{u_2}{\epsilon t_2})^{k_2}} \frac{du_2}{u_2}-\int_{L_{\gamma_{p'_2}}} \omega_{\bk}^{\mathfrak{d}_{p'_1},\tilde{\mathfrak{d}}_{p'_2}}(u_1,u_2,m,\epsilon)
e^{-(\frac{u_2}{\epsilon t_2})^{k_2}} \frac{du_2}{u_2}\right|.$$
We choose a deformation path in the form of that considered in Case 1. We get the previous expression is upper estimated by
$$\varpi_{\mathfrak{d}_{p_1},\tilde{\mathfrak{d}}_{p_2}}C_{p_2,p'_2}(1+ |m|)^{-\mu} e^{-\beta|m|} \frac{ |\frac{u_1}{\epsilon}|}{1 + |\frac{u_1}{\epsilon}|^{2k_1}}
\exp( \nu_1 |\frac{u_1}{\epsilon}|^{k_1})\exp\left(-\frac{M_{p_2,p'_2}}{|\epsilon|^{k_2}}\right),$$
for $\epsilon\in\mathcal{E}_{p_1,p_2}\cap \mathcal{E}_{p'_1,p'_2}$,  $t_2\in\mathcal{T}_{2}\cap D(0,h')$, $u_1\in[0,\rho/2e^{i\theta}]$.
We finally get 
\begin{multline*}
|J_3|\le\frac{k_1k_2}{(2\pi)^{1/2}} C_{p_2,p'_2}\varpi_{\mathfrak{d}_{p_1},\tilde{\mathfrak{d}}_{p_2}}    \left(\int_{-\infty}^{\infty}(1+|m|)^{-\mu}e^{-\beta|m|}e^{-m|\hbox{Im}(z)|}dm\right)\\
\times\left(\int_{0}^{\rho/2e^{i\theta}}\frac{ |\frac{u_1}{\epsilon}|}{1 + |\frac{u_1}{\epsilon}|^{2k_1}}\exp(\nu_1|\frac{u_1}{\epsilon}|^{k_1})|e^{-\left(\frac{u_1}{\epsilon t_1}\right)^{k_1}}|\left|\frac{d u_1}{u_1}\right|\right)\exp\left(-\frac{M_{p_2,p'_2}}{|\epsilon|^{k_2}}\right).
\end{multline*}
We conclude that 
$$|J_3|\le K_{p,3}e^{-\frac{M_{p,3}}{|\epsilon|^{k_2}}},$$
uniformly for $(t_1,t_2)\in (\mathcal{T}_1\cap D(0,h''))\times (\mathcal{T}_2\cap D(0,h''))$ for some $h''>0$, and $z\in H_{\beta'}$ for any fixed $\beta'<\beta$, where $K_{\boldsymbol{p},3},M_{\boldsymbol{p},3}$ are positive constants.
\end{proof}

\section{Asymptotics of the problem in the perturbation parameter}

\subsection{$k-$Summable formal series and Ramis-Sibuya Theorem}

For the sake of completeness, we recall the definition of $k-$Borel summability of formal series with coefficients in a Banach space, and Ramis-Sibuya Theorem. A reference for the details on the first part is~\cite{ba}, whilst the second part of this section can be found in~\cite{ba2}, p. 121, and~\cite{hssi}, Lemma XI-2-6.
 
\begin{defin} Let $k \geq 1$ be an integer. A formal series
$$\hat{X}(\epsilon) = \sum_{j=0}^{\infty}  \frac{ a_{j} }{ j! } \epsilon^{j} \in \mathbb{F}[[\epsilon]]$$
with coefficients in a Banach space $( \mathbb{F}, ||.||_{\mathbb{F}} )$ is said to be $k-$summable
with respect to $\epsilon$ in the direction $d \in \mathbb{R}$ if \medskip

{\bf i)} there exists $\rho \in \mathbb{R}_{+}$ such that the following formal series, called formal
Borel transform of $\hat{X}$ of order $k$ 
$$ \mathcal{B}_{k}(\hat{X})(\tau) = \sum_{j=0}^{\infty} \frac{ a_{j} \tau^{j}  }{ j!\Gamma(1 + \frac{j}{k}) } \in \mathbb{F}[[\tau]],$$
is absolutely convergent for $|\tau| < \rho$, \medskip

{\bf ii)} there exists $\delta > 0$ such that the series $\mathcal{B}_{k}(\hat{X})(\tau)$ can be analytically continued with
respect to $\tau$ in a sector
$S_{d,\delta} = \{ \tau \in \mathbb{C}^{\ast} : |d - \mathrm{arg}(\tau) | < \delta \} $. Moreover, there exist $C >0$, and $K >0$
such that
$$ ||\mathcal{B}(\hat{X})(\tau)||_{\mathbb{F}}
\leq C e^{ K|\tau|^{k}} $$
for all $\tau \in S_{d, \delta}$.
\end{defin}
If this is so, the vector valued Laplace transform of order $k$ of $\mathcal{B}_{k}(\hat{X})(\tau)$ in the direction $d$ is defined by
$$ \mathcal{L}^{d}_{k}(\mathcal{B}_{k}(\hat{X}))(\epsilon) = \epsilon^{-k} \int_{L_{\gamma}}
\mathcal{B}_{k}(\hat{X})(u) e^{ - ( u/\epsilon )^{k} } ku^{k-1}du,$$
along a half-line $L_{\gamma} = \mathbb{R}_{+}e^{i\gamma} \subset S_{d,\delta} \cup \{ 0 \}$, where $\gamma$ depends on
$\epsilon$ and is chosen in such a way that
$\cos(k(\gamma - \mathrm{arg}(\epsilon))) \geq \delta_{1} > 0$, for some fixed $\delta_{1}$, for all
$\epsilon$ in a sector
$$ S_{d,\theta,R^{1/k}} = \{ \epsilon \in \mathbb{C}^{\ast} : |\epsilon| < R^{1/k} \ \ , \ \ |d - \mathrm{arg}(\epsilon) |
< \theta/2 \},$$
where $\frac{\pi}{k} < \theta < \frac{\pi}{k} + 2\delta$ and $0 < R < \delta_{1}/K$. The
function $\mathcal{L}^{d}_{k}(\mathcal{B}_{k}(\hat{X}))(\epsilon)$
is called the $k-$sum of the formal series $\hat{X}(t)$ in the direction $d$. It is bounded and holomorphic on the sector
$S_{d,\theta,R^{1/k}}$ and has the formal series $\hat{X}(\epsilon)$ as Gevrey asymptotic
expansion of order $1/k$ with respect to $\epsilon$ on $S_{d,\theta,R^{1/k}}$. This means that for all
$\frac{\pi}{k} < \theta_{1} < \theta$, there exist $C,M > 0$ such that
$$ ||\mathcal{L}^{d}_{k}(\mathcal{B}_{k}(\hat{X}))(\epsilon) - \sum_{p=0}^{n-1}
\frac{a_p}{p!} \epsilon^{p}||_{\mathbb{F}} \leq CM^{n}\Gamma(1+ \frac{n}{k})|\epsilon|^{n} $$
for all $n \geq 1$, all $\epsilon \in S_{d,\theta_{1},R^{1/k}}$.\medskip

 Multisummability of a formal power series is a recursive process that allows to compute the sum of a formal power series in different Gevrey orders. One of the approaches to multisummability is that stated by W. Balser, which can be found in~\cite{ba}, Theorem 1, p.57. Roughly speaking, given a formal power series $\hat{f}$ which can be decomposed into a sum $\hat{f}(z)=\hat{f}_1(z)+\ldots+\hat{f}_m(z)$ such that each of the terms $\hat{f}_j(z)$ is $k_j$-summable, with sum given by $f_j$, then, $\hat{f}$ turns out to be multisummable, and its multisum is given by $f_1(z)+\ldots+f_m(z)$. More precisely, one has the following definition.

\begin{defin} Let $(\mathbb{F},\left\|\cdot\right\|_{\mathbb{F}})$ be a complex Banach space and let $0<k_1<k_2$. Let $\mathcal{E}$ be a bounded open sector with vertex at 0, and opening $\frac{\pi}{k_2}+\delta_2$ for some $\delta_2>0$, and let $\mathcal{F}$ be a bounded open sector with vertex at the origin in $\C$, with opening $\frac{\pi}{k_1}+\delta_1$, for some $\delta_1>0$ and such that $\mathcal{E}\subseteq\mathcal{F}$ holds.\smallskip

A formal power series $\hat{f}(\epsilon)\in\mathbb{F}[[\epsilon]]$ is said to be $(k_2,k_1)-$summable on $\mathcal{E}$ if there exist $\hat{f}_2(\epsilon)\in\mathbb{F}[[\epsilon]]$ which is $k_2-$summable on $\mathcal{E}$, with $k_2$-sum given by $f_2:\mathcal{E}\to\mathbb{F}$, and $\hat{f}_1(\epsilon)\in\mathbb{F}[[\epsilon]]$ which is $k_1-$summable on $\mathbb{E}$, with $k_1$-sum given by $f_1:\mathcal{F}\to\mathbb{F}$, such that $\hat{f}=\hat{f}_1+\hat{f}_2$. Furthermore, the holomorphic function $f(\epsilon)=f_1(\epsilon)+f_2(\epsilon)$ on $\mathcal{E}$ is called the $(k_2,k_1)-$sum of $\hat{f}$ on $\mathcal{E}$. In that situation, $f(\epsilon)$ can be obtained from the analytic continuation of the $k_1-$Borel transform of $\hat{f}$ by the successive application of accelerator operators and Laplace transform of order $k_2$, see Section 6.1 in~\cite{ba}.
\end{defin}






A novel version of Ramis-Sibuya Theorem has been developed in~\cite{takei}, and has provided successful results in previous works by the authors,~\cite{lama1},~\cite{lama2}. A version of the result in two different levels which fits our needs is now given without proof, which can be found in~\cite{lama1},~\cite{lama2}.

\noindent {\bf Theorem (multilevel-RS)} {\it Let $(\mathbb{F},||.||_{\mathbb{F}})$ be a Banach space over $\mathbb{C}$ and
$\{ \mathcal{E}_{p_1,p_2} \}_{\begin{subarray}{l} 0 \leq p_1 \leq \varsigma_1 - 1\\0 \leq p_2 \leq \varsigma_2 - 1\end{subarray}}$ be a good covering in $\mathbb{C}^{\ast}$. Assume that $0<k_1<k_2$. For all
$0 \leq p_1 \leq \varsigma_1 - 1$ and $0\le p_2\le \varsigma_2-1$, let $G_{p_1,p_2}$ be a holomorphic function from $\mathcal{E}_{p_1,p_2}$ into
the Banach space $(\mathbb{F},||.||_{\mathbb{F}})$ and for every $(p_1,p_2),(p'_1,p'_2)\in\{0,\ldots,\varsigma_1-1\}\times\{0,\ldots,\varsigma_2-1\}$ such that $\mathcal{E}_{p_1,p_2}\cap\mathcal{E}_{p'_1,p'_2}\neq\emptyset$ we define $\Theta_{(p_1,p_2)(p'_1,p'_2)}(\epsilon) = G_{p_1,p_2}(\epsilon) - G_{p'_1,p'_2}(\epsilon)$
be a holomorphic function from the sector $Z_{(p_1,p_2),(p'_1,p'_2)} = \mathcal{E}_{p_1,p_2} \cap \mathcal{E}_{p'_1,p'_2}$ into $\mathbb{F}$.
We make the following assumptions.\medskip

\noindent {\bf 1)} The functions $G_{p_1,p_2}(\epsilon)$ are bounded as $\epsilon \in \mathcal{E}_{p_1,p_2}$ tends to the origin
in $\mathbb{C}$, for all $0 \leq p_1 \leq \varsigma_1 - 1$ and $0\le p_2\le \varsigma_2-1$.\medskip

\noindent {\bf 2)} $(\{0,\ldots,\varsigma_1-1\}\times\{0,\ldots,\varsigma_2\})^2=\mathcal{U}_0\cup\mathcal{U}_{k_1}\cup\mathcal{U}_{k_2}$, where 

$((p_1,p_2),(p'_1,p'_2))\in\mathcal{U}_0$ iff $\mathcal{E}_{p_1,p_2}\cap\mathcal{E}_{p'_1,p'_2}=\emptyset$,

$((p_1,p_2),(p'_1,p'_2))\in\mathcal{U}_{k_1}$ iff $\mathcal{E}_{p_1,p_2}\cap\mathcal{E}_{p'_1,p'_2}\neq\emptyset$ and 
$$ ||\Theta_{(p_1,p_2),(p'_1,p'_2)}(\epsilon)||_{\mathbb{F}} \leq C_{p_1,p_2,p'_1,p'_2}e^{-A_{p_1,p_2,p'_1,p'_2}/|\epsilon|^{k_1}} $$
for all $\epsilon \in Z_{(p_1,p_2),(p'_1,p'_2)}$.

$((p_1,p_2),(p'_1,p'_2))\in\mathcal{U}_{k_2}$ iff $\mathcal{E}_{p_1,p_2}\cap\mathcal{E}_{p'_1,p'_2}\neq\emptyset$ and 
$$ ||\Theta_{(p_1,p_2),(p'_1,p'_2)}(\epsilon)||_{\mathbb{F}} \leq C_{p_1,p_2,p'_1,p'_2}e^{-A_{p_1,p_2,p'_1,p'_2}/|\epsilon|^{k_2}} $$
for all $\epsilon \in Z_{(p_1,p_2),(p'_1,p'_2)}$.

Then, there exists a convergent power series $a(\epsilon)\in \mathbb{F}\{\epsilon\}$ and two formal power series $\hat{G}^1(\epsilon),\hat{G}^2(\epsilon)\in\mathbb{F}[[\epsilon]]$ such that $G_{p_1,p_2}(\epsilon)$ can be split in the form
$$G_{p_1,p_2}(\epsilon)=a(\epsilon)+G^{1}_{p_1,p_2}(\epsilon)+G^{2}_{p_1,p_2}(\epsilon),$$
where $G^{j}_{p_1,p_2}(\epsilon)\in\mathcal{O}(\mathcal{E}_{p_1,p_2},\mathbb{F})$, and admits $\hat{G}^j(\epsilon)$ as its asymptotic expansion of Gevrey order $1/k_j$ on $\mathcal{E}_{p_1,p_2}$, for $j\in\{1,2\}$.\smallskip

Moreover, assume that
$$\{((p_1^0,p_2^0),(p_1^1,p_2^1)), ((p_1^1,p_2^1),(p_1^2,p_2^2)), \ldots, ((p_1^{2y-1},p_2^{2y-1}),(p_1^{2y},p_2^{2y}))     \}$$
is a subset of $\mathcal{U}_{k_2}$, for some positive integer $y$, and
$$\mathcal{E}_{p_1^{y},p_2^y}\subseteq S_{\pi/k_1}\subseteq\bigcup_{0\le j\le 2y}\mathcal{E}_{p_1^{j},p_2^{j}},$$
for some sector $S_{\pi/k_1}$ with opening larger than $\pi/k_1$. Then, the formal power series $\hat{G}(\epsilon)$ is $(k_2,k_1)-$summable on $\mathcal{E}_{p_1^y,p_2^y}$ and its $(k_2,k_1)-$sum is $G_{p_1^y,p_2^y}(\epsilon)$ on $\mathcal{E}_{p_1^y,p_2^y}$.}

\subsection{Existence of formal power series solutions in the complex parameter and asymptotic behavior}

The second main result of our work states the existence of a formal power series in the perturbation parameter $\epsilon$, with coefficients in the Banach space $\mathbb{F}$ of holomorphic and bounded functions on $(\mathcal{T}_1\cap D(0,h''))\times (\mathcal{T}_2\cap D(0,h''))\times H_{\beta'}$, with the norm of the supremum. Here $h''$, $\mathcal{T}_1,\mathcal{T}_2$ are determined in Theorem~\ref{teo1}.

The importance of this result compared to the main one in~\cite{lama} lies on the fact that a multisummability phenomenon can be observed here, in contrast to~\cite{lama}. This situation is attained due to the appearance of different Gevrey levels coming from the different variables in time.

\begin{theo}\label{teo2}
Under the assumptions of Theorem~\ref{teo1}, a formal power series
$$\hat{u}(\bt,z,\epsilon)=\sum_{m\ge0}H_m(\bt,z)\epsilon^m/m!\in\mathbb{F}[[\epsilon]]$$
exists, with the following properties. $\hat{u}$ is a formal solution of (\ref{ICP_main0}). In addition to that, $\hat{u}$ can be split in the form
$$\hat{u}(\bt,z,\epsilon)=a(\bt,z,\epsilon)+\hat{u}_{1}(\bt,z,\epsilon)+\hat{u}_{2}(\bt,z,\epsilon),$$
where $a(\bt,z,\epsilon)\in\mathbb{F}\{\epsilon\}$, and $\hat{u}_{1},\hat{u}_{2}\in\mathbb{F}[[\epsilon]]$. Moreover, for every $p_1\in\{0,\ldots,\varsigma_1-1\}$ and $p_2\in\{0,\ldots,\varsigma_2-1\}$, the function $u_{p_1,p_2}(\bt,z,\epsilon)$ can be written as
$$u_{p_1,p_2}(\bt,z,\epsilon)=a(\bt,z,\epsilon)+u^1_{p_1,p_2}(\bt,z,\epsilon)+u^2_{p_1,p_2}(\bt,z,\epsilon),$$
where $\epsilon\mapsto u^j_{p_1,p_2}(\bt,z,\epsilon)$ is an $\mathbb{F}-$valued function which admits $\hat{u}_{j}(\bt,z,\epsilon)$ as its $k_j-$Gevrey asymptotic expansion on $\mathcal{E}_{p_1,p_2}$, for $j=1,2$.

Moreover, assume that
$$\{((p_1^0,p_2^0),(p_1^1,p_2^1)), ((p_1^1,p_2^1),(p_1^2,p_2^2)), \ldots, ((p_1^{2y-1},p_2^{2y-1}),(p_1^{2y},p_2^{2y}))     \}$$
is a subset of $\mathcal{U}_{k_2}$, for some positive integer $y$, and
$$\mathcal{E}_{p_1^{y},p_2^y}\subseteq S_{\pi/k_1}\subseteq\bigcup_{0\le j\le 2y}\mathcal{E}_{p_1^{j},p_2^{j}},$$
for some sector $S_{\pi/k_1}$ with opening larger than $\pi/k_1$. Then, $\hat{u}(\bt,z,\epsilon)$ is $(k_2,k_1)-$summable on $\mathcal{E}_{p_1^y,p_2^y}$ and its $(k_2,k_1)-$sum is $u_{p_1^y,p_2^y}(\epsilon)$ on $\mathcal{E}_{p_1^y,p_2^y}$.
\end{theo}
\begin{proof}
Let $(u_{p_1,p_2}(\bt,z,\epsilon))_{\begin{subarray}{l}0\le p_1\le \varsigma_1-1\\0\le p_2\le \varsigma_2-1\end{subarray}}$ be the family constructed in Theorem~\ref{teo1}. We recall that $(\mathcal{E}_{p_1,p_2})_{\begin{subarray}{l}0\le p_1\le \varsigma_1-1\\0\le p_2\le \varsigma_2-1\end{subarray}}$ is a good covering in $\C^{\star}$. 

The function $G_{p_1,p_2}(\epsilon):=(t_1,t_2,z)\mapsto u_{p_1,p_2}(t_1,t_2,z,\epsilon)$ belongs to $\mathcal{O}(\mathcal{E}_{p_1,p_2},\mathbb{F})$. We consider $\{(p_1,p_2),(p'_1,p'_2)\}$ such that $(p_1,p_2)$ and $(p'_1,p'_2)$ belong to $\{0,\ldots,\varsigma_1-1\}\times \{0,\ldots,\varsigma_2-1\}$, and $\mathcal{E}_{p_1,p_2}$ and $\mathcal{E}_{p'_1,p'_2}$ are consecutive sectors in the good covering, so their intersection is not empty. In view of (\ref{exp_small_difference_u_p11}) and (\ref{exp_small_difference_u_p12}), one has that $\Delta_{(p_1,p_2),(p'_1,p'_2)}(\epsilon):=G_{p_1,p_2}(\epsilon)- G_{p'_1,p'_2}(\epsilon)$ satisfies exponentially flat bounds of certain Gevrey order, which is $k_1$ in the case that $\{(p_1,p_2),(p'_1,p'_2)\}\in\mathcal{U}_{k_1}$ and $k_2$ if $\{(p_1,p_2),(p'_1,p'_2)\}\in\mathcal{U}_{k_2}$. Multilevel-RS Theorem guarantees the existence of formal power series $\hat{G}(\epsilon),\hat{G}_1(\epsilon),\hat{G}_2(\epsilon)\in\mathbb{F}[[\epsilon]]$ such that 
$$\hat{G}(\epsilon)=a(\epsilon)+\hat{G}_1(\epsilon)+\hat{G}_{2}(\epsilon),$$ and the splitting
$$G_{p_1,p_2}(\epsilon)=a(\epsilon)+G^1_{p_1,p_2}(\epsilon)+G^2_{p_1,p_2}(\epsilon),$$
for some $a\in\mathbb{F}\{\epsilon\}$, such that for every $(p_1,p_2)\in\{0,\ldots,\varsigma_1-1\}\times \{0,\ldots,\varsigma_2-1\}$, one has that $G^1_{p_1,p_2}(\epsilon)$ admits $\hat{G}_{p_1,p_2}^1(\epsilon)$ as its Gevrey asymptotic expansion of order $k_1$, and $G^2_{p_1,p_2}(\epsilon)$ admits $\hat{G}_{p_1,p_2}^2(\epsilon)$ as its Gevrey asymptotic expansion of order $k_2$. We define
$$\hat{G}(\epsilon)=:\hat{u}(\bt,z,\epsilon)=\sum_{m\ge0}H_m(\bt,z)\frac{\epsilon^{m}}{m!}.$$
It only rests to prove that $\hat{u}(\bt,z,\epsilon)$ is a formal solution of (\ref{ICP_main0}). For every $0\le p_1\le \varsigma_1-1$, $0\le p_2\le \varsigma_2-1$ and $j=1,2$, the existence of an asymptotic expansion concerning $G^{j}_{p_1,p_2}(\epsilon)$ and $\hat{G}^{j}(\epsilon)$ implies that
\begin{equation}\label{e1363}
\lim_{\epsilon\to 0,\epsilon\in\mathcal{E}_{p_1,p_2}}\sup_{(\bt,z)\in(\tau_1\cap D(0,h''))\times (\tau_2\cap D(0,h''))\times H_{\beta'}}|\partial_{\epsilon}^{\ell}u_{p_1,p_2}(\bt,z,\epsilon)-H_{\ell}(\bt)|=0,
\end{equation} 
for every $\ell\in\mathbb{N}$. By construction, the function $u_{p_1,p_2}(\bt,z,\epsilon)$ is a solution of (\ref{ICP_main0}). Taking derivatives of order $m\ge0$ with respect to $\epsilon$ on that equation yield
\begin{multline}\label{e1367}
Q_1(\partial_{z})Q_2(\partial_z)\partial_{t_1}\partial_{t_2}\partial_\epsilon^mu_{p_1,p_2}(\bt,z,\epsilon)\\
 =\sum_{m_1+m_2=m}\frac{m!}{m_1!m_2!}\left(\sum_{m_{11}+m_{12}=m_1}\frac{m_1!}{m_{11}!m_{12}!}\partial_\epsilon^{m_{11}}P_1(\partial_z,\epsilon)\partial_{\epsilon}^{m_{12}}u_{p_1,p_2}(\bt,z,\epsilon)\right)\\
\times \left(\sum_{m_{21}+m_{22}=m_2}\frac{m_2!}{m_{21}!m_{22}!}\partial_\epsilon^{m_{21}}P_2(\partial_z,\epsilon)\partial_{\epsilon}^{m_{22}}u_{p_1,p_2}(\bt,z,\epsilon)\right)\\
+\sum_{0\le l_1\le D_1,0\le l_2\le D_2}\left(\sum_{m_{1}+m_2=m}\frac{m!}{m_1!m_2!}\partial_\epsilon^{m_1}(\epsilon^{\Delta_{l_1,l_2}})t_1^{d_{l_1}}\partial_{t_1}^{\delta_{l_1}}t_2^{\tilde{d}_{l_2}}\partial_{t_2}^{\tilde{\delta}_{l_2}}R_{l_1,l_2}(\partial_z)\partial_\epsilon^{m_2}u_{p_1,p_2}(\bt,z,\epsilon)\right)\\
+ \sum_{m_{1}+m_{2}=m} \frac{m!}{m_{1}!m_{2}!} \partial_{\epsilon}^{m_1}c_{0}(\bt,z,\epsilon)
R_{0}(\partial_{z})\partial_{\epsilon}^{m_2}u_{p_1,p_2}(\bt,z,\epsilon) + \partial_{\epsilon}^{m}f(\bt,z,\epsilon),
\end{multline}
for every $m\ge 0$ and $(\bt,z,\epsilon)\in (\mathcal{T}_1\cap D(0,h''))\times (\mathcal{T}_2\cap D(0,h''))\times H_{\beta'}\times\mathcal{E}_{p_1,p_2}$. Tending $\epsilon\to 0$ in (\ref{e1367}) together with (\ref{e1363}), we obtain a recursion formula for the coefficients of the formal solution.
\begin{multline}\label{e1368}
Q_1(\partial_{z})Q_2(\partial_z)\partial_{t_1}\partial_{t_2}H_{m}(\bt,z)\\
 =\sum_{m_1+m_2=m}\frac{m!}{m_1!m_2!}\left(\sum_{m_{11}+m_{12}=m_1}\frac{m_1!}{m_{11}!m_{12}!}\partial_\epsilon^{m_{11}}P_1(\partial_z,0)H_{m_{12}}(\bt,z)\right)\\
\times \left(\sum_{m_{21}+m_{22}=m_2}\frac{m_2!}{m_{21}!m_{22}!}\partial_\epsilon^{m_{21}}P_2(\partial_z,0)H_{m_{12}}(\bt,z)\right)\\
+\sum_{0\le l_1\le D_1,0\le l_2\le D_2}\frac{m!}{(m-\Delta_{l_1,l_2})!}t_1^{d_{l_1}}\partial_{t_1}^{\delta_{l_1}}t_2^{\tilde{d}_{l_2}}\partial_{t_2}^{\tilde{\delta}_{l_2}}R_{l_1,l_2}(\partial_z)H_{m-\Delta_{l_1,l_2}}(\bt,z)\\
+ \sum_{m_{1}+m_{2}=m} \frac{m!}{m_{1}!m_{2}!} \partial_{\epsilon}^{m_1}c_{0}(\bt,z,0)
R_{0}(\partial_{z})H_{m_2}(\bt,z) + \partial_{\epsilon}^{m}f(\bt,z,0),
\end{multline}
for every $m\ge \max_{1\le l_1\le D_1,1\le l_2\le D_2}\Delta_{l_1,l_2}$, and $(\bt,z,\epsilon)\in (\mathcal{T}_1\cap D(0,h''))\times (\mathcal{T}_2\cap D(0,h''))\times H_{\beta'}$. From the analyticity of $c_0$ and $f$ with respect to $\epsilon$ in a vicinity of the origin we get
\begin{equation}
c_{0}(\bt,z,\epsilon) = \sum_{m \geq 0} \frac{(\partial_{\epsilon}^{m}c_{0})(\bt,z,0)}{m!}\epsilon^{m} \ \ , \ \
f(\bt,z,\epsilon) = \sum_{m \geq 0} \frac{(\partial_{\epsilon}^{m}f)(\bt,z,0)}{m!}\epsilon^{m}, \label{Taylor_c0_f}
\end{equation}
for every $\epsilon\in D(0,\epsilon_0)$ and $(\bt,z)$ as above. On the other hand, a direct inspection from the recursion formula (\ref{e1368}) and (\ref{Taylor_c0_f}) allow us to affirm that the formal power series $\hat{u}(\bt,z,\epsilon) = \sum_{m \geq 0} H_{m}(\bt,z)\epsilon^{m}/m!$ solves the equation (\ref{ICP_main0}).
\end{proof}

\end{document}